\documentclass[a4paper,oneside]{amsart}
\usepackage{etoolbox}
\ifdef{\ebook}{
  \usepackage[paperwidth=9cm, paperheight=12cm, hmargin={0.17in, 0.17in}, vmargin={0.50in, 0.17in}]{geometry} \usepackage{kmath,kerkis}
}{}

\usepackage{amsbsy} 
\usepackage{amssymb} 
\usepackage{longtable}
\usepackage{xypic}
\usepackage{booktabs}
\usepackage{color}
\usepackage{booktabs}
\usepackage{paralist}
\usepackage{hyperref}
\usepackage{accents}
\usepackage{multirow}
\usepackage{bigdelim}
\usepackage[svgnames]{xcolor}
\usepackage{pstricks}
\usepackage{pst-plot}
\usepackage{nicefrac}
\usepackage{pstricks-add}

\newgray{llgray}{0.88}
\newgray{grayi}{0.68}
\newgray{grayii}{0.60}
\newgray{grayiii}{0.55}
\newgray{grayiv}{0.50}
\newgray{grayv}{0.45}

\def \marks(#1)(#2)(#3){}
\def \definerays (#1,#2)(#3,#4)(#5,#6)(#7,#8){
\renewcommand{\ra}[2]{! #1 100 mul ##1 add #2 100 mul ##2 add}
\renewcommand{\rb}[2]{! #3 100 mul ##1 add #4 100 mul ##2 add}
\renewcommand{\rc}[2]{! #5 100 mul ##1 add #6 100 mul ##2 add}
\renewcommand{\rd}[2]{! #7 100 mul ##1 add #8 100 mul ##2 add}
}

\def \definerays (#1,#2)(#3,#4)(#5,#6)(#7,#8){
\renewcommand{\ra}[2]{! #1 100 mul ##1 add #2 100 mul ##2 add}
\renewcommand{\rb}[2]{! #3 100 mul ##1 add #4 100 mul ##2 add}
\renewcommand{\rc}[2]{! #5 100 mul ##1 add #6 100 mul ##2 add}
\renewcommand{\rd}[2]{! #7 100 mul ##1 add #8 100 mul ##2 add}
}

\def \defineextrarays (#1,#2)(#3,#4)(#5,#6)(#7,#8){
\renewcommand{\re}[2]{! #1 20 mul ##1 add #2 20 mul  ##2 add}
\renewcommand{\rf}[2]{! #3 20 mul ##1 add #4 20 mul  ##2 add}
\renewcommand{\rg}[2]{! #5 20 mul ##1 add #6 20 mul  ##2 add}
\renewcommand{\rh}[2]{! #7 20 mul ##1 add #8 20 mul  ##2 add}
}

\def \oput (#1,#2){\uput{3.2}[(#1,#2)](0,0){\tiny (#1,#2)}}
\def \oputm(#1,#2){\uput{3}[(#1,#2)](0,0){\tiny (#1,#2)}}
\def \oputmm(#1,#2){\uput{2.6}[(#1,#2)](0,0){\tiny (#1,#2)}}
\def \cput [#1](#2,#3){\uput[#1](#2,#3){(#2,#3)} \psdot[dotstyle=o](#2,#3)}
\newcommand{\rg}{}
\newcommand{\rh}{}
\newcommand{\ra}[2]{! 10 #1 add 10 #2 add}
\def \raya(#1,#2){\psline(#1,#2)(\ra{#1}{#2})}
\newcommand{\rb}[2]{! 20 #1 add -10 #2 add}
\def \rayb(#1,#2){\psline(#1,#2)(\rb{#1}{#2})}
\newcommand{\rc}[2]{! -10 #1 add -10 #2 add}
\def \rayc(#1,#2){\psline(#1,#2)(\rc{#1}{#2})}
\newcommand{\rd}[2]{! -10 #1 add 20 #2 add}
\def \rayd(#1,#2){\psline(#1,#2)(\rd{#1}{#2})}
\newcommand{\re}[2]{! 0 #1 add -10 #2 add}
\def \raye(#1,#2){\psline(#1,#2)(\re{#1}{#2})}
\newcommand{\rf}[2]{! -10 #1 add 0 #2 add}
\def \rayf(#1,#2){\psline(#1,#2)(\rf{#1}{#2})}
\def \rayg(#1,#2){\psline(#1,#2)(\rg{#1}{#2})}
\def \rayh(#1,#2){\psline(#1,#2)(\rh{#1}{#2})}

\def \beginpsclip{\psclip{\psccurve[linestyle=solid,linecolor=white](2,2)(2.8,0)(2.5,-2)(-3.2,0)(-2.2,2.5)(0,2.8)(2,2)}}

\def \stara (#1,#2){
\psset{fillstyle=solid, fillcolor=grayii} \cab(#1,#2) \psset{fillcolor=grayiii} \cbc(#1,#2) \psset{fillcolor=grayiv} \cca(#1,#2)
\pscircle[fillstyle=solid,fillcolor=black](0,0){0.1}
}

\def \starbb (#1,#2){
\psset{fillstyle=solid, fillcolor=grayii} \caf(#1,#2)  \psset{fillcolor=grayiv} \cbf(#1,#2) \psset{fillcolor=grayiii} \cbg(#1,#2)  \psset{fillcolor=lightgray} \ccg(#1,#2) \psset{fillcolor=grayiv} \cch(#1,#2) \psset{fillcolor=grayi} \cdh(#1,#2) \psset{fillcolor=grayi} \cae(#1,#2)\psset{fillcolor=grayv} \cde(#1,#2)
\pscircle[fillstyle=solid,fillcolor=black](0,0){0.1}
}

\def \starb (#1,#2){
\psset{fillstyle=solid, fillcolor=grayii} \cab(#1,#2) \psset{fillcolor=grayiii} \cbc(#1,#2) \psset{fillcolor=grayiv} \ccd(#1,#2) \psset{fillcolor=grayi} \cda(#1,#2)
\pscircle[fillstyle=solid,fillcolor=black](0,0){0.1}
}

\def \subdivh (#1,#2)(#3,#4){
\psset{fillstyle=solid, fillcolor=grayi} \lda (#1,#2)(#3,#4) \psset{fillcolor=grayii} \cab(#3,#4) \psset{fillcolor=grayiii}  \lbc(#3,#4)(#1,#2) \psset{fillcolor=grayiv} \lcd(#1,#2)(#1,#2)
\pscircle[fillstyle=solid,fillcolor=black](0,0){0.1}
}

\def \subdivv (#1,#2)(#3,#4){
\psset{fillstyle=solid, fillcolor=grayi} \cda(#1,#2) \psset{fillcolor=grayii} \lab(#1,#2)(#3,#4)  \psset{fillcolor=grayiii} \cbc(#3,#4)  \psset{fillcolor=grayiv} \lcd(#3,#4)(#1,#2)
\pscircle[fillstyle=solid,fillcolor=black](0,0){0.1}
}

\def \subdivhh (#1,#2)(#3,#4)(#5,#6){
\psset{fillstyle=solid, fillcolor=grayi} \lae(#5,#6)(#3,#4) \psset{fillcolor=grayv} \lde(#1,#2)(#5,#6) 
\psset{fillcolor=grayii} \caf(#3,#4) \psset{fillcolor=grayiv} \cbf(#3,#4) \psset{fillcolor=grayiii}  \lbg(#3,#4)(#5,#6) \psset{fillcolor=lightgray} \lcg(#1,#2)(#5,#6) \psset{fillcolor=grayiv} \cch(#1,#2) \psset{fillcolor=grayi} \cdh(#1,#2)
\pscircle[fillstyle=solid,fillcolor=black](0,0){0.1}
}

\def \subdivhhh (#1,#2)(#3,#4)(#5,#6)(#7,#8){
\psset{fillstyle=solid, fillcolor=grayi} \lae(#7,#8)(#3,#4) \psset{fillcolor=grayv} \lde(#1,#2)(#7,#8) 
\psset{fillcolor=grayii} \caf(#3,#4) \psset{fillcolor=grayiv} \cbf(#3,#4) \psset{fillcolor=grayiii}  \lbg(#3,#4)(#5,#6) \psset{fillcolor=lightgray} \lcg(#1,#2)(#5,#6) \psset{fillcolor=grayiv} \cch(#1,#2) \psset{fillcolor=grayi} \cdh(#1,#2)
\pscircle[fillstyle=solid,fillcolor=black](0,0){0.1}
}

\def \subdivvv (#1,#2)(#3,#4)(#5,#6){
\psset{fillstyle=solid, fillcolor=grayi} \cae(#1,#2)  \psset{fillcolor=grayv} \cde(#1,#2) \psset{fillcolor=grayii} \laf(#1,#2)(#5,#6) \psset{fillcolor=grayiv} \lbf(#3,#4)(#5,#6) \psset{fillcolor=grayiii} \cbg(#3,#4) \psset{fillcolor=lightgray} \ccg(#3,#4)  \psset{fillcolor=grayiv} \lch(#3,#4)(#5,#6) \psset{fillcolor=grayi} \ldh(#1,#2)(#5,#6)
\pscircle[fillstyle=solid,fillcolor=black](0,0){0.1}
}

\def \subdivvvv (#1,#2)(#3,#4)(#5,#6)(#7,#8){
\psset{fillstyle=solid, fillcolor=grayi} \cae(#1,#2)  \psset{fillcolor=grayv} \cde(#1,#2) \psset{fillcolor=grayii} \laf(#1,#2)(#7,#8) \psset{fillcolor=grayiv} \lbf(#3,#4)(#7,#8) \psset{fillcolor=grayiii} \cbg(#3,#4) \psset{fillcolor=lightgray} \ccg(#3,#4)  \psset{fillcolor=grayiv} \lch(#3,#4)(#5,#6) \psset{fillcolor=grayi} \ldh(#1,#2)(#5,#6)
\pscircle[fillstyle=solid,fillcolor=black](0,0){0.1}
}

\def \subdivdd (#1,#2)(#3,#4)(#5,#6){
\psset{fillstyle=solid, fillcolor=grayi} \cae(#1,#2) 
\psset{fillcolor=grayv} \lde(#5,#6)(#1,#2)
\psset{fillcolor=grayii} 
\caf(#1,#2) \psset{fillcolor=grayiv} \lbf(#5,#6)(#1,#2) \psset{fillcolor=grayiii} \lbg(#5,#6)(#3,#4) \psset{fillcolor=lightgray} \ccg(#3,#4)  \psset{fillcolor=grayiv} \cch(#3,#4) \psset{fillcolor=grayi} \ldh(#5,#6)(#3,#4)
\pscircle[fillstyle=solid,fillcolor=black](0,0){0.1}
}

\def \cab (#1,#2){\psline(\ra{#1}{#2})(#1,#2)(\rb{#1}{#2})}
\def \cae (#1,#2){\psline(\ra{#1}{#2})(#1,#2)(\re{#1}{#2})}
\def \caf (#1,#2){\psline(\ra{#1}{#2})(#1,#2)(\rf{#1}{#2})}
\def \cah (#1,#2){\psline(\ra{#1}{#2})(#1,#2)(\rh{#1}{#2})}
\def \cbc (#1,#2){\psline(\rb{#1}{#2})(#1,#2)(\rc{#1}{#2})}
\def \cbg (#1,#2){\psline(\rb{#1}{#2})(#1,#2)(\rg{#1}{#2})}
\def \cbf (#1,#2){\psline(\rb{#1}{#2})(#1,#2)(\rf{#1}{#2})}
\def \ccd (#1,#2){\psline(\rc{#1}{#2})(#1,#2)(\rd{#1}{#2})}
\def \cda (#1,#2){\psline(\rd{#1}{#2})(#1,#2)(\ra{#1}{#2})}
\def \cde (#1,#2){\psline(\rd{#1}{#2})(#1,#2)(\re{#1}{#2})}
\def \cdh (#1,#2){\psline(\rd{#1}{#2})(#1,#2)(\rh{#1}{#2})}
\def \cef (#1,#2){\psline(\re{#1}{#2})(#1,#2)(\rf{#1}{#2})}
\def \cde (#1,#2){\psline(\rd{#1}{#2})(#1,#2)(\re{#1}{#2})}
\def \cca (#1,#2){\psline(\rc{#1}{#2})(#1,#2)(\ra{#1}{#2})}
\def \cce (#1,#2){\psline(\rc{#1}{#2})(#1,#2)(\re{#1}{#2})}
\def \ccg (#1,#2){\psline(\rc{#1}{#2})(#1,#2)(\rg{#1}{#2})}
\def \cch (#1,#2){\psline(\rc{#1}{#2})(#1,#2)(\rh{#1}{#2})}
\def \cgh (#1,#2){\psline(\rg{#1}{#2})(#1,#2)(\rh{#1}{#2})}

\def \laa (#1,#2)(#3,#4){\psline(\ra{#1}{#2})(#1,#2)(#3,#4)(\ra{#3}{#4})}
\def \lbb (#1,#2)(#3,#4){\psline(\rb{#1}{#2})(#1,#2)(#3,#4)(\rb{#3}{#4})}
\def \lee (#1,#2)(#3,#4){\psline(\re{#1}{#2})(#1,#2)(#3,#4)(\re{#3}{#4})\psdot(0,0)}
\def \ldd (#1,#2)(#3,#4){\psline(\rd{#1}{#2})(#1,#2)(#3,#4)(\rd{#3}{#4})}
\def \lgg (#1,#2)(#3,#4){\psline(\rg{#1}{#2})(#1,#2)(#3,#4)(\rg{#3}{#4})}
\def \lhh (#1,#2)(#3,#4){\psline(\rh{#1}{#2})(#1,#2)(#3,#4)(\rh{#3}{#4})}
\def \lab (#1,#2)(#3,#4){\psline(\ra{#1}{#2})(#1,#2)(#3,#4)(\rb{#3}{#4})}
\def \laf (#1,#2)(#3,#4){\psline(\ra{#1}{#2})(#1,#2)(#3,#4)(\rf{#3}{#4})}
\def \lcd (#1,#2)(#3,#4){\psline(\rc{#1}{#2})(#1,#2)(#3,#4)(\rd{#3}{#4})}
\def \lcg (#1,#2)(#3,#4){\psline(\rc{#1}{#2})(#1,#2)(#3,#4)(\rg{#3}{#4})}
\def \lef (#1,#2)(#3,#4){\psline(\re{#1}{#2})(#1,#2)(#3,#4)(\rf{#3}{#4})}
\def \lda (#1,#2)(#3,#4){\psline(\rd{#1}{#2})(#1,#2)(#3,#4)(\ra{#3}{#4})}

\def \lde (#1,#2)(#3,#4){\psline(\re{#1}{#2})(#1,#2)(#3,#4)(\rd{#3}{#4})}
\def \lae (#1,#2)(#3,#4){\psline(\re{#1}{#2})(#1,#2)(#3,#4)(\ra{#3}{#4})}
\def \ldf (#1,#2)(#3,#4){\psline(\rf{#1}{#2})(#1,#2)(#3,#4)(\rd{#3}{#4})}
\def \lfa (#1,#2)(#3,#4){\psline(\rf{#1}{#2})(#1,#2)(#3,#4)(\ra{#3}{#4})}
\def \lde (#1,#2)(#3,#4){\psline(\rd{#1}{#2})(#1,#2)(#3,#4)(\re{#3}{#4})}
\def \lbc (#1,#2)(#3,#4){\psline(\rb{#1}{#2})(#1,#2)(#3,#4)(\rc{#3}{#4})}
\def \lbg (#1,#2)(#3,#4){\psline(\rb{#1}{#2})(#1,#2)(#3,#4)(\rg{#3}{#4})}
\def \lbf (#1,#2)(#3,#4){\psline(\rb{#1}{#2})(#1,#2)(#3,#4)(\rf{#3}{#4})}
\def \lca (#1,#2)(#3,#4){\psline(\rc{#1}{#2})(#1,#2)(#3,#4)(\ra{#3}{#4})}
\def \lch (#1,#2)(#3,#4){\psline(\rc{#1}{#2})(#1,#2)(#3,#4)(\rh{#3}{#4})}
\def \ldh (#1,#2)(#3,#4){\psline(\rd{#1}{#2})(#1,#2)(#3,#4)(\rh{#3}{#4})}

\SpecialCoor
\psset{gridlabels=0,gridwidth=1.3pt,subgridwidth=0.6pt,subgriddiv=1,griddots=1,subgriddots=1,subgridcolor=white}

\def \titlepic {
\definerays(1,1)(1,-1)(-1,-1)(-1,1)
\defineextrarays(1,1)(1,-1)(-1,-1)(-1,1)
\begin{pspicture}(-3,-2.8)(3,3.8)%
\beginpsclip 
 \psset{linewidth=0.5pt,fillstyle=solid}%
 \subdivh(-1,0)(0,0)
\psgrid(0,0)(-3,-3)(3,3)
\endpsclip
\end{pspicture}
\begin{pspicture}(-3.2,-2.8)(3,3.8)%
\beginpsclip 
 \psset{linewidth=0.5pt,fillstyle=solid}%
 \subdivv(0,0)(0,-1)
\psgrid(0,0)(-3,-3)(3,3)\endpsclip 
\end{pspicture}
\begin{pspicture}(-3.2,-2.8)(3,3.8)%
\beginpsclip 
 \psset{linewidth=0.5pt,fillstyle=solid}%
 \starb(0.5,0.5)
\psgrid(0,0)(-3,-3)(3,3)\endpsclip 
\end{pspicture}
}

\def \threefoldQ {
\definerays(1,1)(1,-1)(-1,-1)(-1,1)
\defineextrarays(1,1)(1,-1)(-1,-1)(-1,1)
\begin{pspicture}(-3,-2.8)(3,3.8)%
\beginpsclip 
 \psset{linewidth=0.5pt,fillstyle=solid}%
 \subdivh(-1,0)(0,0)
\psgrid(0,0)(-3,-3)(3,3)
\endpsclip
\end{pspicture}
\begin{pspicture}(-3.2,-2.8)(3,3.8)%
\beginpsclip 
 \psset{linewidth=0.5pt,fillstyle=solid}%
 \subdivv(0,0)(0,-1)
\psgrid(0,0)(-3,-3)(3,3)\endpsclip 
\end{pspicture}
\begin{pspicture}(-3.2,-2.8)(3,3.8)%
\beginpsclip 
 \psset{linewidth=0.5pt,fillstyle=solid}%
 \starb(0.5,0.5)
\psgrid(0,0)(-3,-3)(3,3)\endpsclip 
\end{pspicture}
\begin{pspicture}(-4.2,-2.8)(3.9,3.8)%
\beginpsclip 
 \psset{linewidth=0.6pt,fillstyle=solid}%
 \starb(0,0)
 \psset{linestyle=dashed,fillstyle=none}
 \pspolygon[linestyle=solid,fillstyle=solid,fillcolor=white](0.5,0.5)(-0.5,0.5)(-0.5,-0.5)(0.5,-0.5)
 \psgrid(0,0)(-3,-3)(3,3)
 \psset{linewidth=1pt,linecolor=white}
 \endpsclip 
\end{pspicture}}

\def \threefoldQLabeled {
\definerays(1,1)(1,-1)(-1,-1)(-1,1)
\defineextrarays(1,1)(1,-1)(-1,-1)(-1,1)
\begin{pspicture}(-3,-2.8)(3.2,3.8)%
\beginpsclip 
 \psset{linewidth=0.5pt,fillstyle=solid}%
 \subdivh(-1,0)(0,0)
\psgrid(0,0)(-3,-3)(3,3)
\endpsclip
\put(2,2.7){$\fan_0$}
\end{pspicture}
\begin{pspicture}(-3.2,-2.8)(3.2,3.8)%
\beginpsclip 
 \psset{linewidth=0.5pt,fillstyle=solid}%
 \subdivv(0,0)(0,-1)
\psgrid(0,0)(-3,-3)(3,3)\endpsclip 
\put(2,2.7){$\fan_\infty$}
\end{pspicture}
\begin{pspicture}(-3.2,-2.8)(3.2,3.8)%
\beginpsclip 
 \psset{linewidth=0.5pt,fillstyle=solid}%
 \starb(0.5,0.5)
\psgrid(0,0)(-3,-3)(3,3)\endpsclip 
\put(2,2.7){$\fan_1$}
\end{pspicture}
\begin{pspicture}(-4.2,-2.8)(3.9,3.8)%
\beginpsclip 
 \psset{linewidth=0.6pt,fillstyle=solid}%
 \starb(0,0)
 \psset{linestyle=dashed,fillstyle=none}
 \pspolygon[linestyle=solid,fillstyle=solid,fillcolor=white](0.5,0.5)(-0.5,0.5)(-0.5,-0.5)(0.5,-0.5)
 \psgrid(0,0)(-3,-3)(3,3)
 \psset{linewidth=1pt,linecolor=white}
 \endpsclip 
\put(2,2.7){$\sdeg$}
\end{pspicture}}

\def \threefoldQDegSum{
\definerays(1,1)(1,-1)(-1,-1)(-1,1)
\defineextrarays(1,1)(1,-1)(-1,-1)(-1,1)
\begin{pspicture}(-3.2,-2.3)(4,3.8)%
\beginpsclip
  \psset{linewidth=0.6pt,fillstyle=solid}%
 \starb(-0.5,-0.5)
 \psset{linestyle=dashed,fillstyle=none}
 \pspolygon[linestyle=solid,fillstyle=solid,fillcolor=gray](0,0)(-1,0)(-1,-1)(0,-1)
\psgrid(0,0)(-3,-3)(3,3)\endpsclip 
\uput[150](4.3,0){\Large =}
\end{pspicture}
\begin{pspicture}(-3,-2.3)(4,3.8)%
\beginpsclip 
 \psset{linewidth=0.5pt,fillstyle=solid}%
 \subdivh(-1,0)(0,0)
\psgrid(0,0)(-3,-3)(3,3)
\endpsclip
\uput[150](4.3,0){\Large +}
\end{pspicture}
\begin{pspicture}(-3.2,-2.3)(4,3.8)%
\beginpsclip 
 \psset{linewidth=0.5pt,fillstyle=solid}%
 \subdivv(0,0)(0,-1)
\psgrid(0,0)(-3,-3)(3,3)\endpsclip 
\end{pspicture}
}

\def \threefoldQDeg {
\definerays(1,1)(1,-1)(-1,-1)(-1,1)
\defineextrarays(1,1)(1,-1)(-1,-1)(-1,1)
\begin{pspicture}(-3.2,-2.8)(3,3.8)%
\beginpsclip
  \psset{linewidth=0.6pt,fillstyle=solid}%
 \starb(-0.5,-0.5)
 \psset{linestyle=dashed,fillstyle=none}
 \pspolygon[linestyle=solid,fillstyle=solid,fillcolor=gray](0,0)(-1,0)(-1,-1)(0,-1)
\psgrid(0,0)(-3,-3)(3,3)\endpsclip 
\put(2,2.7){$\fan_0$}
\end{pspicture}
\begin{pspicture}(-3.2,-2.8)(3,3.8)%
\beginpsclip 
 \psset{linewidth=0.5pt,fillstyle=solid}%
 \starb(0.5,0.5)
\psgrid(0,0)(-3,-3)(3,3)\endpsclip 
\put(2,2.7){$\fan_\infty$}
\end{pspicture}
\begin{pspicture}(-4.2,-2.8)(3.9,3.8)%
\beginpsclip 
 \psset{linewidth=0.6pt,fillstyle=solid}%
 \starb(0,0)
 \psset{linestyle=dashed,fillstyle=none}
 \pspolygon[linestyle=solid,fillstyle=solid,fillcolor=white](0.5,0.5)(-0.5,0.5)(-0.5,-0.5)(0.5,-0.5)
 \psgrid(0,0)(-3,-3)(3,3)
 \psset{linewidth=1pt,linecolor=white}
 \endpsclip 
\put(2,2.7){$\sdeg$}
\end{pspicture}}

\def\threefoldAA {
\definerays(1,1)(0,-1)(-1,-1)(-1,0)
\defineextrarays(0,1)(1,0)(-1,-1)(-1,-1)
\begin{pspicture}(-3,-2.8)(3,3.8)%
\beginpsclip 
 \psset{linewidth=0.5pt,fillstyle=solid}%
 \subdivhh(-0.5,0)(0,0)(-0.5,0)
\psgrid(0,0)(-3,-3)(3,3)\endpsclip 
\end{pspicture}
\begin{pspicture}(-3.2,-2.8)(3,3.8)%
\beginpsclip 
 \psset{linewidth=0.5pt,fillstyle=solid}%
 \subdivvv(0,0)(0,-0.5)(0,-0.5)
\psgrid(0,0)(-3,-3)(3,3)\endpsclip 
\end{pspicture}
\begin{pspicture}(-3.2,-2.8)(3,3.8)%
\beginpsclip 
 \psset{linewidth=0.5pt,fillstyle=solid}%
 \subdivdd(0.5,0.5)(0,0)(0,0)
\psgrid(0,0)(-3,-3)(3,3)\endpsclip 
\end{pspicture}
\begin{pspicture}(-4.2,-2.8)(3.9,3.8)%
\beginpsclip 
 \psset{linewidth=0.6pt,fillstyle=solid}%
 \starbb(0,0)
 \psset{linestyle=dashed,fillstyle=none}
 \pspolygon[linestyle=solid,fillstyle=solid,fillcolor=white](0.5,0.5)(0,0.5)(-0.5,0)(-0.5,-0.5)(0,-0.5)(0.5,0)
 \psgrid(0,0)(-3,-3)(3,3)
 \psset{linewidth=1pt,linecolor=white}
\endpsclip 
\end{pspicture}}

\def\threefoldAB {
\definerays(1,1)(1,-1)(-1,-1)(-1,1)
\defineextrarays(0,1)(1,1)(0,-1)(-1,1)
\begin{pspicture}(-3,-2.8)(3,3.8)%
\beginpsclip 
 \psset{linewidth=0.5pt,fillstyle=solid}%
 \subdivhh(-1,0)(0,0)(-0.5,0)
\psgrid(0,0)(-3,-3)(3,3)\endpsclip 
\end{pspicture}
\begin{pspicture}(-3.2,-2.8)(3,3.8)%
\beginpsclip 
 \psset{linewidth=0.5pt,fillstyle=solid}%
 \subdivvv(0,0)(0,-1)(0,0)
\psgrid(0,0)(-3,-3)(3,3)\endpsclip 
\end{pspicture}
\begin{pspicture}(-3.2,-2.8)(3,3.8)%
\beginpsclip 
 \psset{linewidth=0.5pt,fillstyle=solid}%
 \starbb(0.5,0.5)
\psgrid(0,0)(-3,-3)(3,3)\endpsclip 
\end{pspicture}
\begin{pspicture}(-4.2,-2.8)(3.9,3.8)%
\beginpsclip 
 \psset{linewidth=0.6pt,fillstyle=solid}%
 \starbb(0,0)
 \psset{linestyle=dashed,fillstyle=none}
 \pspolygon[linestyle=solid,fillstyle=solid,fillcolor=white](0.5,0.5)(-0.5,0.5)(-0.5,-0.5)(0.5,-0.5)
 \psgrid(0,0)(-3,-3)(3,3)
 \psset{linewidth=1pt,linecolor=white}
 \endpsclip 
\end{pspicture}}

\def \threefoldAC {
\definerays(1,1)(1,-1)(-1,-1)(-1,1)
\defineextrarays(0,1)(1,1)(1,-1)(-1,1)
\begin{pspicture}(-3,-2.8)(3,3.8)%
\beginpsclip 
 \psset{linewidth=0.5pt,fillstyle=solid}%
 \subdivhh(-1,0)(0,0)(0,0)
 \psset{fillcolor=grayiii}
 \lee(-1,0)(0,0)
\psgrid(0,0)(-3,-3)(3,3)\endpsclip 
\end{pspicture}
\begin{pspicture}(-3.2,-2.8)(3,3.8)%
\beginpsclip 
 \psset{linewidth=0.5pt,fillstyle=solid}%
 \subdivvv(0,0)(0,-1)(0,0)
\psgrid(0,0)(-3,-3)(3,3)\endpsclip 
\end{pspicture}
\begin{pspicture}(-3.2,-2.8)(3,3.8)%
\beginpsclip 
 \psset{linewidth=0.5pt,fillstyle=solid}%
\starbb(0.5,0.5)
\psgrid(0,0)(-3,-3)(3,3)\endpsclip 
\end{pspicture}
\begin{pspicture}(-4.2,-2.8)(3.9,3.8)%
\beginpsclip 
 \psset{linewidth=0.6pt,fillstyle=solid}%
 \starbb(0,0)
 \psset{linestyle=dashed,fillstyle=none}
 \pspolygon[linestyle=solid,fillstyle=solid,fillcolor=white](0.5,10)(-0.5,10)(-0.5,-0.5)(0.5,-0.5)
 \psgrid(0,0)(-3,-3)(3,3)
 \psset{linestyle=solid,linewidth=1pt,linecolor=white}
 \endpsclip 
\end{pspicture}
}

\def\threefoldAD {
\definerays(1,1)(1,-1)(-1,-1)(-1,1)
\defineextrarays(0,1)(1,0)(-1,-1)(-1,-1)
\begin{pspicture}(-3,-2.8)(3,3.8)%
\beginpsclip 
 \psset{linewidth=0.5pt,fillstyle=solid}%
 \subdivhh(-1,0)(0,0)(-1,0)
\psgrid(0,0)(-3,-3)(3,3)\endpsclip 
\end{pspicture}
\begin{pspicture}(-3.2,-2.8)(3,3.8)%
\beginpsclip 
 \psset{linewidth=0.5pt,fillstyle=solid}%
 \subdivvv(0,0)(0,-1)(0,-1)
\psgrid(0,0)(-3,-3)(3,3)\endpsclip 
\end{pspicture}
\begin{pspicture}(-3.2,-2.8)(3,3.8)%
\beginpsclip 
 \psset{linewidth=0.5pt,fillstyle=solid}%
 \subdivdd(1,1)(0.5,0.5)(0.5,0.5)
\psgrid(0,0)(-3,-3)(3,3)\endpsclip 
\end{pspicture}
\begin{pspicture}(-4.2,-2.8)(3.9,3.8)%
\beginpsclip 
 \psset{linewidth=0.6pt,fillstyle=solid}%
 \starbb(0,0)
 \psset{linestyle=dashed,fillstyle=none}
\pspolygon[linestyle=solid,fillstyle=solid,fillcolor=white](1,0)(1,1)(0,1)(-0.5,0.5)(-0.5,-0.5)(0.5,-0.5)
\psgrid(0,0)(-3,-3)(3,3)
 \psset{linewidth=1pt,linecolor=white}
\endpsclip 
\end{pspicture}}

\def\threefoldAE {
\definerays(1,1)(0,-1)(-1,-1)(-1,0)
\defineextrarays(0,1)(1,0)(-1,-1)(-1,-1)
\begin{pspicture}(-3,-2.8)(3,3.8)%
\beginpsclip 
 \psset{linewidth=0.5pt,fillstyle=solid}%
 \subdivhh(-1,0)(0,0)(-1,0)
\psgrid(0,0)(-3,-3)(3,3)\endpsclip 
\end{pspicture}
\begin{pspicture}(-3.2,-2.8)(3,3.8)%
\beginpsclip 
 \psset{linewidth=0.5pt,fillstyle=solid}%
 \subdivvv(0,0)(0,-1)(0,-1)
\psgrid(0,0)(-3,-3)(3,3)\endpsclip 
\end{pspicture}
\begin{pspicture}(-3.2,-2.8)(3,3.8)%
\beginpsclip 
 \psset{linewidth=0.5pt,fillstyle=solid}%
 \subdivdd(1,1)(0,0)(0,0)
\psgrid(0,0)(-3,-3)(3,3)\endpsclip 
\end{pspicture}
\begin{pspicture}(-4.2,-2.8)(3.9,3.8)%
\beginpsclip 
 \psset{linewidth=0.6pt,fillstyle=solid}%
 \starbb(0,0)
 \psset{linestyle=dashed,fillstyle=none}
\pspolygon[linestyle=solid,fillstyle=solid,fillcolor=white](1,1)(0,1)(-1,0)(-1,-1)(0,-1)(1,0)
 \psgrid(0,0)(-3,-3)(3,3)
 \psset{linewidth=1pt,linecolor=white}
\endpsclip 
\end{pspicture}}

\def \threefoldBB {
\definerays(1,1)(1,-1)(-1,-1)(-1,1)
\defineextrarays(0,1)(1,0)(0,-1)(-1,0)
\begin{pspicture}(-3,-2.8)(3,3.8)%
\beginpsclip 
 \psset{linewidth=0.5pt,fillstyle=solid}%
 \subdivhh(-1,0)(0,0)(-0.5,0)
\psgrid(0,0)(-3,-3)(3,3)\endpsclip 
\end{pspicture}
\begin{pspicture}(-3.2,-2.8)(3,3.8)%
\beginpsclip 
 \psset{linewidth=0.5pt,fillstyle=solid}%
 \subdivvv(0,0)(0,-1)(0,-0.5)
\psgrid(0,0)(-3,-3)(3,3)\endpsclip 
\end{pspicture}
\begin{pspicture}(-3.2,-2.8)(3,3.8)%
\beginpsclip 
 \psset{linewidth=0.5pt,fillstyle=solid}%
 \starbb(0.5,0.5)
\psgrid(0,0)(-3,-3)(3,3)\endpsclip 
\end{pspicture}
\begin{pspicture}(-4.2,-2.8)(3.9,3.8)%
\beginpsclip 
 \psset{linewidth=0.6pt,fillstyle=solid}%
 \starbb(0,0)
 \psset{linestyle=dashed,fillstyle=none}
\pspolygon[linestyle=solid,fillstyle=solid,fillcolor=white](0.5,0.5)(-0.5,0.5)(-0.5,-0.5)(0.5,-0.5)
 \psgrid(0,0)(-3,-3)(3,3)
 \psset{linewidth=1pt,linecolor=white}
\endpsclip 
\end{pspicture}}

\def \threefoldBBHL {
\definerays(1,1)(1,-1)(-1,-1)(-1,1)
\defineextrarays(0,1)(1,0)(0,-1)(-1,0)
\begin{pspicture}(-3,-2)(3,3.8)%
\beginpsclip 
 \psset{linewidth=0.5pt,fillstyle=solid}%
 \subdivhh(-1,0)(0,0)(-0.5,0)
 \psline[fillstyle=hlines,fillcolor=grayi](-4,0)(-1,0)(-4,3)
\psgrid(0,0)(-3,-1)(3,3)\psdot(0,0)
\endpsclip 
\put(2,2.7){$\fan_0$}
\end{pspicture}
\begin{pspicture}(-3.2,-2)(3,3.8)%
\beginpsclip 
 \psset{linewidth=0.5pt,fillstyle=solid}%
 \subdivvv(0,0)(0,-1)(0,-0.5)
 \psline[fillstyle=hlines,fillcolor=grayi](-4,-0.5)(0,-0.5)(0,0)(-3,3)
\psgrid(0,0)(-3,-1)(3,3)\psdot(0,0)\endpsclip 
\put(2,2.7){$\fan_\infty$}
\end{pspicture}
\begin{pspicture}(-3.2,-2)(3,3.8)%
\beginpsclip 
 \psset{linewidth=0.5pt,fillstyle=solid}%
 \starbb(0.5,0.5)
 \psline[fillstyle=hlines,fillcolor=grayi](-4,0.5)(0.5,0.5)(-2.5,3.5)
\psgrid(0,0)(-3,-1)(3,3)\psdot(0,0)\endpsclip 
\put(2,2.7){$\fan_1$}
\end{pspicture}
\begin{pspicture}(-4.2,-2)(3.9,3.8)%
\beginpsclip 
 \psset{linewidth=0.6pt,fillstyle=solid}%
 \starbb(0,0)
 \psline[fillstyle=none,linestyle=dotted](-4,0)(0,0)(-3.5,3.5)
 \psline[fillstyle=hlines,fillcolor=grayi](-4,0)(-0.5,0)(-0.5,0.5)(-3.5,3.5)
 \psset{linestyle=dashed,fillstyle=none}
 \cef(0,0)
 \pspolygon[linestyle=solid,fillstyle=solid,fillcolor=white](0.5,0.5)(-0.5,0.5)(-0.5,-0.5)(0.5,-0.5)
\psgrid(0,0)(-3,-1)(3,3)
\psdot(0,0)
\endpsclip 
\put(2,2.7){$\sdeg$}
\end{pspicture}}

\def \threefoldBC {
\definerays(1,1)(1,-1)(-1,-1)(-1,1)
\defineextrarays(0,1)(1,1)(0,-1)(-1,1)
\begin{pspicture}(-3,-2.8)(3,3.8)%
\beginpsclip 
 \psset{linewidth=0.5pt,fillstyle=solid}%
 \subdivhh(-1,0)(0,0)(-0.5,0)
 \psset{fillcolor=grayiv}
 \lee(-0.5,0)(0,0)
 \psset{fillcolor=grayii}
 \lee(-0.5,0)(-1,0)
\psgrid(0,0)(-3,-3)(3,3)\endpsclip 
\end{pspicture}
\begin{pspicture}(-3.2,-2.8)(3,3.8)%
\beginpsclip 
 \psset{linewidth=0.5pt,fillstyle=solid}%
 \subdivvv(0,0)(0,-1)(0,0)
\psgrid(0,0)(-3,-3)(3,3)\endpsclip 
\end{pspicture}
\begin{pspicture}(-3.2,-2.8)(3,3.8)%
\beginpsclip 
 \psset{linewidth=0.5pt,fillstyle=solid}%
\starbb(0.5,0.5)
\psgrid(0,0)(-3,-3)(3,3)\endpsclip 
\end{pspicture}
\begin{pspicture}(-4.2,-2.8)(3.9,3.8)%
\beginpsclip 
 \psset{linewidth=0.6pt,fillstyle=solid}%
 \starbb(0,0)
 \psset{linestyle=dashed,fillstyle=none}
 \cef(0,0)
 \pspolygon[linestyle=solid,fillstyle=solid,fillcolor=white](0.5,10)(-0.5,10)(-0.5,-0.5)(0.5,-0.5)
 \psgrid(0,0)(-3,-3)(3,3)
 \psset{linestyle=solid,linewidth=1pt,linecolor=white}
\endpsclip 
\end{pspicture}
}

\def \threefoldBD {
\definerays(1,1)(1,-1)(-1,-1)(-1,1)
\defineextrarays(0,1)(1,1)(0,-1)(-1,1)
\begin{pspicture}(-3,-2.8)(3,3.8)%
\beginpsclip 
 \psset{linewidth=0.5pt,fillstyle=solid}%
 \subdivhh(-1,0)(0,0)(0,0)
 \psset{fillcolor=grayiv}
 \lgg(-1,0)(0,0)
 \psset{fillcolor=grayiii}
 \lee(-1,0)(0,0)
\psgrid(0,0)(-3,-3)(3,3)\endpsclip 
\end{pspicture}
\begin{pspicture}(-3.2,-2.8)(3,3.8)%
\beginpsclip 
 \psset{linewidth=0.5pt,fillstyle=solid}%
 \subdivvv(0,0)(0,-1)(0,0)
\psgrid(0,0)(-3,-3)(3,3)\endpsclip 
\end{pspicture}
\begin{pspicture}(-3.2,-2.8)(3,3.8)%
\beginpsclip 
 \psset{linewidth=0.5pt,fillstyle=solid}%
\starbb(0.5,0.5)
\psgrid(0,0)(-3,-3)(3,3)\endpsclip 
\end{pspicture}
\begin{pspicture}(-4.2,-2.8)(3.9,3.8)%
\beginpsclip 
 \psset{linewidth=0.6pt,fillstyle=solid}%
 \starbb(0,0)
 \psset{linestyle=dashed,fillstyle=none}
 \cef(0,0)
 \pspolygon[linestyle=solid,fillstyle=solid,fillcolor=white](0.5,10)(-0.5,10)(-0.5,-10)(0.5,-10)
 \psgrid(0,0)(-3,-3)(3,3)
 \psset{linestyle=solid,linewidth=1pt,linecolor=white}
 \endpsclip 
\end{pspicture}
}

\def \threefoldBE {
\definerays(1,1)(1,-1)(-1,-1)(-1,1)
\defineextrarays(0,1)(1,0)(0,-1)(-1,0)
\begin{pspicture}(-3,-2.8)(3,3.8)%
\beginpsclip 
 \psset{linewidth=0.5pt,fillstyle=solid}%
 \subdivhhh(-1,0)(0,0)(0,0)(-1,0)
\psgrid(0,0)(-3,-3)(3,3)\endpsclip 
\end{pspicture}
\begin{pspicture}(-3.2,-2.8)(3,3.8)%
\beginpsclip 
 \psset{linewidth=0.5pt,fillstyle=solid}%
 \subdivvvv(0,0)(0,-1)(0,0)(0,-1)
\psgrid(0,0)(-3,-3)(3,3)\endpsclip 
\end{pspicture}
\begin{pspicture}(-3.2,-2.8)(3,3.8)%
\beginpsclip 
 \psset{linewidth=0.5pt,fillstyle=solid}%
 \subdivdd(1,1)(0,0)(0.5,0.5)

\psgrid(0,0)(-3,-3)(3,3)\endpsclip 
\end{pspicture}
\begin{pspicture}(-4.2,-2.8)(3.9,3.8)%
\beginpsclip 
 \psset{linewidth=0.6pt,fillstyle=solid}%
 \starbb(0,0)
 \psset{linestyle=dashed,fillstyle=none}
\pspolygon[linestyle=solid,fillstyle=solid,fillcolor=white](1,1)(0,1)(-1,0)(-1,-1)(0,-1)(1,0)
 \psgrid(0,0)(-3,-3)(3,3)
 \psset{linewidth=1pt,linecolor=white}
\endpsclip 
\end{pspicture}}

\def\threefoldBEA {
\definerays(1,1)(0,-1)(-1,-1)(-1,0)
\defineextrarays(1,2)(2,1)(-1,-1)(-1,-1)
\begin{pspicture}(-3,-2.8)(3,3.8)%
\beginpsclip 
 \psset{linewidth=0.5pt,fillstyle=solid}%
 \subdivhh(-0.5,0)(0,0)(-0.5,0)
\pspolygon[linestyle=solid,fillstyle=solid,fillcolor=gray](3.5,4)(-0.5,0)(0,0)(3,3)
\psdot(0,0)
\psgrid(0,0)(-3,-3)(3,3)\endpsclip 
\end{pspicture}
\begin{pspicture}(-3.2,-2.8)(3,3.8)%
\beginpsclip 
 \psset{linewidth=0.5pt,fillstyle=solid}%
 \subdivvv(0,0)(0,-0.5)(0,-0.5)
 \pspolygon[linestyle=solid,fillstyle=solid,fillcolor=gray](4,3.5)(0,-0.5)(0,0)(3,3)
\psdot(0,0)
\psgrid(0,0)(-3,-3)(3,3)\endpsclip 
\end{pspicture}
\begin{pspicture}(-3.2,-2.8)(3,3.8)%
\beginpsclip 
 \psset{linewidth=0.5pt,fillstyle=solid}%
 \subdivdd(1,1)(0,0)(0,0)
\psdot(0,0)
\psgrid(0,0)(-3,-3)(3,3)\endpsclip 
\end{pspicture}
\begin{pspicture}(-4.2,-2.8)(3.9,3.8)%
\beginpsclip 
 \psset{linewidth=0.6pt,fillstyle=solid}%
 \starbb(0,0)
 \psset{linestyle=dashed,fillstyle=none}
 \pspolygon[linestyle=solid,fillstyle=solid,fillcolor=white](3.5,4)(-0.5,0)(-0.5,-0.5)(0,-0.5)(4,3.5)
 \psgrid(0,0)(-3,-3)(3,3)
 \psset{linewidth=1pt,linecolor=white}
\endpsclip 
\end{pspicture}}

\def\threefoldBF {
\definerays(1,2)(0,-1)(-1,0)(-1,2)
\defineextrarays(0,1)(1,0)(0,-1)(-1,0)
\begin{pspicture}(-3,-2.8)(3,3.8)%
\beginpsclip 
 \psset{linewidth=0.5pt,fillstyle=solid}%
 \subdivhh(-1,0)(0,0)(-1,0)
 \psset{fillcolor=grayiii}
 \lbb(-1,0)(0,0)
 \psset{fillcolor=grayii}
 \lee(-1,0)(0,0)
\psgrid(0,0)(-3,-3)(3,3)\endpsclip 
\end{pspicture}
\begin{pspicture}(-3.2,-2.8)(3,3.8)%
\beginpsclip 
 \psset{linewidth=0.5pt,fillstyle=solid}%
 \subdivvv(0,1)(0,0)(0,0)
\psgrid(0,0)(-3,-3)(3,3)\endpsclip 
\end{pspicture}
\begin{pspicture}(-3.2,-2.8)(3,3.8)%
\beginpsclip 
 \psset{linewidth=0.5pt,fillstyle=solid}%
 \starbb(0.5,0)
 \psset{linestyle=dashed,fillstyle=none}
\psgrid(0,0)(-3,-3)(3,3)\endpsclip 
\end{pspicture}
\begin{pspicture}(-4.2,-2.8)(3.9,3.8)%
\beginpsclip 
 \psset{linewidth=0.6pt,fillstyle=solid}%
 \starbb(0,0)
 \psset{linestyle=dashed,fillstyle=none}
 \pspolygon[linestyle=solid,fillstyle=solid,fillcolor=white](0.5,10)(-0.5,10)(-0.5,-10)(0.5,-10)
 \psgrid(0,0)(-3,-3)(3,3)
 \psset{linestyle=solid,linewidth=1pt,linecolor=white}
\endpsclip 
\end{pspicture}}

\def\threefoldBG {
\definerays(1,1)(1,-1)(-1,-1)(-1,1)
\defineextrarays(0,1)(1,0)(-1,-1)(-1,-1)
\begin{pspicture}(-3,-2.8)(3,3.8)%
\beginpsclip 
 \psset{linewidth=0.5pt,fillstyle=solid}%
 \subdivhh(-1,0)(0,0)(-1,0)
 \psset{fillcolor=grayiii}
 \lee(-1,0)(0,0)
\psgrid(0,0)(-3,-3)(3,3)\endpsclip 
\end{pspicture}
\begin{pspicture}(-3.2,-2.8)(3,3.8)%
\beginpsclip 
 \psset{linewidth=0.5pt,fillstyle=solid}%
 \subdivvv(0,0)(0,-1)(0,-1)
\psgrid(0,0)(-3,-3)(3,3)\endpsclip 
\end{pspicture}
\begin{pspicture}(-3.2,-2.8)(3,3.8)%
\beginpsclip 
 \psset{linewidth=0.5pt,fillstyle=solid}%
 \subdivdd(1,1)(0.5,0.5)(0.5,0.5)
 \psset{fillcolor=grayiii}
 \lee(1,1)(0.5,0.5)
\psgrid(0,0)(-3,-3)(3,3)\endpsclip 
\end{pspicture}
\begin{pspicture}(-4.2,-2.8)(3.9,3.8)%
\beginpsclip 
 \psset{linewidth=0.6pt,fillstyle=solid}%
 \starbb(0,0)
 \psset{linestyle=dashed,fillstyle=none}
\pspolygon[linestyle=solid,fillstyle=solid,fillcolor=white](1,0)(1,10)(-0.5,10)(-0.5,-0.5)(0.5,-0.5)
 \psgrid(0,0)(-3,-3)(3,3)
 \psset{linewidth=1pt,linestyle=solid,linecolor=white}
\endpsclip 
\end{pspicture}}

\def\threefoldBH {
\definerays(1,1)(0,-1)(-1,-1)(-1,0)
\defineextrarays(0,1)(1,0)(-1,-1)(-1,-1)
\begin{pspicture}(-3,-2.8)(3,3.8)%
\beginpsclip 
 \psset{linewidth=0.5pt,fillstyle=solid}%
 \subdivhh(-1,0)(0,0)(-1,0)
 \psset{fillcolor=grayiii}
 \lee(-1,0)(0,0)
\psgrid(0,0)(-3,-3)(3,3)\endpsclip 
\end{pspicture}
\begin{pspicture}(-3.2,-2.8)(3,3.8)%
\beginpsclip 
 \psset{linewidth=0.5pt,fillstyle=solid}%
 \subdivvv(0,0)(0,-1)(0,-1)
\psgrid(0,0)(-3,-3)(3,3)\endpsclip 
\end{pspicture}
\begin{pspicture}(-3.2,-2.8)(3,3.8)%
\beginpsclip 
 \psset{linewidth=0.5pt,fillstyle=solid}%
 \subdivdd(1,1)(0,0)(0,0)
 \psset{fillcolor=grayiii}
 \lee(1,1)(0,0)
\psgrid(0,0)(-3,-3)(3,3)\endpsclip 
\end{pspicture}
\begin{pspicture}(-4.2,-2.8)(3.9,3.8)%
\beginpsclip 
 \psset{linewidth=0.6pt,fillstyle=solid}%
 \starbb(0,0)
 \psset{linestyle=dashed,fillstyle=none}
\pspolygon[linestyle=solid,fillstyle=solid,fillcolor=white](1,10)(-1,10)(-1,-1)(0,-1)(1,0)
 \psgrid(0,0)(-3,-3)(3,3)
 \psset{linewidth=1pt,linecolor=white}
\endpsclip 
\end{pspicture}}

\def \threefoldCA {
\definerays(1,1)(1,-1)(-1,-1)(-1,1)
\defineextrarays(0,1)(1,1)(0,-1)(-1,1)
\begin{pspicture}(-3,-2.8)(3,3.8)%
\beginpsclip 
 \psset{linewidth=0.5pt,fillstyle=solid}%
 \subdivhh(-1,0)(0,0)(0,0)
 \psset{fillcolor=grayiii}
 \lgg(-0.5,0)(0,0)
  \psset{fillcolor=grayiv}
  \lgg(-0.5,0)(-1,0)
 \lee(-0.5,0)(0,0)
 \psset{fillcolor=grayii}
 \lee(-0.5,0)(-1,0)
 \psset{fillcolor=grayi}
\psgrid(0,0)(-3,-3)(3,3)\endpsclip 
\end{pspicture}
\begin{pspicture}(-3.2,-2.8)(3,3.8)%
\beginpsclip 
 \psset{linewidth=0.5pt,fillstyle=solid}%
 \subdivvv(0,0)(0,-1)(0,0)
\psgrid(0,0)(-3,-3)(3,3)\endpsclip 
\end{pspicture}
\begin{pspicture}(-3.2,-2.8)(3,3.8)%
\beginpsclip 
 \psset{linewidth=0.5pt,fillstyle=solid}%
\starbb(0.5,0.5)
\psgrid(0,0)(-3,-3)(3,3)\endpsclip 
\end{pspicture}
\begin{pspicture}(-4.2,-2.8)(3.9,3.8)%
\beginpsclip 
 \psset{linewidth=0.6pt,fillstyle=solid}%
 \starbb(0,0)
 \psset{linestyle=dashed,fillstyle=none}
 \cef(0,0)
 \pspolygon[linestyle=solid,fillstyle=solid,fillcolor=white](0.5,10)(-0.5,10)(-0.5,-10)(0.5,-10)
 \psgrid(0,0)(-3,-3)(3,3)
 \psset{linestyle=solid,linewidth=1pt,linecolor=white}
\endpsclip 
\end{pspicture}
}

\def\threefoldCAA {
\definerays(1,1)(0,-1)(-1,-1)(-1,0)
\defineextrarays(1,2)(2,1)(-1,-1)(-1,-1)
\begin{pspicture}(-3,-2.8)(3,3.8)%
\beginpsclip 
 \psset{linewidth=0.5pt,fillstyle=solid}%
 \subdivhh(-0.5,0)(0,0)(-0.5,0)
\pspolygon[linestyle=solid,fillstyle=solid,fillcolor=gray](-4.5,-4)(-0.5,0)(0,0)(-3,-3)
\pspolygon[linestyle=solid,fillstyle=solid,fillcolor=gray](3.5,4)(-0.5,0)(0,0)(3,3)
\psgrid(0,0)(-3,-3)(3,3)\endpsclip 
\end{pspicture}
\begin{pspicture}(-3.2,-2.8)(3,3.8)%
\beginpsclip 
 \psset{linewidth=0.5pt,fillstyle=solid}%
 \subdivvv(0,0)(0,-0.5)(0,-0.5)
\pspolygon[linestyle=solid,fillstyle=solid,fillcolor=gray](-4,-4.5)(0,-0.5)(0,0)(-3,-3)
 \pspolygon[linestyle=solid,fillstyle=solid,fillcolor=gray](4,3.5)(0,-0.5)(0,0)(3,3)
\psgrid(0,0)(-3,-3)(3,3)\endpsclip 
\end{pspicture}
\begin{pspicture}(-3.2,-2.8)(3,3.8)%
\beginpsclip 
 \psset{linewidth=0.5pt,fillstyle=solid}%
 \subdivdd(1,1)(0,0)(0,0)
\psgrid(0,0)(-3,-3)(3,3)\endpsclip 
\end{pspicture}
\begin{pspicture}(-4.2,-2.8)(3.9,3.8)%
\beginpsclip 
 \psset{linewidth=0.6pt,fillstyle=solid}%
 \starbb(0,0)
 \psset{linestyle=dashed,fillstyle=none}
 \pspolygon[linestyle=solid,fillstyle=solid,fillcolor=white](3.5,4)
(-4,-3.5)(-3.5,-4)(4,3.5)
 \psgrid(0,0)(-3,-3)(3,3)
 \psset{linewidth=1pt,linecolor=white}
\endpsclip 
\end{pspicture}}

\def\threefoldCB {
\definerays(1,1)(0,-1)(-1,-1)(-1,0)
\defineextrarays(0,1)(1,0)(-1,-1)(-1,-1)
\begin{pspicture}(-3,-2.8)(3,3.8)%
\beginpsclip 
 \psset{linewidth=0.5pt,fillstyle=solid}%
 \subdivhh(-1,0)(0,0)(-1,0)
 \psset{fillcolor=grayv}
 \lbb(-1,0)(0,0)
 \psset{fillcolor=grayiii}
 \lee(-1,0)(0,0)
\psgrid(0,0)(-3,-3)(3,3)\endpsclip 
\end{pspicture}
\begin{pspicture}(-3.2,-2.8)(3,3.8)%
\beginpsclip 
 \psset{linewidth=0.5pt,fillstyle=solid}%
 \subdivvv(0,0)(0,-1)(0,-1)
\psgrid(0,0)(-3,-3)(3,3)\endpsclip 
\end{pspicture}
\begin{pspicture}(-3.2,-2.8)(3,3.8)%
\beginpsclip 
 \psset{linewidth=0.5pt,fillstyle=solid}%
 \subdivdd(1,1)(0,0)(0,0)
 \psset{fillcolor=grayv}
 \lbb(1,1)(0,0)
 \psset{fillcolor=grayiii}
 \lee(1,1)(0,0)
\psgrid(0,0)(-3,-3)(3,3)\endpsclip 
\end{pspicture}
\begin{pspicture}(-4.2,-2.8)(3.9,3.8)%
\beginpsclip 
 \psset{linewidth=0.6pt,fillstyle=solid}%
 \starbb(0,0)
 \psset{linestyle=dashed,fillstyle=none}
\pspolygon[linestyle=solid,fillstyle=solid,fillcolor=white](-1,10)(1,10)(1,-10)(-1,-10)
 \psgrid(0,0)(-3,-3)(3,3)
 \psset{linewidth=1pt,linecolor=white}
\endpsclip 
\end{pspicture}}

\def\threefoldCC {
\definerays(1,1)(0,-1)(-1,0)(-1,1)
\defineextrarays(0,1)(1,0)(0,-1)(-1,0)
\begin{pspicture}(-3,-2.8)(3,3.8)%
\beginpsclip 
 \psset{linewidth=0.5pt,fillstyle=solid}%
 \subdivhh(-1,0)(0,0)(-1,0)
 \psset{fillcolor=grayiii}
 \lbb(-1,0)(0,0)
 \psset{fillcolor=grayii}
 \lee(-1,0)(0,0)
\psgrid(0,0)(-3,-3)(3,3)\endpsclip 
\end{pspicture}
\begin{pspicture}(-3.2,-2.8)(3,3.8)%
\beginpsclip 
 \psset{linewidth=0.5pt,fillstyle=solid}%
 \subdivvv(0,1)(0,0)(0,0)
\psgrid(0,0)(-3,-3)(3,3)\endpsclip 
\end{pspicture}
\begin{pspicture}(-3.2,-2.8)(3,3.8)%
\beginpsclip 
 \psset{linewidth=0.5pt,fillstyle=solid}%
 \subdivhh(0,0)(1,0)(0,0)
 \psset{fillcolor=grayii}
 \lee(1,0)(0,0)
\psgrid(0,0)(-3,-3)(3,3)\endpsclip 
\end{pspicture}
\begin{pspicture}(-4.2,-2.8)(3.9,3.8)%
\beginpsclip 
 \psset{linewidth=0.6pt,fillstyle=solid}%
 \starbb(0,0)
 \psset{linestyle=dashed,fillstyle=none}
\pspolygon[linestyle=solid,fillstyle=solid,fillcolor=white](-1,10)(1,10)(1,-10)(-1,-10)
 \psgrid(0,0)(-3,-3)(3,3)
 \psset{linestyle=solid,linewidth=1pt,linecolor=white}
\endpsclip 
\end{pspicture}}
\ifdef{\ebook}{
\psset{unit=0.25cm}}{
\psset{unit=0.35cm}}

\def\dualQ{%
\begin{pspicture}(-4,-4)(4,4)%
\ifdef{\ebook}{\tiny}{}
\psgrid(0,0)(-3,-3)(3,3)%
\psline(3,0)(0,-3)(-3,0)(0,3)(3,0)
\psline[linecolor=gray](0,3)(0,-3)
\put(0,3){$0$}
\put(3,0){$-3$}
\put(-3,0){$0$}
\put(0,-3){$0$}
\end{pspicture}
\begin{pspicture}(-4,-4)(4,4)%
\ifdef{\ebook}{\tiny}{}
\psgrid(0,0)(-3,-3)(3,3)%
\psline(3,0)(0,-3)(-3,0)(0,3)(3,0)
\psline[linecolor=gray](3,0)(-3,0)
\put(0,3){$0$}
\put(3,0){$0$}
\put(-3,0){$0$}
\put(0,-3){$-3$}
\end{pspicture}
\begin{pspicture}(-4,-4)(4,4)%
\ifdef{\ebook}{\tiny}{}
\psgrid(0,0)(-3,-3)(3,3)%
\psline(3,0)(0,-3)(-3,0)(0,3)(3,0)
\put(0,3){$0$}
\put(3,0){$3$}
\put(-3,0){$0$}
\put(0,-3){$3$}
\end{pspicture}
\begin{pspicture}(-4,-4)(4,4)%
\ifdef{\ebook}{\tiny}{}
\psgrid(0,0)(-3,-3)(3,3)%
\psline(3,0)(0,-3)(-3,0)(0,3)(3,0)
\psline[linecolor=gray](3,0)(-3,0)
\psline[linecolor=gray](0,3,0)(0,-3)
\psdot(0,0)
\put(0,0){$\nicefrac{3}{2}$}
\put(0,3){$0$}
\put(3,0){$0$}
\put(-3,0){$0$}
\put(0,-3){$0$}
\end{pspicture}
}
\def\dualAA{%
\begin{pspicture}(-4,-4)(4,4)%
\ifdef{\ebook}{\tiny}{}
\psgrid(0,0)(-3,-3)(3,3)%
\psline(-2,3)(-1,3)(3,-1)(3,-2)(-1,-2)(-2,-1)(-2,3)(-1,3)
\psline[linecolor=gray](-1,3)(-1,-2)(3,-2)(3,-1)(-1,3)
\put(-2,3){$0$}
\put(-1,3){$0$}
\put(-2,-1){$0$}
\put(-1,-2){$0$}
\put(3,-1){$-2$}
\put(3,-2){$-2$}
\end{pspicture}
\begin{pspicture}(-4,-4)(4,4)%
\ifdef{\ebook}{\tiny}{}
\psgrid(0,0)(-3,-3)(3,3)%
\psline(-2,3)(-1,3)(3,-1)(3,-2)(-1,-2)(-2,-1)(-2,3)(-1,3)
\psline[linecolor=gray](3,-1)(-2,-1)(-2,3)(-1,3)(3,-1)
\put(3,-2){$0$}
\put(3,-1){$0$}
\put(-1,-2){$0$}
\put(-2,-1){$0$}
\put(-1,3){$-2$}
\put(-2,3){$-2$}
\end{pspicture}
\begin{pspicture}(-4,-4)(4,4)%
\ifdef{\ebook}{\tiny}{}
\psgrid(0,0)(-3,-3)(3,3)%
\psline(-2,3)(-1,3)(3,-1)(3,-2)(-1,-2)(-2,-1)(-2,3)(-1,3)
\psline[linecolor=gray](3,-2)(-2,3)(-2,-1)(-1,-2)(3,-2)
\put(3,-2){$2$}
\put(3,-1){$2$}
\put(-1,-2){$0$}
\put(-2,-1){$0$}
\put(-1,3){$2$}
\put(-2,3){$2$}
\end{pspicture}
\begin{pspicture}(-4,-4)(4,4)%
\ifdef{\ebook}{\tiny}{}
\psgrid(0,0)(-3,-3)(3,3)%
\psline(-2,3)(-1,3)(3,-1)(3,-2)(-1,-2)(-2,-1)(-2,3)(-1,3)
\psline[linecolor=gray](-1,3)(-1,-2)
\psline[linecolor=gray](3,-1)(-2,-1)
\psline[linecolor=gray](3,-2)(-2,3)
\psdot(0,0)
\put(3,-2){$0$}
\put(3,-1){$0$}
\put(-1,-2){$0$}
\put(-2,-1){$0$}
\put(-1,3){$0$}
\put(-2,3){$0$}
\put(-1,2){$\nicefrac{1}{2}$}
\put(2,-1){$\nicefrac{1}{2}$}
\put(-1,-1){$\nicefrac{1}{2}$}
\end{pspicture}
}

\def\dualAB{%
\begin{pspicture}(-4,-4)(4,4)%
\ifdef{\ebook}{\tiny}{}
\psgrid(0,0)(-3,-3)(3,3)%
\psline[linecolor=gray](-1,2)(-1,-2)(1,-2)(1,2)
\psline(3,0)(1,-2)(-1,-2)(-3,0)(-1,2)(1,2)(3,0)
\put(1,2){$-1$}
\put(-1,2){$0$}
\put(-3,0){$0$}
\put(3,0){$-3$}
\put(1,-2){$-1$}
\put(-1,-2){$0$}
\end{pspicture}
\begin{pspicture}(-4,-4)(4,4)%
\ifdef{\ebook}{\tiny}{}
\psgrid(0,0)(-3,-3)(3,3)%
\psline(3,0)(1,-2)(-1,-2)(-3,0)(-1,2)(1,2)(3,0)
\psline[linecolor=gray](3,0)(-3,0)
\put(1,2){$0$}
\put(-1,2){$0$}
\put(-3,0){$0$}
\put(3,0){$0$}
\put(1,-2){$-2$}
\put(-1,-2){$-2$}
\end{pspicture}
\begin{pspicture}(-4,-4)(4,4)%
\ifdef{\ebook}{\tiny}{}
\psgrid(0,0)(-3,-3)(3,3)%
\psline(3,0)(1,-2)(-1,-2)(-3,0)(-1,2)(1,2)(3,0)
\put(1,2){$1$}
\put(-1,2){$0$}
\put(-3,0){$0$}
\put(3,0){$3$}
\put(1,-2){$3$}
\put(-1,-2){$2$}
\end{pspicture}
\begin{pspicture}(-4,-4)(4,4)%
\ifdef{\ebook}{\tiny}{}
\psgrid(0,0)(-3,-3)(3,3)%
\psline(3,0)(1,-2)(-1,-2)(-3,0)(-1,2)(1,2)(3,0)
\psline[linecolor=gray](3,0)(-3,0)
\psline[linecolor=gray](-1,2)(-1,-2)
\psline[linecolor=gray](1,2)(1,-2)
\psdot(0,0)
\put(0,0){$1$}
\put(1,2){$0$}
\put(-1,2){$0$}
\put(-3,0){$0$}
\put(3,0){$0$}
\put(1,-2){$0$}
\put(-1,-2){$0$}
\end{pspicture}
}

\def\dualAC{%
\begin{pspicture}(-4,-4)(4,4)%
\ifdef{\ebook}{\tiny}{}
\psgrid(0,0)(-3,-3)(3,3)%
\psline(3,0)(0,-3)(-3,0)(-2,1)(2,1)(3,0)
\psline[linecolor=gray](0,1)(0,-3)
\put(-2,1){$0$}
\put(0,1){$0$}
\put(-3,0){$0$}
\put(2,1){$-2$}
\put(3,0){$-3$}
\put(0,-3){$0$}
\end{pspicture}
\begin{pspicture}(-4,-4)(4,4)%
\ifdef{\ebook}{\tiny}{}
\psgrid(0,0)(-3,-3)(3,3)%
\psline(3,0)(0,-3)(-3,0)(-2,1)(2,1)(3,0)
\psline[linecolor=gray](3,0)(-3,0)
\put(-2,1){$0$}
\put(-3,0){$0$}
\put(2,1){$0$}
\put(3,0){$0$}
\put(0,-3){$-3$}
\end{pspicture}
\begin{pspicture}(-4,-4)(4,4)%
\ifdef{\ebook}{\tiny}{}
\psgrid(0,0)(-3,-3)(3,3)%
\psline(3,0)(0,-3)(-3,0)(-2,1)(2,1)(3,0)
\put(-2,1){$0$}
\put(-3,0){$0$}
\put(2,1){$2$}
\put(3,0){$3$}
\put(0,-3){$3$}
\end{pspicture}
\begin{pspicture}(-4,-4)(4,4)%
\ifdef{\ebook}{\tiny}{}
\psgrid(0,0)(-3,-3)(3,3)%
\psline(3,0)(0,-3)(-3,0)(-2,1)(2,1)(3,0)
\psline[linecolor=gray](3,0)(-3,0)
\psline[linecolor=gray](0,1)(0,-3)
\psdot[dotstyle=x,dotsize=6px](0,-2)
\psdot(0,0)
\put(-2,1){$0$}
\put(-3,0){$0$}
\put(2,1){$0$}
\put(3,0){$0$}
\put(0,-3){$0$}\put(0,1){$1$}
\put(0,0){$\nicefrac{3}{2}$}
\end{pspicture}
}

\def\dualAD{%
\begin{pspicture}(-4,-4)(4,4)%
\ifdef{\ebook}{\tiny}{}
\psgrid(0,0)(-3,-3)(3,3)%
\psline(2,-1)(0,-3)(-3,0)(-1,2)(0,2)(2,0)(2,-1)
\psline[linecolor=gray](0,2)(0,-3)
\put(-1,2){$0$}
\put(0,2){$0$}
\put(-3,0){$0$}
\put(2,0){$-2$}
\put(2,-1){$-2$}
\put(0,-3){$0$}
\end{pspicture}
\begin{pspicture}(-4,-4)(4,4)%
\ifdef{\ebook}{\tiny}{}
\psgrid(0,0)(-3,-3)(3,3)%
\psline(2,-1)(0,-3)(-3,0)(-1,2)(0,2)(2,0)(2,-1)
\psline[linecolor=gray](2,0)(-3,0)
\put(-1,2){$0$}
\put(0,2){$0$}
\put(-3,0){$0$}
\put(2,0){$0$}
\put(2,-1){$-1$}
\put(0,-3){$-3$}
\end{pspicture}
\begin{pspicture}(-4,-4)(4,4)%
\ifdef{\ebook}{\tiny}{}
\psgrid(0,0)(-3,-3)(3,3)%
\psline(2,-1)(0,-3)(-3,0)(-1,2)(0,2)(2,0)(2,-1)
\psline[linecolor=gray](-1,2)(2,-1)
\put(-1,2){$0$}
\put(0,2){$0$}
\put(-3,0){$0$}
\put(2,0){$2$}
\put(2,-1){$3$}
\put(0,-3){$3$}
\end{pspicture}
\begin{pspicture}(-4,-4)(4,4)%
\ifdef{\ebook}{\tiny}{}
\psgrid(0,0)(-3,-3)(3,3)%
\psline(2,-1)(0,-3)(-3,0)(-1,2)(0,2)(2,0)(2,-1)
\psline[linecolor=gray](2,-1)(-1,2)
\psline[linecolor=gray](2,0)(-3,0)
\psline[linecolor=gray](0,2)(0,-3)
\psdot[dotstyle=x,dotsize=6px](-1.33,-1.33)
\psdot(0,0)
\put(-1,2){$0$}
\put(0,2){$0$}
\put(-3,0){$0$}
\put(2,0){$0$}
\put(2,-1){$0$}
\put(0,-3){$0$}
\put(-1,-0.5){$\nicefrac{3}{2}$}
\put(1,0){$1$}
\put(0,1){$1$}
\end{pspicture}
}

\def\dualAE {
\begin{pspicture}(-4,-4)(4,4)%
\psgrid(0,0)(-3,-3)(3,3)%
\psline(-2,2)(0,2)(2,0)(2,-2)(0,-2)(-2,0)(-2,2)
\psline[linecolor=gray](0,2)(0,-2)
\put(-2,2){$-2$}
\put(0,2){$0$}
\put(2,0){$0$}
\put(2,-2){$0$}
\put(0,-2){$0$}
\put(-2,0){$-2$}
\end{pspicture}
\begin{pspicture}(-4,-4)(4,4)%
\psgrid(0,0)(-3,-3)(3,3)%
\psline(-2,2)(0,2)(2,0)(2,-2)(0,-2)(-2,0)(-2,2)
\psline[linecolor=gray](2,0)(-2,0)
\put(-2,2){$0$}
\put(0,2){$0$}
\put(2,0){$0$}
\put(2,-2){$-2$}
\put(0,-2){$-2$}
\put(-2,0){$0$}
\end{pspicture}
\begin{pspicture}(-4,-4)(4,4)%
\psgrid(0,0)(-3,-3)(3,3)%
\psline(-2,2)(0,2)(2,0)(2,-2)(0,-2)(-2,0)(-2,2)
\psline[linecolor=gray](-2,2)(2,-2)
\put(-2,2){$2$}
\put(0,2){$0$}
\put(2,0){$0$}
\put(2,-2){$2$}
\put(0,-2){$2$}
\put(-2,0){$2$}
\end{pspicture}
\begin{pspicture}(-4,-4)(4,4)%
\psgrid(0,0)(-3,-3)(3,3)%
\psline(-2,2)(0,2)(2,0)(2,-2)(0,-2)(-2,0)(-2,2)
\psline[linecolor=gray](2,0)(-2,0)
\psline[linecolor=gray](0,2)(0,-2)
\psline[linecolor=gray](-2,2)(2,-2)
\psdot(0,0)
\put(-2,2){$0$}
\put(0,2){$0$}
\put(2,0){$0$}
\put(2,-2){$0$}
\put(0,-2){$0$}
\put(-2,0){$0$}
\put(0,0){$2$}
\end{pspicture}
}

\def\dualBB{%
\begin{pspicture}(-4,-4)(4,4)%
\psgrid(0,0)(-3,-3)(3,3)%
\psline[linecolor=gray](-1,2)(-1,-2)(1,-2)(1,2)
\psline(2,1)(2,-1)(1,-2)(-1,-2)(-2,-1)(-2,1)(-1,2)(1,2)(2,1)
\put(-1,2){$0$}
\put(-2,1){$0$}
\put(1,2){$-1$}
\put(2,1){$-2$}
\put(1,-2){$-1$}
\put(2,-1){$-2$}
\put(-1,-2){$0$}
\put(-2,-1){$0$}
\end{pspicture}
\begin{pspicture}(-4,-4)(4,4)%
\psgrid(0,0)(-3,-3)(3,3)%
\psline[linecolor=gray](2,-1)(-2,-1)(-2,1)(2,1)
\psline(2,1)(2,-1)(1,-2)(-1,-2)(-2,-1)(-2,1)(-1,2)(1,2)(2,1)
\put(-1,2){$0$}
\put(-2,1){$0$}
\put(1,2){$0$}
\put(2,1){$0$}
\put(1,-2){$-2$}
\put(2,-1){$-1$}
\put(-1,-2){$-2$}
\put(-2,-1){$-1$}
\end{pspicture}
\begin{pspicture}(-4,-4)(4,4)%
\psgrid(0,0)(-3,-3)(3,3)%
\psline(2,1)(2,-1)(1,-2)(-1,-2)(-2,-1)(-2,1)(-1,2)(1,2)(2,1)
\put(-1,2){$0$}
\put(-2,1){$0$}
\put(1,2){$1$}
\put(2,1){$2$}
\put(1,-2){$3$}
\put(2,-1){$3$}
\put(-1,-2){$2$}
\put(-2,-1){$1$}
\end{pspicture}
\begin{pspicture}(-4,-4)(4,4)%
\psgrid(0,0)(-3,-3)(3,3)%
\psline[linecolor=gray](2,-1)(-2,-1)
\psline[linecolor=gray](2,1)(-2,1)
\psline(2,1)(2,-1)(1,-2)(-1,-2)(-2,-1)(-2,1)(-1,2)(1,2)(2,1)
\psline[linecolor=gray](-1,2)(-1,-2)
\psline[linecolor=gray](1,2)(1,-2)
\psdot(0,0)
\put(-1,2){$0$}
\put(-2,1){$0$}
\put(1,2){$0$}
\put(2,1){$0$}
\put(1,-2){$0$}
\put(2,-1){$0$}
\put(-1,-2){$0$}
\put(-2,-1){$0$}
\put(0,0){$\nicefrac{1}{2}$}
\end{pspicture}
}

\def\dualBBUnshaded{%
\begin{pspicture}(-4,-4)(4,4)%
\put(3,3){$\Psi_0$}
\psgrid(0,0)(-3,-3)(3,3)%
\psline[linecolor=gray](-1,2)(-1,-2)(1,-2)(1,2)
\psline(2,1)(2,-1)(1,-2)(-1,-2)(-2,-1)(-2,1)(-1,2)(1,2)(2,1)
\put(-1,2){$0$}
\put(-2,1){$0$}
\put(1,2){$-1$}
\put(2,1){$-2$}
\put(1,-2){$-1$}
\put(2,-1){$-2$}
\put(-1,-2){$0$}
\put(-2,-1){$0$}
\end{pspicture}
\begin{pspicture}(-4,-4)(4,4)%
\put(3,3){$\Psi_\infty$}
\psgrid(0,0)(-3,-3)(3,3)%
\psline[linecolor=gray](2,-1)(-2,-1)(-2,1)(2,1)
\psline(2,1)(2,-1)(1,-2)(-1,-2)(-2,-1)(-2,1)(-1,2)(1,2)(2,1)
\put(-1,2){$0$}
\put(-2,1){$0$}
\put(1,2){$0$}
\put(2,1){$0$}
\put(1,-2){$-2$}
\put(2,-1){$-1$}
\put(-1,-2){$-2$}
\put(-2,-1){$-1$}
\end{pspicture}
\begin{pspicture}(-4,-4)(4,4)%
\put(3,3){$\Psi_1$}
\psgrid(0,0)(-3,-3)(3,3)%
\psline[](2,1)(2,-1)(1,-2)(-1,-2)(-2,-1)(-2,1)(-1,2)(1,2)(2,1)
\put(-1,2){$0$}
\put(-2,1){$0$}
\put(1,2){$1$}
\put(2,1){$2$}
\put(1,-2){$3$}
\put(2,-1){$3$}
\put(-1,-2){$2$}
\put(-2,-1){$1$}
\end{pspicture}
\begin{pspicture}(-4,-4)(4,4)%
\psgrid(0,0)(-3,-3)(3,3)%
\put(3,3){$\deg \Psi$}
\psline[linecolor=gray](2,-1)(-2,-1)
\psline[linecolor=gray](2,1)(-2,1)
\psline(2,1)(2,-1)(1,-2)(-1,-2)(-2,-1)(-2,1)(-1,2)(1,2)(2,1)
\psline[linecolor=gray](-1,2)(-1,-2)
\psline[linecolor=gray](1,2)(1,-2)
\psdot(0,0)
\put(-1,2){$0$}
\put(-2,1){$0$}
\put(1,2){$0$}
\put(2,1){$0$}
\put(1,-2){$0$}
\put(2,-1){$0$}
\put(-1,-2){$0$}
\put(-2,-1){$0$}
\put(0,0){$\nicefrac{1}{2}$}
\end{pspicture}
}

\def\dualBC{%
\begin{pspicture}(-4,-4)(4,4)%
\psgrid(0,0)(-3,-3)(3,3)%
\psline[linecolor=gray](-1,1)(-1,-2)(1,-2)(1,1)
\psline(3,0)(1,-2)(-1,-2)(-3,0)(-2,1)(2,1)(3,0)
\put(-3,0){$0$}
\put(-2,1){$0$}
\put(3,0){$-3$}
\put(2,1){$-2$}
\put(1,-2){$-1$}
\put(1,1){$-1$}
\put(-1,1){$0$}
\put(-1,-2){$0$}
\end{pspicture}
\begin{pspicture}(-4,-4)(4,4)%
\psgrid(0,0)(-3,-3)(3,3)%
\psline(3,0)(1,-2)(-1,-2)(-3,0)(-2,1)(2,1)(3,0)
\psline[linecolor=gray](3,0)(-3,0)
\put(-3,0){$0$}
\put(-2,1){$0$}
\put(3,0){$0$}
\put(2,1){$0$}
\put(1,-2){$-2$}
\put(-1,-2){$-2$}
\end{pspicture}
\begin{pspicture}(-4,-4)(4,4)%
\psgrid(0,0)(-3,-3)(3,3)%
\psline(3,0)(1,-2)(-1,-2)(-3,0)(-2,1)(2,1)(3,0)
\put(-3,0){$0$}
\put(-2,1){$0$}
\put(3,0){$3$}
\put(2,1){$2$}
\put(1,-2){$3$}
\put(-1,-2){$2$}
\end{pspicture}
\begin{pspicture}(-4,-4)(4,4)%
\psgrid(0,0)(-3,-3)(3,3)%
\psline(3,0)(1,-2)(-1,-2)(-3,0)(-2,1)(2,1)(3,0)
\psline[linecolor=gray](3,0)(-3,0)
\psline[linecolor=gray](-1,1)(-1,-2)
\psline[linecolor=gray](1,1)(1,-2)
\psdot[dotstyle=x,dotsize=6px](0,-0.875)
\psdot(0,0)
\put(-3,0){$0$}
\put(-2,1){$0$}
\put(3,0){$0$}
\put(2,1){$0$}
\put(1,-2){$0$}
\put(-1,-2){$0$}
\put(1,-0.5){$1$}
\put(-1,-0.5){$1$}
\put(-1.2,1){$\nicefrac{1}{2}$}
\put(0.7,1){$\nicefrac{1}{2}$}
\end{pspicture}
}

\def\dualBD{%
\begin{pspicture}(-4,-4)(4,4)%
\psgrid(0,0)(-3,-3)(3,3)%
\psline(3,0)(2,-1)(-2,-1)(-3,0)(-2,1)(2,1)(3,0)
\psline[linecolor=gray](0,1)(0,-1)
\put(-3,0){$0$}
\put(-2,1){$0$}
\put(3,0){$-3$}
\put(2,1){$-2$}
\put(-2,-1){$0$}
\put(2,-1){$-2$}
\put(0,-1){$0$}
\put(0,1){$0$}
\end{pspicture}
\begin{pspicture}(-4,-4)(4,4)%
\psgrid(0,0)(-3,-3)(3,3)%
\psline(3,0)(2,-1)(-2,-1)(-3,0)(-2,1)(2,1)(3,0)
\psline[linecolor=gray](3,0)(-3,0)
\put(-3,0){$0$}
\put(-2,1){$0$}
\put(3,0){$0$}
\put(2,1){$0$}
\put(-2,-1){$-1$}
\put(2,-1){$-1$}
\end{pspicture}
\begin{pspicture}(-4,-4)(4,4)%
\psgrid(0,0)(-3,-3)(3,3)%
\psline(3,0)(2,-1)(-2,-1)(-3,0)(-2,1)(2,1)(3,0)
\put(-3,0){$0$}
\put(-2,1){$0$}
\put(3,0){$3$}
\put(2,1){$2$}
\put(-2,-1){$1$}
\put(2,-1){$3$}
\end{pspicture}
\begin{pspicture}(-4,-4)(4,4)%
\psgrid(0,0)(-3,-3)(3,3)%
\psline(3,0)(2,-1)(-2,-1)(-3,0)(-2,1)(2,1)(3,0)
\psline[linecolor=gray](3,0)(-3,0)
\psline[linecolor=gray](0,1)(0,-1)
\psdot(0,0)
\put(-3,0){$0$}
\put(-2,1){$0$}
\put(3,0){$0$}
\put(2,1){$0$}
\put(-2,-1){$0$}
\put(2,-1){$0$}
\put(0,-1){$1$}
\put(0,1){$1$}
\put(0,0){$\nicefrac{3}{2}$}
\end{pspicture}
}

\def\dualBE{%
\begin{pspicture}(-4,-4)(4,4)%
\psgrid(0,0)(-3,-3)(3,3)%
\psline(2,-1)(1,-2)(0,-2)(-2,0)(-2,1)(-1,2)(0,2)(2,0)(2,-1)
\psline[linecolor=gray](0,2)(0,-2)
\put(-2,1){$0$}
\put(-1,2){$0$}
\put(0,2){$0$}
\put(2,0){$-2$}
\put(2,-1){$-2$}
\put(-2,0){$0$}
\put(0,-2){$0$}
\put(1,-2){$-1$}
\end{pspicture}\begin{pspicture}(-4,-4)(4,4)%
\psgrid(0,0)(-3,-3)(3,3)%
\psline(2,-1)(1,-2)(0,-2)(-2,0)(-2,1)(-1,2)(0,2)(2,0)(2,-1)
\psline[linecolor=gray](2,0)(-2,0)
\put(-2,1){$0$}
\put(-1,2){$0$}
\put(0,2){$0$}
\put(2,-1){$-1$}
\put(2,0){$0$}
\put(-2,0){$0$}
\put(0,-2){$-2$}
\put(1,-2){$-2$}
\end{pspicture}
\begin{pspicture}(-4,-4)(4,4)%
\psgrid(0,0)(-3,-3)(3,3)%
\psline[linecolor=gray](2,-1)(1,-2)(-2,1)(-1,2)(2,-1)
\psline(2,-1)(1,-2)(0,-2)(-2,0)(-2,1)(-1,2)(0,2)(2,0)(2,-1)
\put(-2,1){$0$}
\put(-1,2){$0$}
\put(0,2){$0$}
\put(2,0){$2$}
\put(2,-1){$3$}
\put(-2,0){$0$}
\put(0,-2){$2$}
\put(1,-2){$3$}
\end{pspicture}
\begin{pspicture}(-4,-4)(4,4)%
\psgrid(0,0)(-3,-3)(3,3)%
\psline(2,-1)(1,-2)(0,-2)(-2,0)(-2,1)(-1,2)(0,2)(2,0)(2,-1)
\psline[linecolor=gray](-2,1)(1,-2)
\psline[linecolor=gray](2,-1)(-1,2)
\psline[linecolor=gray](2,0)(-2,0)
\psline[linecolor=gray](0,2)(0,-2)
\psdot(0,0)
\put(-2,1){$0$}
\put(-1,2){$0$}
\put(0,2){$0$}
\put(2,0){$0$}
\put(2,-1){$0$}
\put(-2,0){$0$}
\put(0,-2){$0$}
\put(1,-2){$0$}
\put(0,0){$\nicefrac{3}{2}$}
\put(-1,0){$1$}
\put(1,0){$1$}
\put(0,-1){$1$}
\put(0,1){$1$}
\end{pspicture}
}

\def\dualBEA{%
\begin{pspicture}(-4,-4)(4,4)%
\psgrid(0,0)(-3,-3)(3,3)%
\psline(1,-2)(3,-3)(3,-1)(-1,3)(-3,3)(-2,1)(1,-2)
\psline[linecolor=gray](-1,3)(-1,0)
\put(-3,3){$0$}
\put(-1,3){$0$}
\put(-3,1){$0$}
\put(-1,0){$0$}
\put(3,-3){$-2$}
\put(3,-1){$-2$}
\put(1,-3){$-1$}
\end{pspicture}\begin{pspicture}(-4,-4)(4,4)%
\psgrid(0,0)(-3,-3)(3,3)%
\psline(1,-2)(3,-3)(3,-1)(-1,3)(-3,3)(-2,1)(1,-2)
\psline[linecolor=gray](3,-1)(0,-1)
\put(3,-3){$0$}
\put(3,-1){$0$}
\put(1,-3){$0$}
\put(0,-1){$0$}
\put(-3,3){$-2$}
\put(-1,3){$-2$}
\put(-3,1){$-1$}
\end{pspicture}
\begin{pspicture}(-4,-4)(4,4)%
\psgrid(0,0)(-3,-3)(3,3)%
\psline(1,-2)(3,-3)(3,-1)(-1,3)(-3,3)(-2,1)(1,-2)
\psline[linecolor=gray](-3,3)(3,-3)
\put(3,-3){$2$}
\put(3,-1){$2$}
\put(1,-3){$1$}
\put(-3,3){$2$}
\put(-1,3){$2$}
\put(-3,1){$1$}
\end{pspicture}
\begin{pspicture}(-4,-4)(4,4)%
\psgrid(0,0)(-3,-3)(3,3)%
\psline(1,-2)(3,-3)(3,-1)(-1,3)(-3,3)(-2,1)(1,-2)
\psline[linecolor=gray](-3,3)(3,-3)
\psline[linecolor=gray](-1,3)(-1,0)
\psline[linecolor=gray](3,-1)(0,-1)
\psdot(0,0)
\psdot[dotstyle=x,dotsize=6px](0.875,0.875)
\put(3,-3){$0$}
\put(3,-1){$0$}
\put(1,-3){$0$}
\put(-3,3){$0$}
\put(-1,3){$0$}
\put(-3,1){$0$}
\put(-0.7,-2){$\nicefrac{1}{2}$}
\put(-2.2,-0.5){$\nicefrac{1}{2}$}
\put(-1,1){$1$}
\put(1,-1){$1$}
\end{pspicture}
}

\def\dualBF {
\begin{pspicture}(-4,-4)(4,4)%
\psgrid(0,0)(-3,-3)(3,3)%
\psline(-1,1)(1,1)(3,0)(3,-1)(-3,-1)(-3,0)(-1,1)
\psline[linecolor=gray](0,1)(0,-1)

\put(0,1){$0$}
\put(-1,1){$0$}
\put(-3,0){$0$}
\put(-3,-1){$0$}
\put(0,-1){$0$}
\put(1,1){$-1$}
\put(3,0){$-3$}
\put(3,-1){$-3$}
\end{pspicture}
\begin{pspicture}(-4,-4)(4,4)%
\psgrid(0,0)(-3,-3)(3,3)%
\psline(-1,1)(1,1)(3,0)(3,-1)(-3,-1)(-3,0)(-1,1)
\psline[linecolor=gray](3,0)(-3,0)
\put(-1,1){$-1$}
\put(-3,0){$0$}
\put(-3,-1){$0$}
\put(1,1){$-1$}
\put(3,0){$0$}
\put(3,-1){$0$}
\end{pspicture}
\begin{pspicture}(-4,-4)(4,4)%
\psgrid(0,0)(-3,-3)(3,3)%
\psline(-1,1)(1,1)(3,0)(3,-1)(-3,-1)(-3,0)(-1,1)
\put(-1,1){$1$}
\put(-3,0){$0$}
\put(-3,-1){$0$}
\put(1,1){$2$}
\put(3,0){$3$}
\put(3,-1){$3$}\end{pspicture}
\begin{pspicture}(-4,-4)(4,4)%
\psgrid(0,0)(-3,-3)(3,3)%
\psline(-1,1)(1,1)(3,0)(3,-1)(-3,-1)(-3,0)(-1,1)
\psline[linecolor=gray](3,0)(-3,0)
\psline[linecolor=gray](0,1)(0,-1)
\psdot[dotstyle=x,dotsize=6px](0,-0.66)
\psdot(0,0)
\put(0,1){$\nicefrac{1}{2}$}
\put(-1,1){$0$}
\put(-3,0){$0$}
\put(-3,-1){$0$}
\put(0,-1){$\nicefrac{3}{2}$}
\put(0,0){$\nicefrac{3}{2}$}
\put(1,1){$0$}
\put(3,0){$0$}
\put(3,-1){$0$}
\end{pspicture}
}

\def\dualBFlabeled {
\begin{pspicture}(-4,-4)(4,4)%
\put(3,3){$\Psi_0$}
\psgrid(0,0)(-3,-3)(3,3)%
\psline(-1,1)(1,1)(3,0)(3,-1)(-3,-1)(-3,0)(-1,1)
\psline[linecolor=gray](0,1)(0,-1)
\put(0,1){$0$}
\put(-1,1){$0$}
\put(-3,0){$0$}
\put(-3,-1){$0$}
\put(0,-1){$0$}
\put(1,1){$-1$}
\put(3,0){$-3$}
\put(3,-1){$-3$}
\end{pspicture}
\begin{pspicture}(-4,-4)(4,4)%
\put(3,3){$\Psi_\infty$}
\psgrid(0,0)(-3,-3)(3,3)%
\psline(-1,1)(1,1)(3,0)(3,-1)(-3,-1)(-3,0)(-1,1)
\psline[linecolor=gray](3,0)(-3,0)
\put(-1,1){$-1$}
\put(-3,0){$0$}
\put(-3,-1){$0$}
\put(1,1){$-1$}
\put(3,0){$0$}
\put(3,-1){$0$}
\end{pspicture}
\begin{pspicture}(-4,-4)(4,4)%
\put(3,3){$\Psi_1$}
\psgrid(0,0)(-3,-3)(3,3)%
\psline[linecolor=gray](-1,1)(1,1)(3,0)(3,-1)(-3,-1)(-3,0)(-1,1)
\put(-1,1){$1$}
\put(-3,0){$0$}
\put(-3,-1){$0$}
\put(1,1){$2$}
\put(3,0){$3$}
\put(3,-1){$3$}\end{pspicture}
\begin{pspicture}(-4,-4)(4,4)%
\put(3,3){$\deg \Psi$}
\psgrid(0,0)(-3,-3)(3,3)%
\psline(-1,1)(1,1)(3,0)(3,-1)(-3,-1)(-3,0)(-1,1)
\psline[linecolor=gray](3,0)(-3,0)
\psline[linecolor=gray](0,1)(0,-1)
\psdot[dotstyle=x,dotsize=6px](0,-0.66)
\psdot(0,0)
\put(0,1){$\nicefrac{1}{2}$}
\put(-1,1){$0$}
\put(-3,0){$0$}
\put(-3,-1){$0$}
\put(0,-1){$\nicefrac{3}{2}$}
\put(0,0){$\nicefrac{3}{2}$}
\put(1,1){$0$}
\put(3,0){$0$}
\put(3,-1){$0$}
\end{pspicture}
}

\def\dualBG{%
\begin{pspicture}(-4,-4)(4,4)%
\psgrid(0,0)(-3,-3)(3,3)%
\psline(2,-1)(0,-3)(-3,0)(-2,1)(1,1)(2,0)(2,-1)
\psline[linecolor=gray](0,1)(0,-3)
\put(-3,0){$0$}
\put(-2,1){$0$}
\put(2,0){$-2$}
\put(1,1){$-1$}
\put(0,1){$0$}
\put(2,-1){$-2$}
\put(0,-3){$0$}
\end{pspicture}
\begin{pspicture}(-4,-4)(4,4)%
\psgrid(0,0)(-3,-3)(3,3)%
\psline(2,-1)(0,-3)(-3,0)(-2,1)(1,1)(2,0)(2,-1)
\psline[linecolor=gray](2,0)(-3,0)
\put(-3,0){$0$}
\put(-2,1){$0$}
\put(2,0){$0$}
\put(1,1){$0$}
\put(2,-1){$-1$}
\put(0,-3){$-3$}
\end{pspicture}
\begin{pspicture}(-4,-4)(4,4)%
\psgrid(0,0)(-3,-3)(3,3)%
\psline(2,-1)(0,-3)(-3,0)(-2,1)(1,1)(2,0)(2,-1)
\psline[linecolor=gray](0,1)(2,-1)
\put(-3,0){$0$}
\put(-2,1){$0$}
\put(2,0){$2$}
\put(1,1){$1$}
\put(0,1){$1$}
\put(2,-1){$3$}
\put(0,-3){$3$}
\end{pspicture}
\begin{pspicture}(-4,-4)(4,4)%
\psgrid(0,0)(-3,-3)(3,3)%
\psline(2,-1)(0,-3)(-3,0)(-2,1)(1,1)(2,0)(2,-1)
\psline[linecolor=gray](2,-1)(0,1)
\psline[linecolor=gray](2,0)(-3,0)
\psline[linecolor=gray](0,1)(0,-3)
\psdot[dotstyle=x,dotsize=6px](-1.08,-2.21)
\psdot(0,0)
\put(-3,0){$0$}
\put(-2,1){$0$}
\put(2,0){$0$}
\put(1,1){$0$}
\put(0,1){$1$}
\put(2,-1){$0$}
\put(0,-3){$0$}
\put(-0.7,-0.5){$\nicefrac{3}{2}$}
\put(0.7,-0.5){$1$}
\end{pspicture}
}

\def\dualBH {
\begin{pspicture}(-4,-4)(4,4)%
\psgrid(0,0)(-3,-3)(3,3)%
\psline(-2,1)(1,1)(2,0)(2,-2)(0,-2)(-2,0)(-2,1)
\psline[linecolor=gray](0,1)(0,-2)
\put(-2,0){$-2$}
\put(-2,1){$-2$}
\put(1,1){$0$}
\put(0,1){$0$}
\put(2,0){$0$}
\put(2,-2){$0$}
\put(0,-2){$0$}
\end{pspicture}
\begin{pspicture}(-4,-4)(4,4)%
\psgrid(0,0)(-3,-3)(3,3)%
\psline(-2,1)(1,1)(2,0)(2,-2)(0,-2)(-2,0)(-2,1)
\psline[linecolor=gray](2,0)(-2,0)
\put(-2,0){$0$}
\put(-2,1){$0$}
\put(1,1){$0$}
\put(2,0){$0$}
\put(2,-2){$-2$}
\put(0,-2){$-2$}
\end{pspicture}
\begin{pspicture}(-4,-4)(4,4)%
\psgrid(0,0)(-3,-3)(3,3)%
\psline(-2,1)(1,1)(2,0)(2,-2)(0,-2)(-2,0)(-2,1)
\psline[linecolor=gray](-1,1)(2,-2)
\put(-2,0){$2$}
\put(-2,1){$2$}
\put(1,1){$0$}
\put(-1,1){$2$}
\put(2,0){$0$}
\put(2,-2){$2$}
\put(0,-2){$2$}
\end{pspicture}
\begin{pspicture}(-4,-4)(4,4)%
\psgrid(0,0)(-3,-3)(3,3)%
\psline(-2,1)(1,1)(2,0)(2,-2)(0,-2)(-2,0)(-2,1)
\psline[linecolor=gray](-1,1)(2,-2)
\psline[linecolor=gray](2,0)(-2,0)
\psline[linecolor=gray](0,1)(0,-2)
\psdot[dotstyle=x,dotsize=6px](0.66,-1.33)
\psdot(0,0)
\put(-2,0){$0$}
\put(-2,1){$0$}
\put(0,1){$1$}
\put(1,1){$0$}
\put(-1,1){$1$}
\put(2,0){$0$}
\put(2,-2){$0$}
\put(0,-2){$0$}
\put(0,0){$2$}
\end{pspicture}
}

\def\dualCA{%
\begin{pspicture}(-4,-4)(4,4)%
\psgrid(0,0)(-3,-3)(3,3)%
\psline[linecolor=gray](1,1)(1,-1)(-1,-1)(-1,1)
\psline(3,0)(2,-1)(-2,-1)(-3,0)(-2,1)(2,1)(3,0)
\put(-3,0){$0$}
\put(-2,1){$0$}
\put(3,0){$-3$}
\put(2,1){$-2$}
\put(-2,-1){$0$}
\put(2,-1){$-2$}
\put(1,-1){$-1$}
\put(1,1){$-1$}
\put(-1,-1){$0$}
\put(-1,1){$0$}
\end{pspicture}
\begin{pspicture}(-4,-4)(4,4)%
\psgrid(0,0)(-3,-3)(3,3)%
\psline(3,0)(2,-1)(-2,-1)(-3,0)(-2,1)(2,1)(3,0)
\psline[linecolor=gray](3,0)(-3,0)
\put(-3,0){$0$}
\put(-2,1){$0$}
\put(3,0){$0$}
\put(2,1){$0$}
\put(-2,-1){$-1$}
\put(2,-1){$-1$}
\put(0,-1){$-1$}
\put(0,1){$0$}
\end{pspicture}
\begin{pspicture}(-4,-4)(4,4)%
\psgrid(0,0)(-3,-3)(3,3)%
\psline(3,0)(2,-1)(-2,-1)(-3,0)(-2,1)(2,1)(3,0)
\put(-3,0){$0$}
\put(-2,1){$0$}
\put(3,0){$3$}
\put(2,1){$2$}
\put(-2,-1){$1$}
\put(2,-1){$3$}
\end{pspicture}
\begin{pspicture}(-4,-4)(4,4)%
\psgrid(0,0)(-3,-3)(3,3)%
\psline[linecolor=gray](1,1)(1,-1)(-1,-1)(-1,1)
\psline(3,0)(2,-1)(-2,-1)(-3,0)(-2,1)(2,1)(3,0)
\psline[linecolor=gray](3,0)(-3,0)
\psdot(0,0)
\put(-3,0){$0$}
\put(-2,1){$0$}
\put(3,0){$0$}
\put(2,1){$0$}
\put(-2,-1){$0$}
\put(2,-1){$0$}
\put(1,0){$1$}
\put(-1,0){$1$}
\put(-1,1){$\nicefrac{1}{2}$}
\put(-1,-1){$\nicefrac{1}{2}$}
\put(1,1){$\nicefrac{1}{2}$}
\put(1,-1){$\nicefrac{1}{2}$}
\end{pspicture}
}

\def\dualCAA{%
\begin{pspicture}(-4,-4)(4,4)%
\psgrid(0,0)(-3,-3)(3,3)%
\psline[linecolor=gray](-1,2)(-1,0)(1,-2)(3,-3)(3,-2)(-1,2)
\psline(1,-2)(3,-3)(3,-2)(-2,3)(-3,3)(-2,1)(1,-2)
\put(-3,3){$0$}
\put(-2,3){$0$}
\put(-3,1){$0$}
\put(-1,0){$0$}
\put(3,-3){$-2$}
\put(3,-2){$-2$}
\put(1,-3){$-1$}
\put(-1,2){$0$}
\end{pspicture}\begin{pspicture}(-4,-4)(4,4)%
\psgrid(0,0)(-3,-3)(3,3)%
\psline[linecolor=gray](2,-1)(0,-1)(-2,1)(-3,3)(-2,3)(2,-1)
\psline(1,-2)(3,-3)(3,-2)(-2,3)(-3,3)(-2,1)(1,-2)
\put(3,-3){$0$}
\put(3,-2){$0$}
\put(1,-3){$0$}
\put(0,-1){$0$}
\put(-3.7,3){$-2$}
\put(-2.4,3){$-2$}
\put(-3,1){$-1$}
\put(2,-1){$0$}
\end{pspicture}
\begin{pspicture}(-4,-4)(4,4)%
\psgrid(0,0)(-3,-3)(3,3)%
\psline(1,-2)(3,-3)(3,-2)(-2,3)(-3,3)(-2,1)(1,-2)
\psline[linecolor=gray](-3,3)(3,-3)
\put(3,-3){$2$}
\put(3,-2){$2$}
\put(1,-3){$1$}
\put(-3,3){$2$}
\put(-2,3){$2$}
\put(-3,1){$1$}
\end{pspicture}
\begin{pspicture}(-4,-4)(4,4)%
\psgrid(0,0)(-3,-3)(3,3)%
\psline(1,-2)(3,-3)(3,-2)(-2,3)(-3,3)(-2,1)(1,-2)
\psline[linecolor=gray](-3,3)(3,-3)
\psline[linecolor=gray](-1,2)(-1,0)
\psline[linecolor=gray](2,-1)(0,-1)
\psdot(0,0)
\psdot[dotstyle=x,dotsize=6px](0.21,0.21)

\put(-1,2){$\nicefrac{1}{2}$}
\put(2,-1){$\nicefrac{1}{2}$}
\put(3,-3){$0$}
\put(3,-2){$0$}
\put(1,-3){$0$}
\put(-3,3){$0$}
\put(-2,3){$0$}
\put(-3,1){$0$}
\put(-0.7,-2){$\nicefrac{1}{2}$}
\put(-2.2,-0.5){$\nicefrac{1}{2}$}
\put(-1,1){$1$}
\put(1,-1){$1$}
\end{pspicture}
}

\def\dualCB {
\begin{pspicture}(-4,-4)(4,4)%
\psgrid(0,0)(-3,-3)(3,3)%
\psline(-2,1)(1,1)(2,0)(2,-1)(-1,-1)(-2,0)(-2,1)
\psline[linecolor=gray](0,1)(0,-1)
\put(-2,0){$-2$}
\put(-2,1){$-2$}
\put(1,1){$0$}
\put(2,0){$0$}
\put(2,-1){$0$}
\put(-1,-1){$-1$}
\put(0,-1){$0$}
\put(0,1){$0$}
\end{pspicture}
\begin{pspicture}(-4,-4)(4,4)%
\psgrid(0,0)(-3,-3)(3,3)%
\psline(-2,1)(1,1)(2,0)(2,-1)(-1,-1)(-2,0)(-2,1)
\psline[linecolor=gray](2,0)(-2,0)
\put(-2,0){$0$}
\put(-2,1){$0$}
\put(1,1){$0$}
\put(2,0){$0$}
\put(2,-1){$-1$}
\put(-1,-1){$-1$}
\end{pspicture}
\begin{pspicture}(-4,-4)(4,4)%
\psgrid(0,0)(-3,-3)(3,3)%
\psline(-2,1)(1,1)(2,0)(2,-1)(-1,-1)(-2,0)(-2,1)
\psline[linecolor=gray](-1,1)(1,-1)
\put(-2,0){$2$}
\put(-2,1){$2$}
\put(1,1){$0$}
\put(2,0){$0$}
\put(2,-1){$0$}
\put(-1,-1){$2$}
\put(1,-1){$2$}
\put(-1,1){$2$}
\end{pspicture}
\begin{pspicture}(-4,-4)(4,4)%
\psgrid(0,0)(-3,-3)(3,3)%
\psline(-2,1)(1,1)(2,0)(2,-1)(-1,-1)(-2,0)(-2,1)
\psline[linecolor=gray](-1,1)(1,-1)
\psline[linecolor=gray](2,0)(-2,0)
\psline[linecolor=gray](0,1)(0,-1)
\psdot(0,0)
\put(-2,0){$0$}
\put(-2,1){$0$}
\put(0,1){$1$}
\put(1,1){$0$}
\put(-1,1){$1$}
\put(2,0){$0$}
\put(1,-1){$1$}
\put(0,-1){$1$}
\put(2,-1){$0$}
\put(-1,-1){$0$}
\put(0,0){$2$}
\end{pspicture}
}

\def\dualCC {
\begin{pspicture}(-4,-4)(4,4)%
\psgrid(0,0)(-3,-3)(3,3)%
\psline(-1,1)(1,1)(2,0)(2,-1)(-2,-1)(-2,0)(-1,1)
\psline[linecolor=gray](0,1)(0,-1)
\put(-2,0){$-2$}
\put(-2,-1){$-2$}
\put(-1.4,1){$-1$}
\put(0,1){$0$}
\put(0,-1){$0$}
\put(2,0){$0$}
\put(2,-1){$0$}
\put(1,1){$0$}
\end{pspicture}
\begin{pspicture}(-4,-4)(4,4)%
\psgrid(0,0)(-3,-3)(3,3)%
\psline(-1,1)(1,1)(2,0)(2,-1)(-2,-1)(-2,0)(-1,1)
\psline[linecolor=gray](2,0)(-2,0)
\put(-2,0){$2$}
\put(-2,-1){$2$}
\put(-1,1){$1$}
\put(2,0){$2$}
\put(2,-1){$2$}
\put(1,1){$1$}
\end{pspicture}
\begin{pspicture}(-4,-4)(4,4)%
\psgrid(0,0)(-3,-3)(3,3)%
\psline(-1,1)(1,1)(2,0)(2,-1)(-2,-1)(-2,0)(-1,1)
\psline[linecolor=gray](0,1)(0,-1)
\put(2,0){$-2$}
\put(2,-1){$-2$}
\put(1,1){$-1$}
\put(0,1){$0$}
\put(0,-1){$0$}
\put(-2,0){$0$}
\put(-2,-1){$0$}
\put(-1,1){$0$}
\end{pspicture}
\begin{pspicture}(-4,-4)(4,4)%
\psgrid(0,0)(-3,-3)(3,3)%
\psline(-1,1)(1,1)(2,0)(2,-1)(-2,-1)(-2,0)(-1,1)
\psline[linecolor=gray](2,0)(-2,0)
\psline[linecolor=gray](0,1)(0,-1)
\psdot[dotstyle=x,dotsize=6px](0,-1.08)
\psdot(0,0)
\put(-2,0){$0$}
\put(-2,-1){$0$}
\put(-1,1){$0$}
\put(2,0){$0$}
\put(2,-1){$0$}
\put(1,1){$0$}
\put(0,0){$2$}
\put(0,1){$1$}
\put(0,-1){$2$}
\end{pspicture}
}

\newtheorem{theorem}{Theorem}[section]
\newtheorem{lemma}[theorem]{Lemma}
\newtheorem{proposition}[theorem]{Proposition}
\newtheorem{corollary}[theorem]{Corollary}
{\theoremstyle{remark}
\newtheorem{remark}[theorem]{Remark}

}
\theoremstyle{definition}

\newtheorem{definition}[theorem]{Definition}
\newtheorem{example}[theorem]{Example}

\newcommand{\cI}{\mathcal{I}}
\newcommand{\cJ}{\mathcal{J}}

\newcommand{\CC}{\mathbb{C}}
\newcommand{\RR}{\mathbb{R}}
\newcommand{\QQ}{\mathbb{Q}}
\newcommand{\ZZ}{\mathbb{Z}}

\newcommand{\PP}{\mathbb{P}}

\newcommand{\fan}{\mathcal{S}}
\renewcommand{\Vert}{\fan^{(0)}}
\newcommand{\Hor}{(\tail \fan)^\times}
\newcommand{\CO}{{\mathcal{O}}}

\newcommand{\sdeg}{\deg \fan}

\newcommand{\walls}[1]{{\Delta^\partial}}

\DeclareMathOperator{\id}{id}
\newcommand{\rvol}{\vol_1}
\DeclareMathOperator{\bc}{bc}
\newcommand{\rbc}{\bc_1}
\DeclareMathOperator{\Cox}{Cox}
\DeclareMathOperator{\Aut}{Aut}

\DeclareMathOperator{\lin}{lin}

\DeclareMathOperator{\codim}{codim}

\DeclareMathOperator{\tail}{tail}

\DeclareMathOperator{\wdiv}{Div}

\DeclareMathOperator{\proj}{Proj}

\DeclareMathOperator{\conv}{conv}
\DeclareMathOperator{\pyr}{pyr}
\DeclareMathOperator{\pos}{pos}
\DeclareMathOperator{\vol}{vol}
\DeclareMathOperator{\cl}{Cl}

\newlength{\tpheight}
\newlength{\txtheight}
\newlength{\tpwidth}
\newlength{\txtwidth}

\title[Fano 3folds picture book]{Fano threefolds with 2-torus action\\
 \footnotesize  \rm A picture book}
\subjclass[2010]{14L30 (Primary) 14J45, 32Q20 (Secondary)}
\keywords{T-variety, torus action, moment polytope, Fano variety, Kahler-Einstein metric, Cox ring}
\thanks{The author was supported by a fellowship of the Alexander von Humboldt Foundation.}

\author[H. S\"u{\ss}]{Hendrik S\"u\ss}
\address{Hendrik S\"u\ss\\ School of Mathematics
The University of Edinburgh
James Clerk Maxwell Building
The King's Buildings
Mayfield Road
Edinburgh EH9 3JZ}
\email{\href{mailto:suess@math.tu-cottbus.de}{suess@sdf-eu.org}}

\begin{document}
\thispagestyle{empty}
\maketitle
\vspace{-1.2em}
\begin{center}
  {\psset{unit=0.13cm}\titlepic}
\end{center}
  
\begin{abstract}
Following the work of Altmann and Hausen we give a combinatorial description for smooth Fano threefolds admitting a 2-torus action. We show that a whole variety of properties and invariants can be read off from this description. As an application we prove and disprove the existence of K\"ahler-Einstein metrics for some of these Fano threefolds, calculate their Cox rings and some of their toric canonical degenerations.
\end{abstract}

\section{Introduction}
\label{sec:introduction}
Knowing the fan or the polytope of a projective toric variety does not only give a concrete construction of the variety but also enables us to explore its properties and to easily calculate a whole bunch of invariants. In low dimensions a close look on the corresponding pictures is enough to obtain detailed information about the variety. Loosely speaking we can even prove statements by drawing pictures.

 In \cite{pre05013675} and \cite{divfans} the description of toric varieties by cones and fans was generalized to the case of an arbitrary $T$-variety, i.e. a variety $X$ endowed with the effective action of an algebraic torus $T$.  This general description consists of combinatorial data living on the Chow quotient $X/\!\!/T$. A series of papers by different authors followed this approach to derive properties and calculate invariants out of the combinatorial data. This works especially well for the case of complexity-one torus action, i.e. when the torus has one dimension less then the variety and the Chow quotient is just a curve.

The aim of this paper is to make these results available for Fano threefolds with an action of a two-dimensional torus. More precisely,  we are going to state the needed combinatorial description for a list of those Fano threefolds and relate it to the classification of Mori-Mukai \cite{mm86}. By the mentioned preliminary results we can immediately derive a lot of additional information about these Fano varieties. 

We want to highlight two concrete applications. First, by \cite{cox} we are able to calculate the Cox rings of these T-varieties. The knowledge of the Cox ring may give further inside in the geometry of the variety. For example the automorphism group can be studied by the methods from \cite{arzhantsev2012automorphism}. Since, these Cox ring are complete intersections we can see the corresponding Fano varieties as complete intersections inside certain toric varieties. This was observed also in \cite{coatesFanoCI}.

As a second application we consider the question of the existence of K\"ahler-Einstein metrics on Fano manifolds.  In general this is a hard question. There are some known obstructions such as the vanishing of the Futaki character \cite{0506.53030}. There is also a sufficient criterion due to Tian \cite{0599.53046}. 
Again for toric Fano varieties the K\"ahler-Einstein property can be easily reformulated in terms of the defining polytope. By a result of Wang and Zhu a toric Fano variety is K\"ahler-Einstein if and only if the Futaki character vanishes \cite{1086.53067}. Moreover, the Futaki character can be easily calculated as the barycenter of the polytope corresponding to the toric Fano manifold \cite{0661.53032}. We are going to generalize the last result for our situation of complexity-one torus actions in order to show that the Futaki character does not vanish for some of the threefolds, and hence disprove the existence of K\"ahler-Einstein metrics for them. We are using an application of Tian's criterion from \cite{kesym} to prove the existence of K\"ahler-Einstein metrics for three of the remaining ones.

Our main conclusions drawn from the given combinatorial description are summarized in the following theorem.
\begin{theorem}
\label{thm:main}
    The Fano threefolds $Q$, 2.24\footnote{Only one variety of the family admits a 2-torus action}, 2.29--2.32, 3.10\footnotemark[\value{footnote}], 3.18--3.24, 4.4, 4.5, 4.7 and 4.8 from Mori's and Mukai's classification admit a 2-torus action. 

The moment polytopes together with their Duistermaat-Heckman measures are given in Section~\ref{sec:pictures}. The Cox rings, Futaki characters $F(X)$  as well as the existence of a K\"ahler-Einstein metric can be found in the following table. 
\begin{center}
\begin{tabular}[htbp]{llcc}
\toprule
No.&Cox ring&F(X)&K\"ahler-Einstein\\
\midrule
$Q$&$k[\underline{T}]/(T_1T_2+T_3T_4+T_5^2)$&$0$&yes\\
$2.24\footnotemark[\value{footnote}]$&$k[\underline{T}]/(T_1T_2^2+T_3T_4^2+T_5T_6^2)$&$0$&yes\\
$2.29$&$k[\underline{T}]/(T_1T_2^2T_3+T_4T_5+T_6^2)$&$0$&?\\
$2.30$&$k[\underline{T},S_1]/(T_1T_2+T_3T_4+T_5^2)$&$\binom{0}{-2}$&no\\
$2.31$&$k[\underline{T}]/(T_1T_2+T_3T_4+T_5T_6^2)$&$\binom{-\nicefrac{4}{3}}{-\nicefrac{4}{3}}$&no\\
$2.32$&$k[\underline{T}]/(T_1T_2+T_3T_4+T_5T_6)$&$0$&yes\\
\midrule
$3.10\footnotemark[\value{footnote}]$&$k[\underline{T}]/(T_1T_2^2T_3+T_4T_5^2T_6+T_7^2)$&$0$&yes\\
$3.18$&$k[\underline{T},S_1]/(T_1T_2^2T_3+T_4T_5+T_6^2)$&$\binom{0}{-\nicefrac{7}{8}}$&no\\
$3.19$&$k[\underline{T},S_1,S_2]/(T_1T_2+T_3T_4+T_5^2)$&$0$&?\\
$3.20$&$k[\underline{T}]/(T_1T_2+T_3T_4+T_5T_6^2T_7)$&$0$&?\\
$3.21$&$k[\underline{T},S_1]/(T_1T_2^2+T_3T_4^2+T_5T_6)$&$\binom{\nicefrac{7}{8}}{\nicefrac{7}{8}}$&no\\
$3.22$&$k[\underline{T},S_1,S_2]/(T_1T_2+T_3T_4+T_5^2)$&$\binom{0}{-\nicefrac{2}{3}}$&no\\
$3.23$&$k[\underline{T},S_1]/(T_1T_2+T_3T_4+T_5T_6^2)$&$\binom{-\nicefrac{13}{12}}{-\nicefrac{53}{24}}$&no\\
$3.24$&$k[\underline{T},S_1]/(T_1T_2+T_3T_4+T_5T_6)$&$\binom{-\nicefrac{2}{3}}{-\nicefrac{4}{3}}$&no\\
\midrule
$4.4$&$k[\underline{T},S_1,S_2]/(T_1T_2^2T_3+T_4T_5+T_6^2)$&$0$&?\\
$4.5$&$k[\underline{T},S_1,S_2]/(T_1T_2^2+ T_3T_4^2 + T_5T_6)$&$\binom{\nicefrac{5}{24}}{\nicefrac{5}{24}}$&no\\
$4.7$&$k[\underline{T},S_1,S_2]/(T_1T_2+T_3T_4+T_5T_6)$&$0$&?\\
$4.8$&$k[\underline{T},S_1,S_2]/(T_1T_2+T_3T_4+T_5T_6)$&$\binom{0}{-\nicefrac{13}{12}}$&?\\
\bottomrule
\end{tabular}
\end{center}
\end{theorem}

\vspace{0.5em}
The paper is organized as follows. In Section~\ref{sec:torus-actions} we recall the combinatorial description of complexity-one torus actions from \cite{tvars}. In Section~\ref{sec:torus-invariant-divisors} we describe equivariant polarizations and give a characterization of the Fano property.  Moreover, we recall how to calculate the Cox ring from the combinatorial description and we state a simple formula for the Futaki character. In Section~\ref{sec:degenerations} we show how to find toric degenerations for our T-varieties. Finally in Section~\ref{sec:pictures} we state the list with the combinatorial descriptions for all the considered Fano varieties, which in turn leads to the proof of Theorem~\ref{thm:main}. The appendix provides the proof for the formula for the Futaki character.

\subsection*{Correction} In the published version of this paper in Theorem~\ref{thm:main} it was wrongly claimed that there is also an element with 2-torus action in the family 3.8. However, it was overlooked that the corresponding element is not smooth. For comparison see also  \cite{2018arXiv180909223C}.

\subsection*{Acknowledgments} I would like to thank Ivan Arzhantsev, Nathan Ilten and Johan Martens for important comments and conversations and Ivan Cheltsov, Victor Przyjalkowski and Constantin Shramov for pointing out an error in the published version of this paper.
 
\goodbreak
\section{Combinatorial description of torus actions of complexity one}
\label{sec:torus-actions}
\subsection*{F-divisors}
We consider a variety $X/_{\CC}$ with an effective action of an algebraic torus $T \cong (\CC^*)^r$. Then $X$ is called a $T$-variety of complexity $(\dim X - \dim T)$. In the following we restrict ourselves to the case of rational $T$-varieties of complexity one. Such objects where described and studied earlier in \cite{0271.14017,1151.14037}. We are using the more general approach of \cite{pre05013675} which describes $T$-varieties of arbitrary complexity, but restrict it to the case of rational complexity-one $T$-varieties. Our general references are \cite{pre05013675,divfans,tvars}.

We follow the standard terminology for toric varieties. The character lattice of the torus $T$ is denoted by $M$ and we set $N:=M^*$. For the associated vector spaces we write $M_\QQ$ and $N_\QQ$ or $M_\RR$ and $N_\RR$, respectively. 
 
For a polyhedron $\Delta \subset N_\QQ$ we consider its tail (or recession) cone $\tail \Delta$ defined by $\tail \Delta := \{v \in N_\QQ \mid \Delta + \QQ_{\geq 0} \cdot v = \Delta \}$. For a complete polyhedral subdivision $\Xi$ of $N_\QQ$ the set of tail cones has the structure of a fan, which is called the tail fan of $\Xi$ and will be denoted by $\tail \Xi$. Now, consider a pair \[\fan = \left(\sum_{P\in \PP^1} \fan_P \otimes P,\quad \sdeg\right)\] Here,  $\sum_P \fan_P \otimes P$ is just a formal sum and $\fan_P$ are complete polyhedral subdivisions of $N_\QQ$ with some common tail fan denoted by $\tail \fan$. For the moment the degree $\sdeg$ is just a subset of $N_\QQ$. The subdivisions  $\fan_P$ are called \emph{slices} of $\fan$. We assume that there are only finitely many \emph{non-trivial} slices, i.e. those which differ from the tail fan. Note, that for every full-dimensional $\sigma\in\tail(\fan)$ there is a unique polyhedron $\Delta_P^\sigma$ in $\fan_P$ with $\tail(\Delta_P^\sigma)=\sigma$. 
\begin{definition}
\label{def:f-divisor}
  A pair as above is called a (complete) \emph{f-divisor} if for any full-dimensional $\sigma\in\tail(\fan)$ we have either $\sdeg\cap\sigma = \emptyset$ or \[\sum_P \Delta^\sigma_P = \sdeg\cap\sigma \subsetneq \sigma.\]
\end{definition}

A (complete) f-divisor $\fan$ as above corresponds to a rational complete $T$-variety $X(\fan)$ of complexity one, see \cite[Section~5.3]{tvars}. It comes with a rational quotient map $\pi:X(\fan) \dashrightarrow \PP^1$.

\begin{example}[Fano threefold 3.10]  
\label{ex:threefold3.10-1}
We consider the f-divisor $\fan$ with three non-trivial slices and the degree given in Figure~\ref{fig:f-div-310}. It is straight forward to check, that the condition of Definition~\ref{def:f-divisor} is indeed fulfilled.
 \begin{figure}[htpb]
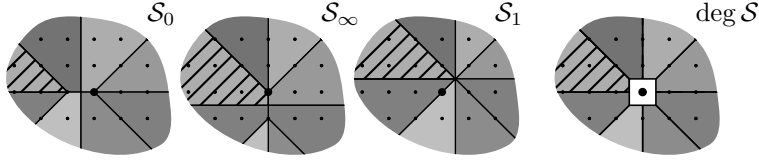

    \centering
    \threefoldBBHL
    \caption{f-divisor of the Fano threefold 3.10.}
    \label{fig:f-div-310}
  \end{figure}

As we will see later, this f-divisor corresponds to a blow up of the quadric threefold in two disjoint conics, i.e. it is an element of the family 3.10 from the classification of Mori and Mukai.
\end{example}

By \cite[Proposition~5.1, Theorem~5.3]{tsing} we have the following criterion for the smoothness of $X(\fan)$.

\begin{theorem}
  \label{thm:smoothness}
  The T-variety $X(\fan)$ is smooth if and only if for every  cone $\sigma \in \tail \fan$ one of the following two conditions is fulfilled 
  \begin{enumerate}
  \item if $\sigma \cap \sdeg=\emptyset$ then cone \(\QQ_{\geq 0} \cdot (\Delta \times \{1\} \,+\, \sigma \times \{0\})\) is regular, for every $P$ and every maximal polyhedron $\Delta \subset \fan_P$ with tail cone $\sigma$.
  \item if $\sigma$ is maximal and $\sigma \cap \sdeg \neq \emptyset$ then for at most two points $P=Q,R$ the polyhedron $\Delta^\sigma_P$ is not a lattice translate of $\sigma$ and the cone
\[\QQ_{\geq 0} \cdot \big(\sum_{P \neq Q}\Delta^\sigma_P  \times \{1\} \;\cup\; \sigma \times \{0\} \;\cup\; \Delta^\sigma_{Q} \times \{-1\}\big)\]
is regular.
  \end{enumerate}
\end{theorem}

\begin{example}[Fano threefold 3.10 (continued)]  
\label{ex:threefold3.10-1b}
We consider again the f-divisor from Example~\ref{ex:threefold3.10-1}. There is no  maximal polyhedron $\Delta$ with $\tail \Delta \cap \sdeg=\emptyset$. Look at the shaded polyhedra in Figure~\ref{fig:f-div-310}. They have tail cone  $\sigma = \QQ_{\geq 0} \cdot (-1,0) +  \QQ_{\geq 0} \cdot (-1,1)$. Only $\Delta_1^\sigma$ and $\Delta^\sigma_\infty$ are not a lattice translate of $\sigma$ and we obtain
 \begin{align*}
    & \; \QQ_{\geq 0} \cdot \big((\Delta_1^\sigma + {\scriptstyle (-1,0)}) \times \{1\} \,\cup\, \sigma \times \{0\} \,\cup\, \Delta^\sigma_\infty \times \{-1\}\big)\\
= & \; \QQ_{\geq 0} \cdot (-1,1,2) \,+\, \QQ_{\geq 0} \cdot (0,0,-1) \,+\, \QQ_{\geq 0} \cdot (0,-1,-2),
 \end{align*}
which is regular cone. Similarly for the other cones from the tail fan we get the same result.  Hence, by Theorem~\ref{thm:smoothness} we obtain the smoothness of $X(\fan)$.
\end{example}

\subsection*{Symmetries}
\label{sec:isom-symm}
In \cite[Section~5.3]{tvars} there is also a characterization of f-divisors giving rise to isomorphic T-varieties. We consider isomorphisms  $\varphi: \PP^1 \rightarrow \PP^1$ and $F:N\rightarrow N'$. Having this we define 
\[F(\varphi^*\fan) :=\left(\sum_P(F(\fan_{\varphi^{-1}(P)})) \otimes P,\quad F(\sdeg)\right).\]
\begin{definition}
  Two f-divisors $\fan$, $\fan'$ on $\PP^1$ are called equivalent if $\sdeg = \sdeg'$ and there are lattice points $v_P \in N$ with $\sum_P v_P=0$, such that $\fan_P + v_P = \fan'_P$. In this situation we write $\fan \sim \fan'$.
\end{definition}

\begin{proposition}
\label{prop:iso-global}
  $\fan$ and $\fan'$ lead to isomorphic T-varieties if and only if there are 
  isomorphisms  $\varphi: \PP^1 \rightarrow \PP^1$ and $F:N\rightarrow N'$, such that $F(\varphi^*\fan) \sim \fan'$.
\end{proposition}

By considering the case $\fan=\fan'$ we even obtain a description of the group of equivariant automorphisms $\Aut_T(X)$.  Therefore, we consider the set $\Aut(\fan)$ of pairs $(F,\varphi)$ as above, such that $F(\varphi^*\fan) \sim \fan$ holds. This set becomes a group via $(F,\varphi) \circ (F',\varphi'):=(F \circ F',\; \varphi' \circ \varphi).$
Note, that there is a natural action of $\Aut(\fan)$ on the characters $M$ and on $\PP^1$, by considering only the first or only the second component of the pairs, respectively.

Now, we have the following proposition, due to  \cite[Section~5.3]{tvars}.
\begin{proposition}
\label{prop:symmetries1}
  The group $\Aut(\fan)$ is isomorphic to $\Aut_T(X(\fan))/T$. 
\end{proposition}

Moreover, this isomorphism  preserves the natural action of $\Aut_T(X(\fan))/T$ on $M$ by conjugation. Hence, we will identify $\Aut_T(X(\fan))/T$ and $\Aut(\fan)$.

\begin{definition}
  A $T$-variety of complexity one is called \emph{symmetric} if  there is a finite subgroup $G \subset \Aut_T(X(\fan))$, with $M^G=\{0\}$.
\end{definition}

\begin{lemma}
\label{lem:finite-symmetries}
  Consider a complete rational variety $X$ with torus action of complexity $\leq 1$. Let $T$ be the maximal torus of the automorphisms. Then $\Aut_T(X)/T$ is finite and $\Aut_T(X)$ is reductive.
\end{lemma}
\begin{proof}
  If $X$ is a toric variety the statement is well known. Hence, we may consider the case of a complexity-one torus action. We consider the f-divisor $\fan$ with $X=X(\fan)$ with non-trivial slices $\fan_{P_1}, \ldots, \fan_{P_\ell}$. All other slices are assumed to be just lattice translates of the tail fan. Assume that $\ell=2$ then we may assume, that $P_1=0$ and $P_2=\infty$. Now, every element $\varphi \in \CC^* \subset \Aut(\PP^1)$ will give rise to a pair $(\id, \varphi) \in \Aut(\fan)$. Hence, $T$ was not a maximal torus and $X$ is just a toric variety with an restricted torus action. Further on we may assume that $\ell \geq 3$. 

Consider an element $(F,\varphi)$ of $\Aut(\fan)$. Since, $F$ has to be an automorphism of the tail fan, there are only finitely many choices for $F$. Since, 
$\ell \geq 3$ the automorphism $\varphi$ of $\PP^1$ is induced by a permutation of $P_1, \ldots, P_\ell$. Hence, there are only finitely many choices of for $\varphi$, as well.

Since, $X$ is complete $\Aut_T(X)$ is an algebraic group and we have a Levi decomposition $\Aut_T(X) = H \ltimes R_u$, where $H$ is reductive and $R_u$ is the unipotent radical. We may assume that $T \subset H$. Hence, we obtain $\Aut_T(X)/T \cong H/T  \times R_u$. By the finiteness of $\Aut_T(X)/T$ we conclude that $R_u$ is trivial and $\Aut_T(X)$ is reductive.
\end{proof} 

\begin{proposition}
\label{prop:symmetries2}
  Consider an f-divisor $\fan$ on $\PP^1$ with at least three non-trivial slices. Then $X = X(\fan)$ is symmetric if and only if $M^{\Aut(\fan)}=0$ holds.
\end{proposition}
\begin{proof}
  One direction is obvious. Now, assume that $M^{\Aut(S)} = 0$.
  By Lemma~\ref{lem:finite-symmetries} we know that $\Aut(\fan)=\Aut_T(X)/T$ is finite and $\Aut_T(X)$ is reductive. Now by the results of \cite{zbMATH03234443} we can lift $\Aut(\fan)$ to a finite subgroup $G \subset \Aut_T(X)$, where $G$ is not necessarily isomorphic to $\Aut(\fan)$. Now, we also have $M^{G}=0$.
\end{proof}

The notion of symmetry for toric varieties goes back to \cite{0939.32016} it was generalized in \cite{kesym} to give a sufficient criterion for the existence of a K\"ahler-Einstein metric on $X(\fan)$. For a vertex $v$ in $\fan_P$ let us denote by $\mu(v)$ the minimal natural number, such that $\mu(v)\cdot v$ is a lattice point. We call $\mu(v)$ the \emph{multiplicity} of $v$. For $P \in \PP^1$ we set $\mu(P)= \max\{\mu(v) \mid v \in \fan_{P}^{(0)}\}$, where $\fan_{P}^{(0)}$ denotes the set of vertices in $\fan_{P}$.
\begin{theorem}
\label{thm:kesym}
   Let $X=X(\fan)$ be a symmetric smooth Fano $T$-variety of complexity one, given by an f-divisor $\fan$. If one of the following conditions is fulfilled:
   \begin{enumerate}
   \item there are three points $P_1,P_2,P_3 \in \PP^1$ such that $\mu(P_i) > 1$ for $i=1,2,3$,\label{item:1}
  \item there are two points as in (\ref{item:1}) which are swapped by an element of $\Aut(\fan)$,
  \item $\Aut(\fan)$ acts fixed-point-free on $\PP^1$,
\end{enumerate}
then $X$ is K\"ahler-Einstein.
\end{theorem}
\begin{proof}
  This is only a reformulation of \cite[Theorem~1.1]{kesym}. The condition that $T$ acts on $\pi^{-1}(P)$  with disconnected stabilizers is replaced by our condition $\mu(P) > 1.$ But this is equivalent by \cite[Proposition 4.11]{cox}.
\end{proof}

Now, we make use of Proposition~\ref{prop:symmetries2} to check the symmetry of a variety $X(\fan)$. 

\begin{example}[{Fano threefold 3.10 (continued)}]
\label{sec:ex-symmetry}
  We reconsider the f-divisor and variety from Example~\ref{ex:threefold3.10-1}.
  By setting $v_0=(-1,0)$, $v_1=(-1,0)$, $v_\infty=(1,1)$ and $v_P=0$ for all other points $P \in \PP^1$ we see that $\fan \sim -\fan$. Hence, the pair $(-\id_N, \id_{\PP^1})$ is an element of $\Aut(\fan)$. Hence, $X(\fan)$ is symmetric, since $-\id_N$ has only $0$ as a fixed point. Moreover, in the slices $\fan_0, \fan_\infty, \fan_1$ there are vertices of multiplicity $2$. By Theorem~\ref{thm:kesym}~(\ref{item:1}) this shows, that this Fano threefolds is actually K\"ahler-Einstein.
\end{example}

\goodbreak

\section{Torus invariant divisors}
\label{sec:torus-invariant-divisors}
In this section we recall the results of \cite{tidiv} concerning Weil and Carier divisors on T-varieties, see also \cite{tvars,IS10}.
\subsection*{The class group}
On a T-variety $X(\fan)$ we find two different types of torus invariant prime divisors. Every vertex $v$ in $\fan_P$ correspond to a so called \emph{vertical prime divisors} $V_{v}:=V_{P.v}$ which projects to the point $P\in \PP^1$, via the quotient map $\pi$. We denote the set of all vertices in $\fan_P$ by $\Vert_P$ and the disjoint union $\coprod_P \Vert_P$ by $\Vert$. As before, for a vertex $v \in \Vert_P$ we denote by $\mu(v)$ the smallest positive integer $\mu$ such that $\mu \cdot v \in N$. \emph{Horizontal  prime divisors} $H_\rho$ project surjectively onto $\PP^1$ via the quotient map $\pi$. They correspond to the rays $\rho$ in $\tail \fan$ which do not intersect $\sdeg$. We denote the set of all such rays by $\Hor$.

\begin{theorem}[{\cite[Cor.~3.15]{tidiv}}]\label{thm:divclass}
  The divisor class group of $X(\fan)$ is isomorphic to
  \[\bigoplus_{\rho \in \Hor} \ZZ \cdot H_\rho \oplus \bigoplus_{v \in \Vert} \ZZ \cdot  V_{v}\] modulo the relations
\begin{eqnarray*}
 \sum_{v \in \Vert_P} \mu(v)V_v &=& \sum_{v \in \Vert_Q} \mu(v)V_v,\\
 0  &=& \sum_{\rho}  \langle u,\rho \rangle H_\rho  + \sum_{P, v} \mu(v) \langle u,v \rangle V_{v}\,.
\end{eqnarray*}
where $P,Q \in \PP^1$ and $u \in M$.
\end{theorem}

\begin{corollary}
\label{sec:cor-picard-rank}
The class group of $X$ has rank
\[
1 + \sum_P (\#\Vert_P -1) + \#\Hor - \dim N.
\]
\end{corollary}

\begin{example}[{Fano threefold 3.10 (continued)}]
  \label{exp:canonical}
  We are coming back to our T-variety from
  Example~\ref{ex:threefold3.10-1}. Since every ray in $\tail \fan$ intersects the degree, by Theorem~\ref{thm:divclass} the
  divisor class group of $X$ is generated by the seven vertical
  divisors
  \[
  \begin{array}{lll}
    D_1=V_{0,(0,0)}, &D_2=V_{0,(-\nicefrac{1}{2},0)}, &D_3=V_{0,(-1,0)},\\
    D_4=V_{\infty,(0,0)}, &D_5=V_{\infty,(0,-\nicefrac{1}{2})}, &D_6=V_{\infty,(0,-1)},\\
    &D_7=V_{1,(\nicefrac{1}{2},\nicefrac{1}{2})}.&\\
  \end{array}
  \]
  The relations are given by the rows of the following matrix
  \[
  M=\left(\begin{smallmatrix}
    1 & 2 & 1 & 0 & 0 & 0 &  -2 \\
    0 & 0 & 0 & 1 & 2 & 1 & -2 \\
    -1 & -1 & 0 & 0 & 0 & 0 &  1 \\
    0 & 0 & 0 & 0 & -1 & -1 & 1
  \end{smallmatrix}\right)
  \]
  Calculating the Smith normal form $D=P\cdot M \cdot Q$ gives
  \[
  D=\left(\begin{smallmatrix}
    1& 0& 0& 0& 0& 0& 0\\
    0& 1& 0& 0& 0& 0& 0\\
    0& 0& 1& 0& 0& 0& 0\\
    0& 0& 0& 1& 0& 0& 0
  \end{smallmatrix}\right), \quad P=I_4, \quad Q=
   \left(\begin{smallmatrix}
     0& 0& {-1}& 1&{-1}& 1& 1\\
     0& 0& 0& 0& 1& 0&0\\
    1& 0& 1& 1& {-1}& 1&1\\
    0& 1& 0& {2}& 0& 0&1\\
    0& 0& 0& 0& 0& 1&0\\
    0& 0& 0& 0& 0& 0&1\\
    0& 0& 0& 1& 0& 1&1\\
  \end{smallmatrix}\right)
  \]
Hence, $\cl(X)\cong \ZZ^3$. Where the identification is given by the last three columns of $Q$, i.e.  $[D_1], [D_3] \mapsto (-1,1,1)$, $[D_2]\mapsto (1,0,0)$, $[D_4] \mapsto (0,0,1)$, $[D_5]\mapsto (0,1,0)$,  $[D_6] \mapsto (0,0,1)$, $[D_7]\mapsto (0,1,1)$.
\end{example}

\subsection*{Ample divisors}
We consider a piecewise affine concave function \(\Psi: \Box \rightarrow \wdiv_\RR \PP^1,\) defined on a lattice polytope $\Box \subset M_\RR$. It induces corresponding  piecewise affine  and concave functions $\Psi_P:\Box \to \RR$  via $\Psi(u) = \sum_{P \in \PP^1} \Psi_P(u) \cdot P$. Using this notation we are going to define the crucial object of this section.
\begin{definition}
  A \emph{divisorial polytope} is a continuous piecewise affine concave function 
\(\Psi: \Box \rightarrow \wdiv_\RR \PP^1,\) such that
\begin{enumerate}
\item for every $u$ in the interior of $\Box$, $\deg \Psi(u) > 0$,
\item for every $P$ the graph of $\Psi_P$ is integral, i.e. has its vertices in $M \times \ZZ$.
\end{enumerate}

\noindent We define the \emph{support function} of $\Psi$ to be the family $\Psi^*=\{\Psi^*_P\}_{P\in \PP^1}$ of piecewise affine concave functions $\Psi^*_P: N_\RR \rightarrow \RR$ with   
 \[\Psi^*_P:=\min_{u \in \Box_h}(\langle u ,v  \rangle -\Psi_P(u)).\]
We call $\lin(\Psi^*):=\min \langle \Box, \cdot \rangle$ the linear part of $\Psi^*$. It is a piecewise linear function on the normal fan of $\Box$.
\end{definition}

\begin{proposition}[{\cite[Theorem~63]{tvars}, \cite[Theorem~3.2]{IS10}}]
\label{prop:divpolys}
  There is a one-to-one correspondence between invariant ample divisors $D$ on $X(\fan)$ and divisorial polytopes $\Psi:\Box \rightarrow \wdiv_\RR \PP^1$, with
  \begin{enumerate}
  \item $\Psi^*_P$ induces the subdivision $\fan_P$, i.e. the maximal polyhedra of $\fan_P$ are the regions of affine linearity of $\Psi^*_P$,
  \item $\deg \Psi(u)=0$ holds for a point $u \in \Box_D$ iff $\langle u, v\rangle = \min \langle \Box ,v \rangle$ for some $v \in\sdeg$\label{item:vertexcondition}.
  \end{enumerate}
\end{proposition}


Moreover, by \cite[Corollary 3.19, Theorem 3.12]{tidiv} $\Psi$ as in Proposition~\ref{prop:divpolys} corresponds to an anti-canonical divisor on $X=X(\fan)$, if there exists an integral divisor $K_{\PP^1} = \sum a_P \cdot P$ of degree $-2$ on $\PP^1$, such that  
\begin{equation}
   \label{eq:fano}
   \Psi^*_P(v)=-a_P - 1 + \nicefrac{1}{\mu(v)},\quad \lin(\Psi^*)(v_\rho)=1
\end{equation}
for every vertex $v \in \fan_P$ and every ray $\rho$ of $(\tail \fan)^\times$. In particular, $X$ will be Fano. In this case we say that $\Psi$ is a divisorial polytope corresponding to the Fano T-variety $X$.

\begin{remark} 
\label{rem:lattice-distance}

By Proposition~\ref{prop:divpolys} (i) we have a one-to-one correspondence between facets of the graph of $\Psi_P$ and vertices of $\fan_P$. A facet of the graph of $\Psi_P$ projects to a region $U \subset \Box$ where $\Psi_P$ is affine linear, i.e. $\Psi_P|_U = \langle \cdot , v \rangle + c$ for some $v \in N_\RR$. The corresponding vertex of $\fan_P$ is exactly $v$ and we have $\Psi^*(v) = c$. Taking this into account we can reformulate condition (\ref{eq:fano}) as follows.
  \begin{enumerate}
  \item For every facet $F$ of the graph of $\Psi_P + a_P + 1$ there is a vertex $v$ such that $\langle\;\cdot\; , \mu(v)(v,1) \rangle \equiv 1$ on $F$. In particular, $F$ has lattice distance $1$ from the origin. 
Moreover, we have $\Psi_P \equiv -a_P$ if $\fan_P$ is trivial.
  \item the facets $F$ of $\Box$ with $\deg \Psi|_F \not \equiv 0$ have lattice distance $1$ from the origin.
  \end{enumerate}
\end{remark}

\begin{definition}
  The \emph{volume} of a divisorial polytope is defined, by
  \[\vol \Psi := \int_{\Box} \! \deg \Psi \; d \mu.\]
\end{definition}

Here, we are integrating with respect to the Euclidean measure $\mu$ induced by the inclusion $M \subset M_\RR$.

\begin{theorem}[{\cite[Proposition 3.31]{tidiv}}]
\label{sec:intersection-nr}
  If $D$ is ample on a T-variety $X=X(\fan)$ of dimension $n$, then its top self intersection number is given by
  \[(D)^{n} = n!\vol \Psi_D.\]
\end{theorem}

\begin{example}[{Fano threefold 3.10 (continued)}]
\label{exp:degree}
In Figure~\ref{fig:divpoly-3.10} the divisorial polytope $\Psi: \Box \rightarrow \wdiv_\RR \PP^1$ for the canonical divisor is sketched. More precisely we give the non-trivial concave functions $\Psi_P$ by drawing $\Box$ and the regions of affine linearity of $\Psi_P$ and giving the values of $\Psi$ at the vertices of these regions. 
 
 \begin{figure}[htbp]
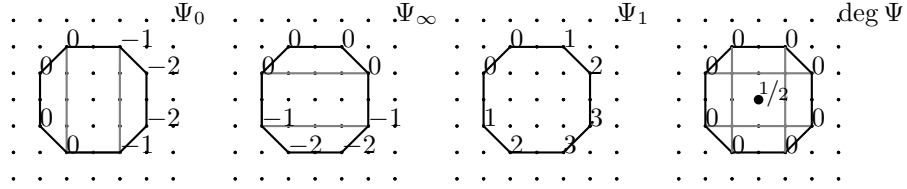

    \centering
    \dualBBUnshaded
    \vspace{-1.5em}
    \caption{A divisorial polytope for the threefold 3.10.}
    \label{fig:divpoly-3.10}
  \end{figure}

We have $\deg \Psi=0$ on the boundary of $\Box$ and $\sdeg$ intersects all tail cones. Hence, condition (ii) of Proposition~\ref{prop:divpolys} is fulfilled. We will now check that the $\Psi$ induces the subdivisions $\fan_P$ and fulfills (\ref{eq:fano}), for $K_{\PP^1}=-2\cdot[1]$. Calculating $\Psi^*$ gives
\begin{align*}
  \Psi^*_0 =  \min\big(\; &\textstyle\langle {2 \choose 1},\cdot \rangle + 2,\;  \langle {1 \choose 2},\cdot \rangle + 1, \; \langle {-1 \choose 2},\cdot \rangle + 0,\;  \langle {-2 \choose 1},\cdot \rangle + 0,\\
                   &\textstyle \langle {2 \choose -1},\cdot \rangle + 2,\;  \langle {1 \choose -2},\cdot \rangle + 1, \; \langle {-1 \choose -2},\cdot \rangle+0,\;  \langle {-2 \choose -1},\cdot \rangle + 0 \big).\\
\Psi^*_\infty =  \min\big(\; &\textstyle\langle {2 \choose 1},\cdot \rangle + 0,\;  \langle {1 \choose 2},\cdot \rangle + 0, \; \langle {-1 \choose 2},\cdot \rangle + 0,\;  \langle {-2 \choose 1},\cdot \rangle + 0,\\
                   &\textstyle \langle {2 \choose -1},\cdot \rangle + 1,\;  \langle {1 \choose -2},\cdot \rangle + 2, \; \langle {-1 \choose -2},\cdot \rangle+2,\;  \langle {-2 \choose -1},\cdot \rangle + 1 \big).\\
\Psi^*_1 =  \min\big(\; &\textstyle\langle {2 \choose 1},\cdot \rangle  -2,\;  \langle {1 \choose 2},\cdot \rangle - 1, \; \langle {-1 \choose 2},\cdot \rangle + 0,\;  \langle {-2 \choose 1},\cdot \rangle + 0,\\
                   &\textstyle \langle {2 \choose -1},\cdot \rangle -3,\;  \langle {1 \choose -2},\cdot \rangle -3, \; \langle {-1 \choose -2},\cdot \rangle-2,\;  \langle {-2 \choose -1},\cdot \rangle - 1 \big).
\end{align*}
It's not hard to check, that $\Psi^*_P$ is affine linear exactly on the maximal polyhedra of $\fan_P$. Hence, by Proposition~\ref{prop:divpolys} it corresponds to an ample divisor $D$ on $X(\fan)$. Moreover, we obtain 
\[
\begin{array}{rrr}
  \Psi^*_0(0,0)=0&h_0(-\nicefrac{1}{2},0)=\nicefrac{1}{2}&h_0(-1,0)=0\\
\Psi^*_\infty(0,0)=0&h_\infty(0,-\nicefrac{1}{2})=\nicefrac{1}{2}&h_\infty(0,-1)=0\\
              \Psi^*_1(\nicefrac{1}{2},\nicefrac{1}{2})=1& &
\end{array}
\]
Hence,  (\ref{eq:fano}) is fulfilled for our choice $K_{\PP^1}=-2\cdot[1]$ and $D$ is an anti-canonical divisor. 

By elementary calculations for $\vol \Psi$ we get
\[\vol \Psi=\int_{\Box} \! \deg \Psi  = \frac{13}{3}.\]

Putting together the observations of Example~\ref{ex:threefold3.10-1}, \ref{exp:canonical} and \ref{exp:degree} we conclude that $X$ is a smooth Fano threefold of Picard rank $3$ and Fano degree $26$. By \cite{mm86} this is a blow up of the quadric threefold in two disjoint conics.
\end{example}

\subsection{Moment map and Duistermaat-Heckman measure}
\label{sec:moment-map}
In this section we give an interpretation of divisorial polytopes in terms of moment maps. 
  
The polarisation of a T-variety $X$ by $L=\CO(D)$, for an ample invariant divisor $D$, induces a symplectic structure on $X$ by pulling back the Fubini-Study form of the corresponding torus equivariant embedding into projective space (at least for a suitable multiple of $D$). The action of the algebraic torus action on $X$ induces a Hamiltonian action of the contained (maximal) compact torus $T_c$ on the corresponding symplectic manifold. This action comes with a moment polytope $P$ and a moment map $\phi:X \rightarrow P$. Now, the push forward of the canonical symplectic measure on $X$ via the moment map  gives a measure on $P$. This measure is called the Duistermaat-Heckman measure and the corresponding continuous density function $f$ is called the Duistermaat-Heckman function. It is known to be piecewise polynomial and the degree corresponds to the complexity of the torus action, cf. \cite{dh-measure}.  

The global sections $L$ decompose in homogeneous components of weight $u\in M$. 
By \cite[Proposition~3.23]{tidiv} we may express such a homogeneous component in terms of the corresponding divisorial polytope $\Psi_D$:
  \begin{equation}
    \label{eq:sections}
    H^0\big(X, L\big)_u \cong
    \begin{cases}
      H^0 \big(\PP^1,\, \CO(\Psi_D(u))\big)&, u\in \Box_D\\
      0&, \text{else}
    \end{cases}
  \end{equation}
  From this one can deduce that $\Box=\Box_D$ equals the moment polytope $P$, see \cite[Section~14.7]{tvars}. More generally for powers if $L$ we get by \cite[Prop. 3.1]{IS10}
  \begin{equation}
    \label{eq:sections-rational}
H^0\big(X, L^{\otimes k} \big)_{k \cdot u} = H^0 \big(\PP^1,\, \CO(k\cdot \Psi_D(u))\big)
\end{equation}
for $u \in \Box \cap \frac{1}{k}\cdot M$.

 One obtains the value $f(0)$ as the volume of the symplectic reduction of the moment fiber $\phi^{-1}(0)$. By the Kempf-Ness Theorem the symplectic reduction  of $\phi^{-1}(0)$ by $T_c$ is the same as the GIT quotient $Y=X/\!\!/^{L} T = \proj \bigoplus_k  H^0(X, L^{\otimes k})^T$. The GIT quotient comes together with a polarization by a $\QQ$-line bundle $\CO_Y(1)$, given by the invariant sections of $L$. This polarization is compatible with the symplectic structure on $\phi^{-1}(0)/T_c$.  Hence, we have $\vol \phi^{-1}(0)/T_c = (\CO_Y(1))^{\dim Y}$.

If  $0$ is in the interior of $\Box$ we get $Y=X/\!\!/^{L} T = \PP^1$. Moreover, by (\ref{eq:sections}) we have $\CO_Y(1) = \CO(\Psi(0))$ and we get $f(0) = \deg(\Psi(0))$. By shifting the linearization of $L$ (and therefore the moment map) by $u \in \Box \cap M$ we get $f(u) = \deg(\Psi(u))$ for all lattice points in $\Box$. By (\ref{eq:sections-rational}) we get the same for rational points in $\Box$ and eventually by continuity for all points.

Altogether we obtain
\begin{proposition}
\label{prop:dh-measure}
  The Duistermaat-Heckman function corresponding to the polarization given by an invariant ample divisor $D$ with divisorial polytope $\Psi_D$ is given by $\deg \Psi:\Box_D \rightarrow \RR$.
\end{proposition}

The other way around, we may interpret the divisorial polytope  $\Psi_D$ as a decomposition of the Duistermaat-Heckman function corresponding to the polarization of $X$ by $L$. This shows the relation between our divisorial polytopes and the description of Hammiltonian torus actions given by Karshon and Tolman in \cite{kt03}.

\subsection*{The Cox ring}
\label{sec:cox-ring}
For a normal variety with a class group isomorphic to $\ZZ^r$ the Cox ring is defined as
\[\Cox(X)=\bigoplus_{\alpha \in \ZZ^r} H^0\left(X,\CO(\textstyle\sum_{i} \alpha_i D_i)\right),\]
with $D_1, \ldots, D_r$ being a basis of $\cl(X)$. For class groups which are not free the definition is little bit more involved. Since smooth Fano varieties always have free class group, we don't have to discuss this here.



In \cite[Theorem~1.2]{cox} the Cox ring of a T-variety  was calculated in terms of the combinatorial data. Due to \cite{tvars} we have the following reformulation of this theorem in our language.
\begin{theorem}[{\cite[Theorem~40]{tvars}}]
\label{thm:cox-ring}
  \[\Cox(X)=\frac{\CC\left[T_{v},S_\rho \mid v \in \Vert, \rho \in \Hor\right]}{\langle T^{\mu(0)} + cT^{\mu(\infty)} + T^{\mu(c)} \mid c \in \CC^* \rangle},\]
where \(T^{\mu(P)}=\prod_{v \in \Vert_P} T_{v}^{\mu(v)}.\)

The $\cl(X)$-grading is given by 
\(\deg(T_{v})= [V_{v}]$,  $\deg(S_\rho)=[H_\rho].\)
\end{theorem}


\begin{example}[{Fano threefold 3.10 (continued)}]
  \label{exp:cox}
For our f-divisor from Example~\ref{ex:threefold3.10-1} we get the Cox ring
\[\CC[T_1,\ldots,T_7]/\langle T_1T_2^2T_3+T_4T_5^2T_6+T_7\rangle.\] 
The grading is given by $\deg(T_i) = [D_i]$, with the notation of Example~\ref{exp:canonical}. Using the identification $\cl X  \cong \ZZ^3$ from Example~\ref{exp:canonical} we obtain the weight matrix

\[\begin{array}{rrrrrrrrl}
  &T_1&T_2&T_3&T_4&T_5&T_6&T_7
\vspace{2mm}\\
 \ldelim({3}{0.5ex}&-1&1&-1&0&0&0&0\rdelim){3}{0ex}\\
  &1& 0&1&0&1&0&1\\
  &1& 0&1&1&0&1&1\\
\end{array}\]
\end{example}

\bigskip

\subsection*{The Futaki character}
\label{sec:futaki-character}
In \cite{0506.53030} Futaki introduced the invariant as an obstruction for the existence of K\"ahler-Einstein metrics on Fano manifolds. Hence, for a K\"ahler-Einstein Fano manifold the Futaki character has to vanish. 

Donaldson gave a an algebraic redefinition of this invariant in \cite{1074.53059}, which we are using in the following. Consider a $\CC^*$-action $\lambda$ on a normal variety $X$ of dimension $n$ and an invariant Cartier divisor $D$. Now, let $l_k=\dim H^0(X,\CO(kD))$ and $w_k$ the total weight of the $\CC^*$-action on $H^0(X,\CO(kD))$, i.e. $w_k= \sum_i i\cdot \dim H^0(X,\CO(kD))_i$. Then $l_k$ and $w_k$ obtain expansions
\begin{align*}
  l_k &= a_0 k^n + a_1 k^{n-1} + O(k^{n-2}),\\
  w_k &= b_0 k^{n+1} + a_1 k^{n} + O(k^{n-1}).
\end{align*}
\begin{definition}
   \label{sec:def-futaki}
   The Futaki invariant of this $\CC^*$-action is defined as 
   \[F_D(\lambda)=2 \cdot \frac{a_1b_0-a_0b_1}{a_0}.\]
\end{definition}
\begin{remark}
  Note, that we choose a different scaling and sign as Donaldson for the invariant. More precisely his invariant equals $-\frac{1}{2a_0}F_D$ in our notation.
\end{remark}

 By plugging in one-parameter subgroups of an acting torus we obtain a linear map $F_D:N \rightarrow \RR$, hence an element of $M_\RR$. If $X$ is Fano we set $F(X):=F_{(-K_X)}$.

 For a toric variety Donaldson in \cite{1074.53059} gives an easy formula to actually calculate the Futaki character. Let $\Delta$ be the polytope corresponding to an ample divisor $D$ on a toric variety. Then the Futaki character is given by
 \begin{equation}
   \label{eq:futaki}
   F_D(v) \;=\; \int_{\partial \Delta} v \;d\mu_1 \; - \; \frac{\rvol \partial\Delta}{\vol \Delta}\cdot\int_\Delta v \; d\mu.
 \end{equation}
   Here, we are integrating over $\Delta$ with the Euclidean measure $\mu$ induced by the lattice $M \subset M_\RR$. For integrating over the boundary we are using the \emph{facet measure} $\mu_1$, which is induced by $(\RR \cdot \delta) \cap M$ for every facet $\delta \prec \Delta$. The corresponding \emph{facet volume} is denoted by $\rvol$ and the \emph{facet barycenter} by $\rbc$. Note, that up to scaling by $(\rvol \partial\Delta)$, for the Futaki character we obtain the difference of the barycenters of $\Delta$ and its boundary:
 \begin{equation}
   \label{eq:futaki-barycenter}
   F_D =(\rvol \partial\Delta)\cdot(\rbc(\partial\Delta) - \bc(\Delta)).
 \end{equation}

It's easy to see that for the standard polytope of the canonical divisor $\Delta=\{u \mid \forall_{\rho \in \Sigma^{(1)}}: \langle u, n_{\rho} \rangle \geq -1\}$ the formula simplifies to $F(X) = \vol \Delta \cdot \bc \Delta$, i.e. the Futaki character is up to positive scaling  given by the barycenter of the polytope. This was already observed by Mabuchi \cite{0661.53032}. For generalizing the last result to our case we define the \emph{barycenter} $\overline{\bc}(\Psi)$ of a divisorial polytope $\Psi:\Box \to \wdiv_\RR \PP^1$ by
\[\langle \overline{\bc}(\Psi) , v \rangle =  \int_{\Box_h} \!v \cdot \deg \Psi \; d\mu.\]
This gives the same as the projection to $M_\RR$ of the barycenter of
\[\Delta:=\Delta(\deg \Psi):=\{(u,x) \mid u \in \Box,\;0 \leq x \leq \deg(u) \}.\]

\begin{theorem}
\label{thm:futaki-simplified}
  For a Fano T-variety $X$ and a corresponding divisorial polytope $\Psi$ we have \[F(X) = \vol\Psi\cdot \overline{\bc}\Psi.\]
\end{theorem}

We postpone the proof to the appendix.

\begin{remark}
  By Proposition~\ref{prop:dh-measure} we may interpret the Futaki character as barycenter of the moment polytope with respect to the Duistermaat-Heckman measure $(\deg \Psi) \cdot \mu$, see also \cite[Theorem 9.2.3]{0661.53032} for closely related results.   
\end{remark}

\begin{example}
\label{exp:futaki}
We would like to compute the Futaki invariant of the Fano threefold 3.23. Its divisorial polytope $\Psi$ is given by Figure~\ref{fig:futaki}.  We can read off $\Delta(\deg \Psi)$ directly from the rightmost picture. 
\begin{figure}[htpb]
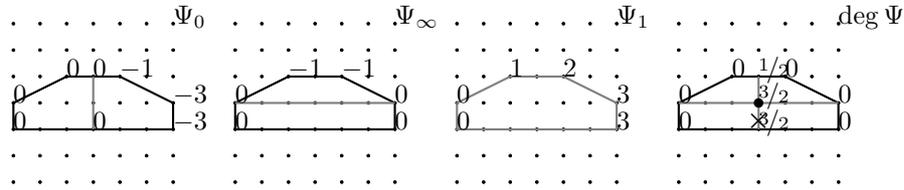

  \centering
   \dualBFlabeled
   \vspace{-1em}
  \caption{Divisorial polytope of 3.23}
  \label{fig:futaki}
\end{figure}
The barycenter of this polytope is $(0, -\nicefrac{9}{40}, \nicefrac{37}{80})^t$ and the volume $\nicefrac{20}{3}$. Note, that we have $\vol \Psi= \vol \Delta(\deg \Psi)$. Hence, by applying Theorem~\ref{thm:futaki-simplified} we get $F(X) = {0 \choose -\nicefrac{3}{2}}\neq 0$. In particular, there does not exists a K\"ahler-Einstein metric on $X$.
\end{example}

\section{Toric degenerations}
\label{sec:degenerations}
We are using the notation and results of \cite{1244.14044} to describe degenerations to toric varieties via f-divisors. 

First note, that the torus action of a T-variety $X(\fan)$ is not necessarily maximal. Consider an f-divisor $\fan$ with two non-trivial slices $\fan_0$ and $\fan_\infty$. Remark 1.8 in \cite{1244.14044} shows that the $T$-action on $X(\fan)$ extends to the action of a torus $T'$ of full dimension. Hence, $X$ is essentially toric with a fan $\Sigma$ as follows:  We embed the polyhedra of $\fan_0$ in height $1$ of $N_\QQ \times \QQ$ and consider the fan $\Sigma_0$ consisting of the cones over these polyhedra. We are doing the same with $\fan_\infty$ embedded in height $-1$ and obtain $\Sigma_\infty$. Now, we obtain $\Sigma$ from $\Sigma_0 \cup \Sigma_\infty$ by joining cones which intersect in a common facet of the form $\sigma \times \{0\}$, such that $\sigma \cap \sdeg \neq \emptyset$. 
\begin{example}
  \label{exp:downgrade-global}
Consider the f-divisor pictured in Figure~\ref{fig:downgrade-global} having two non-trivial slice $\fan_0$ and $\fan_\infty$.
\begin{figure}[htbp]
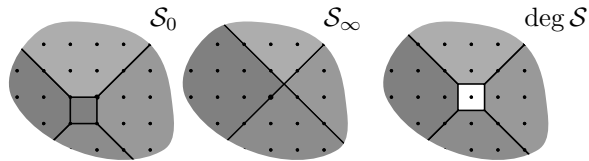

  \centering
  \threefoldQDeg
  \caption{F-divisor of a downgraded toric variety.}
  \label{fig:downgrade-global}
\end{figure}
The fan $\Sigma_0$ from above consists of five maximal cones: 
\begin{align*}
  \sigma_0&=\pos\big((0,0,1),(-1,0,1),(-1,-1,1),(0,-1,1)\big)\\
  \sigma_1&=\pos\big((0,0,1),(-1,0,1),(1,1,0),(-1,1,0)\big)\\
  \sigma_2&=\pos\big((0,0,1),(0,-1,1),(1,1,0),(1,-1,0)\big)\\
  \sigma_3&=\pos\big((-1,0,1),(-1,-1,1),(-1,1,0),(-1,-1,0)\big)\\
  \sigma_4&=\pos\big((-1,-1,1),(0,-1,1),(1,-1,0),(-1,-1,0)\big).
\end{align*}
The fan $\Sigma_\infty$ consists of the four cones 
\begin{align*}
  \tau_1&=\pos\big((1,1,-2),(1,1,0),(-1,1,0)\big)\\
  \tau_2&=\pos\big((1,1,-2),(1,1,0),(1,-1,0)\big)\\
  \tau_3&=\pos\big((1,1,-2),(-1,1,0),(-1,-1,0)\big)\\
  \tau_4&=\pos\big((1,1,-2),(1,-1,0),(-1,-1,0)\big).
\end{align*}
We have to join four pairs  of cones $(\sigma_i,\tau_i)$, $i=1,\ldots,4$ and obtain the fan $\Sigma$ consisting of the five maximal cones
\begin{align*}
  \delta_0&=\sigma_0\\
  \delta_1&=\pos\big((0,0,1),(-1,0,1),(1,1,-2)\big)\\
  \delta_2&=\pos\big((0,0,1),(0,-1,1),(1,1,-2)\big)\\
  \delta_3&=\pos\big((-1,0,1),(-1,-1,1),(1,1,-2)\big)\\
  \delta_4&=\pos\big((-1,-1,1),(0,-1,1),(1,1,-2)\big).
\end{align*}
This is the fan of the projective cone over the quadric surface. Moreover, it is the face fan of the (reflexive) polytope 
\[\conv\big((0,0,1),(-1,0,1),(-1,-1,1),(0,-1,1),(1,1,-2)\big)\]
which has ID 544395 in the classification of canonical toric Fano threefolds from \cite{Kas08}\footnote{see \url{http://grdb.lboro.ac.uk/search/toricf3c?ID=544395}}. 
\end{example}

Now, we consider a polyhedral subdivision $\Xi$. A Minkowski-decomposition of $\Xi$ consist of polyhedral subdivisions $\Delta=\Delta^1 + \ldots + \Delta^r$ for every $\Delta \in \Xi$ such that 
\begin{enumerate}
\item $(\Delta \cap \nabla)^i = \Delta^i \cap \nabla^i$, for $\Delta,\nabla \in \Xi$ and $\Delta \cap \nabla \neq \emptyset$,
\item For $I \subset \{1, \ldots, r\}$, $\cJ \subset \cI \subset \Xi$ we have
\[\sum_{i \in I} \bigcap_{\Delta \in \cI}\Delta^i \prec \sum_{i \in I} \bigcap_{\Delta \in \cJ}\Delta^i\]
\end{enumerate}

We obtain polyhedral subdivisions $\Xi^i := \{\Delta^i \mid \Delta \in \Xi\}$. By abuse of notation we will also write $\Xi=\Xi^1 + \ldots + \Xi^r$ for this situation. Such a decomposition is called \emph{admissible} if for every vertex $v$ of $\Xi$ there is at most one of the corresponding vertices $v_i \in \Xi_i$ with $v=\sum_i v_i$ which is not a lattice point.

 We may start with an f-divisor $\fan$ with non-trivial slices $\Xi_0=\fan_{P_0}, \ldots ,\Xi_\ell = \fan_{P_\ell}$ and an admissible Minkowski decomposition $\Xi_0=\Xi_{\ell +1} + \ldots + \Xi_{\ell +r}$. As described in \cite[Sections 2 and 4]{1244.14044} this data gives rise to a deformation of $X(\fan)$. A general fiber of this deformation corresponds to an f-divisor $\fan'$ with non-trivial slices being exactly 
\[\fan'_{P_1'}=\Xi_1,\; \ldots\;,\; \fan'_{P_{\ell+r}'} = \Xi_{\ell +r}.\]
For some distinct points $P_1', \ldots P_{\ell+r}' \in \PP^1$.

 We now reverse the above procedure in order to obtain toric degenerations. Let's start with an f-divisor and assume that all non-trivial slices are contained in 
$\{\fan_{P_1}, \ldots \fan_{P_{r}}, \fan_\infty\}$ and the first $r$ of them form an admissible Minkowski decomposition of some subdivision $\Xi$
\[\Xi = \fan_{P_1} + \ldots +\fan_{P_r}.\]
Then $X(\fan)$ is a deformation of the T-variety corresponding to the f-divisor
\[(\Xi \otimes [0] + \fan_\infty \otimes [\infty], \sdeg)\]
which describes a subtorus action on a toric variety as we have seen above.

\begin{example}
  \label{exp:degeneration}
  We consider the f-divisor with slices $\fan_0$, $\fan_\infty$, $\fan_1$ and degree sketched in Figure~\ref{fig:q}. The corresponding variety turns out to be the quadric threefold.
  \begin{figure}[htbp]
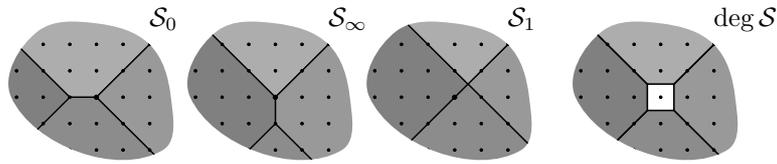

    \centering
    \threefoldQLabeled
    \caption{F-divisor of the quadric threefold}
    \label{fig:q}
  \end{figure}

  Note, that $\fan_0$ and $\fan_\infty$ form an admissible Minkowski decomposition, which is sketched in Figure~\ref{fig:decomposition}.
  \begin{figure}[htbp]
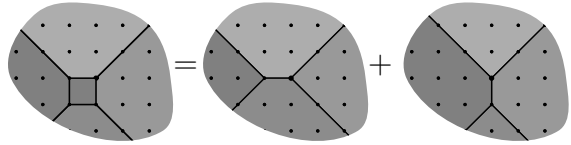

    \centering
    \threefoldQDegSum
    \caption{An admissible Minkowski decomposition}
    \label{fig:decomposition}
  \end{figure}
\end{example}

Hence, $X$ degenerates to the toric variety from Example~\ref{exp:downgrade-global}. We find another toric degeneration to  the canonical toric Fano variety with ID 547378 in \cite{Kas08} by adding up $\Xi=\fan_0+\fan_1$. Adding up $\Xi=\fan_0+\fan_\infty + \fan_1$ will give a toric degeneration to a non-canonical toric Fano variety.
\begin{remark}
  Note, that we do not get all toric degeneration with this method, but only those which are equivariant with respect to the $T$-action on $X$.
\end{remark}


\section{Proof by pictures}
\label{sec:pictures}
Aim of this section is to proof Theorem~\ref{thm:main}. For every variety from the list in Theorem~\ref{thm:main} we will state the f-divisor $\fan$ and a divisorial polytope $\Psi$, which corresponds to a canonical divisor. To verify the proof for a particular threefold one has to check first, that the f-divisor and divisorial polytope actually correspond to the given threefold. This is done by carrying out the following procedure.
\begin{enumerate}
\item[A.1] Applying Theorem~\ref{thm:smoothness}: checking the smoothness of $X(\fan)$ by checking the regularity of certain cones (Example~\ref{ex:threefold3.10-1}).
\item[A.2] Calculating the Picard rank of $X(\fan)$ by Corollary~\ref{sec:cor-picard-rank}.
\item[A.3] Checking that $\Psi$ corresponds to a canonical divisor (Example~\ref{exp:degree}).
\item[A.4] Calculating the Fano degree of $X(\fan)$ by Theorem~\ref{sec:intersection-nr} (Example~\ref{exp:degree}).
\end{enumerate}
Now, by the classification in \cite{mm86} we know that the given f-divisor describes a Fano threefold in the stated family. Hence, we constructed a Fano threefold with 2-torus action within this family. To check the other statement of Theorem~\ref{thm:main} we have to apply additional steps
\begin{enumerate}
\item[B.1] Applying Theorem~\ref{thm:cox-ring} to calculate the Cox ring (Example~\ref{exp:cox}).
\item[B.2] Applying Theorem~\ref{thm:futaki-simplified} to calculate the Futaki invariant (Example~\ref{exp:futaki}).
\item[B.3] Checking symmetry as in Example~\ref{sec:ex-symmetry}.
\item[B.4] Checking if $F(X)=0$ and if possible apply Theorem~\ref{thm:kesym} to check for the K\"ahler-Einstein property (Examples~\ref{sec:ex-symmetry}, \ref{exp:futaki}).
\item[B.5] ``Adding up slices'' to obtain toric degenerations (Example~\ref{exp:degeneration}).
\end{enumerate}
Beside the given examples we leave it to the reader to carry out the verification procedure for the given f-divisors. 

We give some hints how to interpret the pictures and information given below. As in the examples we first state  the non-trivial slice $\fan_0$, $\fan_\infty$, $\fan_1$ and $\sdeg$ of an f-divisor $\fan$.  Next, we state the divisorial polytope of the canonical divisor by giving $\Psi_0$, $\Psi_\infty$ and $\Psi_1$ (all other $\Psi_P$ are assumed to be trivial). For convenience of the reader we state $\deg \Psi$, as well. In the pictures $\bullet$ marks the origin and $\times$ the Futaki character (if it differs from the origin).  If there exists equivariant degenerations to canonical toric varieties then we state the corresponding IDs for the Fano polytopes in the \emph{Graded Ring Database} \cite{grdb}.

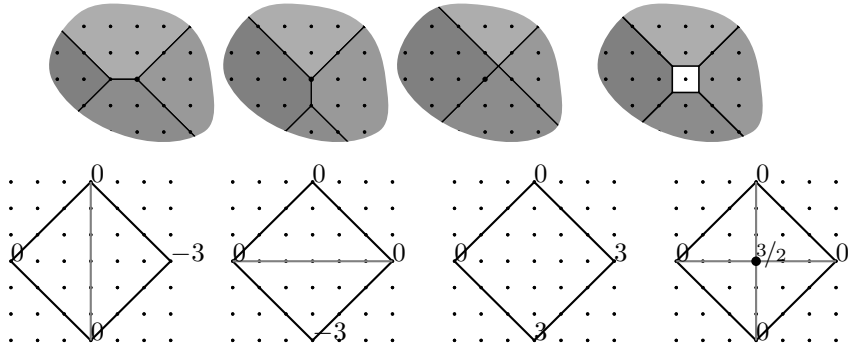
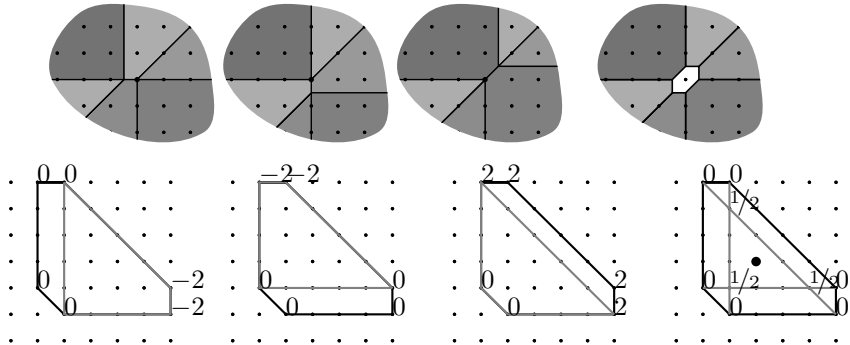
\begin{center}
\begin{tabular}{lrclrclr}
\toprule
 \bf Name :& Q &\multicolumn{6}{l}{{\bf Description:}  Quadric threefold}\\
\midrule
 \bf Picard rank: & 1 & & \bf Fano Degree: & 54 & &  \bf symmetric: & yes\\
 \bf Futaki-char.: & 0 & &\bf Degenerations: & \multicolumn{4}{l}{544395, 547378}\\
\multicolumn{4}{c}{\bf Cox ring:  $\CC[T_1,\ldots,T_5]/\langle T_1T_2+T_3T_4+T_5^2\rangle$} & \multicolumn{4}{c}{$\deg(\underline{T})=(1,1,1,1,1)$}\\
\end{tabular}
 \threefoldQ\\ 
\dualQ\\

\enlargethispage{1cm}
\vspace{0.3em}
\begin{tabular}{lrclrclr}
\toprule
\bf Name :& 2.24 &\multicolumn{6}{l}{{\bf Description:} Divisor of bidegree $(1,2)$ in $\PP^2 \times \PP^2$}\\
\midrule
 \bf Picard rank: & 2 & & \bf Fano Degree: & 54 & &  \bf symmetric: & yes\\
 \bf Futaki-char.: & 0 & &  \bf Degenerations: & \multicolumn{4}{l}{---}\\
\multicolumn{4}{c}{\bf Cox ring:  $\CC[T_1,\ldots,T_6]/\langle T_1T_2^2+T_3T_4^2+T_5T_6^2\rangle$} & \multicolumn{4}{c}{$\deg(\underline{T})=
  \left(\begin{smallmatrix}
    1 & 0 & 1 & 0 & 1 & 0\\
    0 & 1 & 0 & 1 & 0 & 1
  \end{smallmatrix}\right)
$}\\
\end{tabular}
 \threefoldAA\\
 \dualAA\\

\vspace{0.3em}

\begin{tabular}{lrclrclr}
  \toprule
  \bf Name :& 2.29 & \multicolumn{6}{l}{{\bf Description:} Blowup of $Q$ in a conic}\\
  \midrule
  \bf Picard rank: & 2 & & \bf Fano Degree: & 40 & &  \bf symmetric: & yes \\
  \bf Futaki-char.: & 0 & &  \bf Degenerations: & \multicolumn{4}{l}{430083, 544339}\\
  \multicolumn{4}{c}{\bf Cox ring:  $\CC[T_1,\ldots,T_6]/\langle T_1T_2^2T_3+T_4T_5+T_5^2\rangle$} & \multicolumn{4}{c}{$
    \left(\begin{smallmatrix}
        {-1}& 1& {-1}& 0& 0& 0\\
        1& 0& 1&    1& 1&   1\\
      \end{smallmatrix}\right)
$}\\

\end{tabular}
 \threefoldAB\\
 \dualAB\\

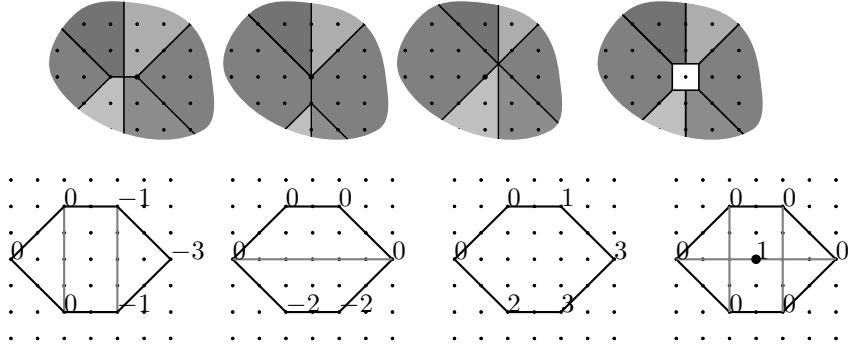
\vspace{0.3em}
\begin{tabular}{lrclrclr}
\toprule
\bf Name :& 2.30 &\multicolumn{6}{l}{{\bf Description:} Blow up of $Q$ in a point}\\
\midrule
 \bf Picard rank: & 2 & & \bf Fano Degree: & 46 & &  \bf symmetric: & no\\
 \bf Futaki-char.: & ${0 \choose -2}$ & &  \bf Degenerations: & \multicolumn{4}{l}{520157, 544343}\\
\multicolumn{4}{c}{\bf Cox ring:  $\CC[T_1,\ldots,T_6]/\langle T_1T_2+T_3T_4+T_5^2\rangle$} & \multicolumn{4}{c}{$
    \left(\begin{smallmatrix}
1&
      1&
      {2}&
      0&
      1&
      {-1}\\
      0&
      0&
      {-1}&
      1&
      0&
      1\\
      \end{smallmatrix}\right)
$}\\
\end{tabular}
 \threefoldAC\\
 \dualAC\\

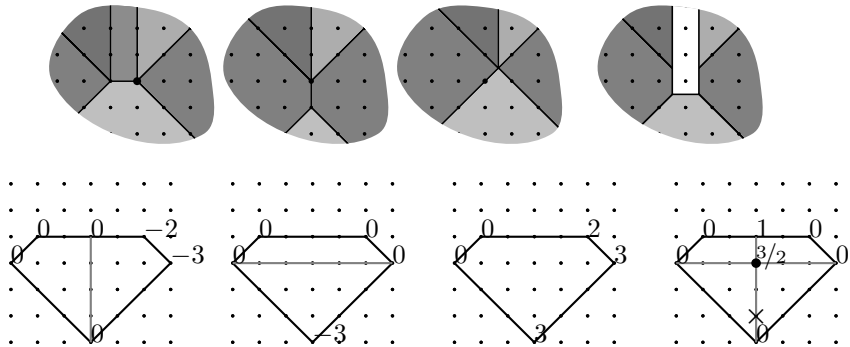
\vspace{0.3em}
\begin{tabular}{lrclrclr}
\toprule
\bf Name :& 2.31 &\multicolumn{6}{l}{{\bf Description:} Blow up of $Q$ in a line}\\
\midrule
 \bf Picard rank: & 2 & & \bf Fano Degree: & 46 & &  \bf symmetric: & no\\
 \bf Futaki-char.: & $-{\nicefrac{4}{3} \choose \nicefrac{4}{3}}$ & &  \bf Degenerations: & \multicolumn{4}{l}{520058, 520159}\\
\multicolumn{4}{c}{\bf Cox ring:  $\CC[T_1,\ldots,T_6]/\langle T_1T_2+T_3T_4+T_5T_6^2\rangle$} & \multicolumn{4}{c}{$
 \left(\begin{smallmatrix}
      0&
      1&
      1&
      0&
      -1&
      {1}\\
      1&
      0&
      0&
      1&
      1&
      0\\
      \end{smallmatrix}\right)
$}\\
\end{tabular}
 \threefoldAD\\
 \dualAD\\

\vspace{0.3em}
\begin{tabular}{lrclrclr}
\toprule
\bf Name :& W or 2.32 &\multicolumn{6}{l}{{\bf Description: }Divisor of bidegree $(1,1)$ in $\PP^2 \times \PP^2$}\\
\midrule
 \bf Picard rank: & 2 & & \bf Fano Degree: & 46 & &  \bf symmetric: & yes\\
 \bf Futaki-char.: & 0 & &  \bf Degenerations: & \multicolumn{4}{l}{520058, 520159}\\
\multicolumn{4}{c}{\bf Cox ring:  $\CC[T_1,\ldots,T_6]/\langle T_1T_2+T_3T_4+T_5T_6\rangle$} & \multicolumn{4}{c}{$\deg(\underline{T})=
  \left(\begin{smallmatrix}
    1 & 0 & 1 & 0 & 1 & 0\\
    0 & 1 & 0 & 1 & 0 & 1
  \end{smallmatrix}\right)
$}\\
\end{tabular}
 \threefoldAE\\
 \dualAE\\


\enlargethispage{1cm}
\vspace{0.3em}
\begin{tabular}{lrclrclr}
\toprule
\bf Name :& 3.10 &\multicolumn{6}{l}{{\bf Description: }Blowup of $Q$ in two disjoint invariant conics}\\
\midrule
 \bf Picard rank: & 3 & & \bf Fano Degree: & 26 & &  \bf symmetric: & yes\\
 \bf Futaki-char.: & $0$ & &  \bf Degenerations: & \multicolumn{4}{l}{---}\\
\multicolumn{4}{c}{\bf Cox ring:  $\CC[T_1,\ldots,T_6]/\langle T_1T^2_2T_3+T_4T_5^2T_6+T_7\rangle$} & \multicolumn{4}{c}{$\deg(\underline{T})=
  \left(\begin{smallmatrix}
{-1}&
      1&
      {-1}&
      0&
      0&
      0&
      0\\
      1&
      0&
      1&
      0&
      1&
      0&
      1\\
      1&
      0&
      1&
      1&
      0&
      1&
      1\\
  \end{smallmatrix}\right)
$}\\
\end{tabular}
 \threefoldBB\\
 \dualBB\\

\vspace{0.3em}
\begin{tabular}{lrclrclr}
\toprule
\bf Name :& 3.18 &\multicolumn{6}{l}{{\bf Description: }Blowup of $Q$ in a point and a conic}\\
\midrule
 \bf Picard rank: & 3 & & \bf Fano Degree: & 36 & &  \bf symmetric: & yes\\
 \bf Futaki-char.: & ${0 \choose -\nicefrac{7}{8}}$& &  \bf Degenerations: & \multicolumn{4}{l}{255687, 519918}\\
\multicolumn{4}{c}{\bf Cox ring:  $\CC[T_1,\ldots,T_7]/\langle T_1T_2^2T_3+T_4T_5+T_6^2\rangle$} & \multicolumn{4}{c}{$\deg(\underline{T})=
  \left(\begin{smallmatrix}
{-1}&
      1&
      {-1}&
      0&
      0&
      0&
      0\\
      1&
      0&
      1&
      {2}&
      0&
      1&
      {-1}\\
      0&
      0&
      0&
      {-1}&
      1&
      0&
      1\\
  \end{smallmatrix}\right)
$}\\
\end{tabular}
 \threefoldBC\\
 \dualBC\\

\enlargethispage{1cm}
\vspace{0.3em}
\begin{tabular}{lrclrclr}
\toprule
\bf Name :& 3.19 &\multicolumn{6}{l}{{\bf Description: }Blowup of $Q$ in two points}\\
\midrule
 \bf Picard rank: & 3 & & \bf Fano Degree: & 38 & &  \bf symmetric: & yes\\
 \bf Futaki-char.: & $0$ & &  \bf Degenerations: & \multicolumn{4}{l}{430426, 519917}\\
\multicolumn{4}{c}{\bf Cox ring:  $\CC[T_1,\ldots,T_7]/\langle T_1T_2+T_3T_4+T_5^2\rangle$} & \multicolumn{4}{c}{$\deg(\underline{T})=
  \left(\begin{smallmatrix}
1&
      1&
      {2}&
      0&
      1&
      {-1}&
      0\\
      0&
      0&
      {-1}&
      1&
      0&
      1&
      0\\
      0&
      0&
      0&
      0&
      0&
      1&
      1\\
  \end{smallmatrix}\right)
$}\\
\end{tabular}
 \threefoldBD\\
 \dualBD\\
\enlargethispage{1cm}

\vspace{0.3em}
\begin{tabular}{lrclrclr}
\toprule
\bf Name :& 3.20 &\multicolumn{6}{l}{{\bf Description: }Blowup of $Q$ in two disjoint lines}\\
\midrule
 \bf Picard rank: & 3 & & \bf Fano Degree: & 38 & &  \bf symmetric: & yes\\
 \bf Futaki-char.: & $0$ & &  \bf Degenerations: & \multicolumn{4}{l}{255692, 430424}\\
\multicolumn{4}{c}{\bf Cox ring:  $\CC[T_1,\ldots,T_7]/\langle T_1T_2+T_3T_4+T_5T_6^2T_7\rangle$} & \multicolumn{4}{c}{$\deg(\underline{T})=
  \left(\begin{smallmatrix}
0&
      1&
      1&
      0&
      1&
      0&
      0\\
      0&
      0&
      0&
      0&
      {-1}&
      1&
      {-1}\\
      1&
      0&
      0&
      1&
      0&
      0&
      1\\
  \end{smallmatrix}\right)
$}\\
\end{tabular}
 \threefoldBE\\
 \dualBE\\

\vspace{0.3em}
\enlargethispage{1cm}
\begin{tabular}{lrclrclr}
\toprule
\bf Name :& 3.21 &\multicolumn{6}{l}{{\bf Description: }Blowup of $\PP^1\!\!\times\!\PP^2$ in a curve of degree (2,1)}\\
\midrule
 \bf Picard rank: & 3 & & \bf Fano Degree: & 38 & &  \bf symmetric: & no\\
 \bf Futaki-char.: & ${\nicefrac{7}{8} \choose \nicefrac{7}{8}}$ & &  \bf Degenerations: & \multicolumn{4}{l}{429943}\\
\multicolumn{4}{c}{\bf Cox ring:  $\CC[T_1,\ldots,T_7]/\langle T_1T_2^2+T_3T_4^2+T_5T_6\rangle$} &  \multicolumn{4}{c}{$\deg(\underline{T})= \left(\begin{smallmatrix}1&
     0&
     1&
     0&
     1&
     0&
     0\\
     0&
     1&
     0&
     1&
     1&
     1&
     0\\
     0&
     0&
     0&
     0&
     1&
     {-1}&
     1\\
     \end{smallmatrix}\right)$}\\
\end{tabular}
 \threefoldBEA\\
 \dualBEA\\

\vspace{0.3em}
\begin{tabular}{lrclrclr}
\toprule
\bf Name :& 3.22 &\multicolumn{6}{l}{{\bf Description: } Blowup of $\PP^1 \times \PP^2$ in a conic in $\{0\} \times \PP^2$}\\
\midrule
 \bf Picard rank: & 3 & & \bf Fano Degree: & 40 & &  \bf symmetric: & no\\
 \bf Futaki-char.: & ${0 \choose -\nicefrac{2}{3}}$ & &  \bf Degenerations: & \multicolumn{4}{l}{430417, 519928}\\
\multicolumn{4}{c}{\bf Cox ring:  $\CC[T_1,\ldots,T_7]/\langle T_1T_2+T_3T_4+T_5^2\rangle$} & \multicolumn{4}{c}{$\deg(\underline{T})=
  \left(\begin{smallmatrix}
1&
      1&
      0&
      {2}&
      1&
      0&
      0\\
      0&
      0&
      1&
      {-1}&
      0&
      0&
      1\\
      0&
      0&
      0&
      0&
      0&
      1&
      1\\
  \end{smallmatrix}\right)
$}\\
\end{tabular}
 \threefoldBF\\
 \dualBF\\

\enlargethispage{1cm}
\vspace{0.3em}
\begin{tabular}{lrclrclr}
\toprule
\bf Name :& 3.23 &\multicolumn{6}{l}{{\bf Description: }Blowup of $Q$ in a point and} \\
& & \multicolumn{6}{l}{the strict transform of a line passing through}\\
\midrule
 \bf Picard rank: & 3 & & \bf Fano Degree: & 42 & &  \bf symmetric: & no\\
 \bf Futaki-char.: & ${-\nicefrac{13}{12} \choose -\nicefrac{53}{24}}$ & &  \bf Degenerations: & \multicolumn{4}{l}{254876, 430425}\\
\multicolumn{4}{c}{\bf Cox ring:  $\CC[T_1,\ldots,T_7]/\langle T_1T_2+T_3T_4+T_5T_6^2\rangle$} & \multicolumn{4}{c}{$\deg(\underline{T})=
  \left(\begin{smallmatrix}
0&
      {-1}&
      0&
      {-1}&
      1&
      {-1}&
      0\\
      1&
      1&
      1&
      1&
      0&
      1&
      0\\
      {-1}&
      {-1}&
      0&
      {-2}&
      0&
      {-1}&
      1\\
  \end{smallmatrix}\right)
$}\\
\end{tabular}
 \threefoldBG\\
 \dualBG\\
\enlargethispage{1cm}
\vspace{0.3em}
\begin{tabular}{lrclrclr}
\toprule
\bf Name :& 3.24 &\multicolumn{6}{l}{{\bf Description: }Blowup of $W$ in $(0:0:1,\,*:*:0)$}\\
\midrule
 \bf Picard rank: & 3 & & \bf Fano Degree: & 42 & &  \bf symmetric: & no\\
 \bf Futaki-char.: & ${\nicefrac{2}{3} \choose -\nicefrac{4}{3}}$ & &  \bf Degenerations: & \multicolumn{4}{l}{429952, 430423}\\
\multicolumn{4}{c}{\bf Cox ring:  $\CC[T_1,\ldots,T_7]/\langle T_1T_2+T_3T_4+T_5T_6\rangle$} & \multicolumn{4}{c}{$\deg(\underline{T})=
  \left(\begin{smallmatrix}
 0&
      1&
      0&
      1&
      1&
      0&
      0\\
      1&
      0&
      0&
      1&
      0&
      1&
      {-1}\\
      0&
      0&
      1&
      {-1}&
      0&
      0&
      1\\
  \end{smallmatrix}\right)
$}\\
\end{tabular}
 \threefoldBH\\
 \dualBH\\
\enlargethispage{1cm}

\vspace{0.3em}
\begin{tabular}{lrclrclr}
\toprule
\bf Name :& 4.4 &\multicolumn{6}{l}{{\bf Description: } Blow up of $Q$ in two non-colinear point and}\\
& & \multicolumn{6}{l}{the strict transform of a conic passing through both of them}\\
\midrule
 \bf Picard rank: & 4 & & \bf Fano Degree: & 32 & &  \bf symmetric: & yes\\
 \bf Futaki-char.: & $0$ & &  \bf Degenerations: & \multicolumn{4}{l}{254352, 430357}\\
\multicolumn{4}{c}{\bf Cox ring:  $\CC[T_1,\ldots,T_8]/\langle T_1T_2^2T_3+T_4T_5+T_6^2\rangle$} & \multicolumn{4}{c}{$\deg(\underline{T})=
  \left(\begin{smallmatrix}
{-1}&
      1&
      {-1}&
      0&
      0&
      0&
      0&
      0\\
      1&
      0&
      1&
      {2}&
      0&
      1&
      {-1}&
      0\\
      0&
      0&
      0&
      {-1}&
      1&
      0&
      1&
      0\\
      0&
      0&
      0&
      0&
      0&
      0&
      1&
      1\\
  \end{smallmatrix}\right)
$}\\
\end{tabular}
 \threefoldCA\\
 \dualCA\\

\vspace{0.3em}
\enlargethispage{1cm}
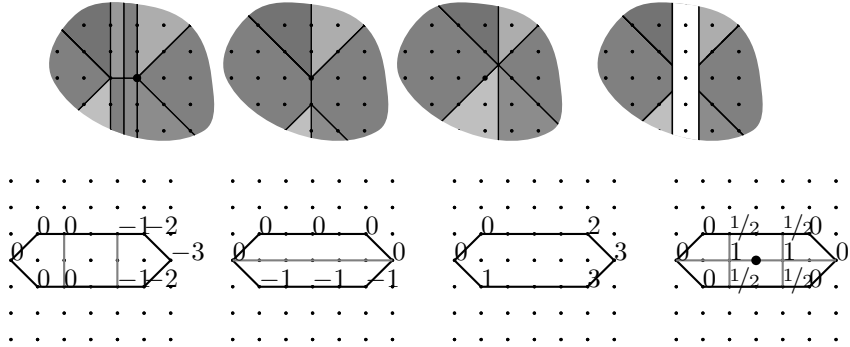
\begin{tabular}{lrclrclr}
\toprule
\bf Name :& 4.5 &\multicolumn{6}{l}{{\bf Description: }Blowup of $\PP^1\!\!\times\!\PP^2$ in a curve of degree (2,1)}\\
\midrule
 \bf Picard rank: & 3 & & \bf Fano Degree: & 32 & &  \bf symmetric: & no\\
 \bf Futaki-char.: & ${\nicefrac{5}{24} \choose \nicefrac{5}{24}}$ & &  \bf Degenerations: & \multicolumn{4}{l}{255339}\\
 \multicolumn{4}{c}{\bf Cox ring:  $\CC[T_1,\ldots,T_8]/\langle T_1T_2^2+T_3T_4^2+T_5T_6\rangle$} &  \multicolumn{4}{c}{$\deg(\underline{T})=\left(\begin{smallmatrix}1&
      0&
      1&
      0&
      1&
      0&
      0&
      0\\
      1&
      0&
      1&
      0&
      0&
      1&
      {-1}&
      0\\
      {-2}&
      1&
      {-2}&
      1&
      0&
      0&
      1&
      0\\
      \end{smallmatrix}\right)$}\\
\end{tabular}
 \threefoldCAA\\
 \dualCAA\\

\vspace{0.3em}
\enlargethispage{1cm}
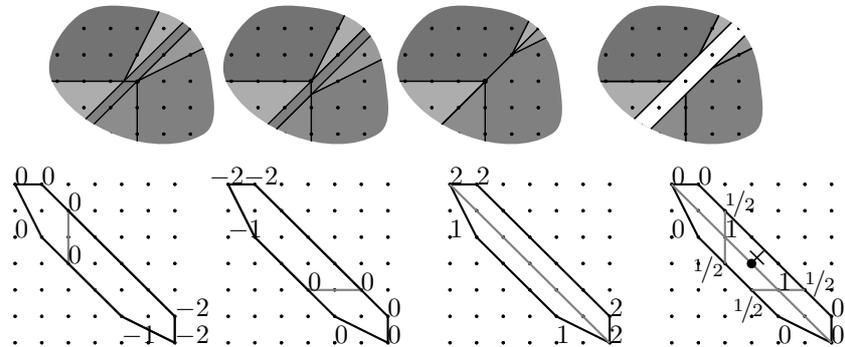
\begin{tabular}{lrclrclr}
\toprule
\bf Name :& 4.7 &\multicolumn{6}{l}{{\bf Description: }Blowup of $W$ in $(0:0:1,\,*:*:0)$}\\
& & \multicolumn{6}{l}{and $(*:*:0,\,0:0:1)$}\\
\midrule
 \bf Picard rank: & 4 & & \bf Fano Degree: & 36 & &  \bf symmetric: & yes\\
 \bf Futaki-char.: & $0$ & &  \bf Degenerations: & \multicolumn{4}{l}{254603, 255837}\\
\multicolumn{4}{c}{\bf Cox ring:  $\CC[T_1,\ldots,T_8]/\langle T_1T_2+T_3T_4+T_5T_6\rangle$} & \multicolumn{4}{c}{$\deg(\underline{T})=
  \left(\begin{smallmatrix}
0&
      1&
      1&
      0&
      1&
      0&
      0&
      0\\
      1&
      0&
      1&
      0&
      0&
      1&
      {-1}&
      0\\
      0&
      0&
      {-1}&
      1&
      0&
      0&
      1&
      0\\
      0&
      0&
      0&
      0&
      0&
      0&
      1&
      1\\
  \end{smallmatrix}\right)
$}\\
\end{tabular}
 \threefoldCB\\
 \dualCB\\

\vspace{0.3em}
\enlargethispage{1cm}
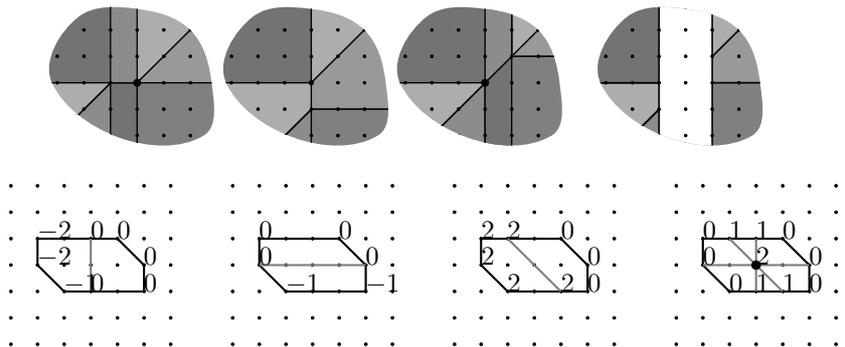
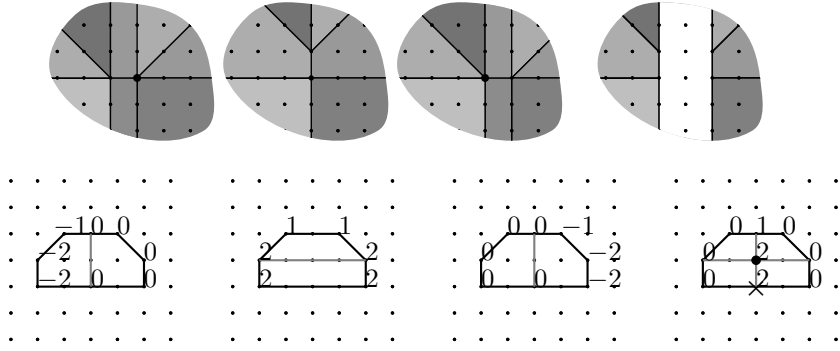
\begin{tabular}{lrclrclr}
\toprule
\bf Name :& 4.8 &\multicolumn{6}{l}{{\bf Description: }Blowup of $(\PP^1)^3$ in a curve}\\
& & \multicolumn{6}{l}{of degree $(0,1,1)$}\\
\midrule
 \bf Picard rank: & 4 & & \bf Fano Degree: & 38 & &  \bf symmetric: & no\\
 \bf Futaki-char.: & ${0 \choose -\nicefrac{13}{12}}$ & &  \bf Degenerations: & \multicolumn{4}{l}{255836, 428555}\\
\multicolumn{4}{c}{\bf Cox ring:  $\CC[T_1,\ldots,T_8]/\langle T_1T_2+T_3T_4+T_5T_6\rangle$} & \multicolumn{4}{c}{$\deg(\underline{T})=
  \left(\begin{smallmatrix}
0&
      1&
      1&
      0&
      1&
      0&
      0&
      0\\
      1&
      0&
      1&
      0&
      0&
      1&
      0&
      0\\
      0&
      0&
      {-1}&
      1&
      0&
      0&
      0&
      1\\
      0&
      0&
      0&
      0&
      0&
      0&
      1&
      1\\
  \end{smallmatrix}\right)
$}\\
\end{tabular}
 \threefoldCC\\
 \dualCC\\
\end{center}

\appendix 
\section{Calculating the Futaki character}
\label{sec:calulating-F}
 The Appendix is devoted to the, somewhat technical, proof of Theorem~\ref{thm:futaki-simplified}. For an invariant ample Cartier divisor $D$ we have the divisorial polytope $\Psi:\Box \rightarrow \wdiv_\RR \PP^1$, with $\Box \subset M_\RR$. We cannot directly use Donaldson's result for the toric case, since the dimension of the global sections of $\CO(kD)$ does not grow like the number of lattice points of a multiple of a polytope. But as it turns out it grows like the number of  ``lattice points'' of a weighted sum of polytopes. More precisely, 
we consider the abelian group $\widetilde\Pi(V)$ generated by all polytopes in some $\RR$-vector space $V$ with the relations $[\Delta] + [\nabla] = [\Delta \cup \nabla] + [\Delta \cap \nabla]$,  whenever $\Delta,\nabla,\Delta \cup \nabla$ are polytopes. A subset of $V$ which is covered by polyhedra $\nabla= \bigcup_{i \in I} \Delta_i$ can be naturally identified with an element of this group
 \begin{equation}
   \label{eq:union}
   [\nabla]:=\sum_{J \subset I} (-1)^{|J|+1}\cdot\left[\textstyle\bigcap_{j \in J}\Delta_j\right].
 \end{equation}
In particular this applies for the relative boundary $d\Delta$ of a polytope. We also introduce a boundary operator $\partial$ for the class of a polytope in $\widetilde\Pi$.
\begin{equation}
  \label{eq:boundary}
  \partial[\Delta]:=
  \begin{cases}
    [d \Delta] &, \codim(\Delta) = 0\\
    2[\Delta] - [d \Delta]&, \codim(\Delta) = 1,\\
     \;0&, \text{ else}.\\
   \end{cases}
\end{equation}

For an object $\boldsymbol\Delta=\sum_i a_i [\Delta_i] \subset \widetilde\Pi$ simply by linear continuation we have well defined notions of  dilation, boundary, volume, numbers of lattice point, and barycenter: 
 \begin{align*}
   k*\boldsymbol\Delta &:= \sum_i a_i\cdot [k\Delta_i], \\
   \partial\boldsymbol\Delta &:= \sum_i a_i \cdot \partial[\Delta_i],\\
   \vol(\boldsymbol\Delta) &:=\sum_i a_i \cdot \vol(\Delta_i),\\ 
   N(\boldsymbol\Delta) &:=\sum_i a_i \cdot N(\Delta_i),\\
   \bc(\boldsymbol\Delta) &:=\sum_i a_i \cdot \frac{\vol(\Delta_i)}{\vol(\boldsymbol\Delta)}\cdot\bc(\Delta_i).
\end{align*}
Similarly we can extend the facet volume $\rvol$ and the facet barycenter $\rbc$ to $\widetilde\Pi$, where $\rvol$ and $\rbc$ are defined to be zero for polytopes of codimension $\neq 1$. More generally, we may integrate linear forms on $V$ over elements of $\widetilde\Pi$.

\begin{proposition}
\label{sec:prop-asymptotic}
  For an element $\boldsymbol\Delta \subset \widetilde\Pi(\RR^m)$ we have the following asymptotic formula for the number of lattice points
\[N(k*\boldsymbol\Delta)=(\vol\boldsymbol\Delta) \cdot k^m + \frac{\rvol\partial\Delta}{2}\cdot k^{m-1} + \CO(k^{m-2}).\]
\end{proposition}
\begin{proof}
  This is a well known fact for polytopes of maximal dimension (see also \cite[Proposition~4.1.3]{1074.53059} for a proof). The same fact follows for weighted sums of maximal dimensional polytopes just by linearity. Our definition (\ref{eq:boundary}) of a boundary $\partial[P]$ makes it work for polytopes of lower dimension, as well.
\end{proof}

We are now going to associate to a divisorial polytope $\Psi$ an element $[\Psi]$ of $\widetilde\Pi(M_\RR \times \RR)$. We consider the natural projection $M_\RR \times \RR \rightarrow M_\RR$ and inclusion $N_\RR \hookrightarrow N_\RR \times \RR$. First, we need some notation. For a function $f:\RR^n \supset U \rightarrow \RR$ with $f \geq R$ let $\Delta^R(f)$ be the region enclosed by its graph and the $R$-level:
 \[\Delta^R(f)=\{(u,a) \in M_\RR \times \RR \mid u \in D, R \leq a \leq f(u)\}.\] 
If $f$ is non-negative we set $\Delta(f) := \Delta^0(f)$. Now, for a concave piecewise affine function $f$ on $\Box$ we denote its graph by $\widehat{f}$ and associate an element $[f]\subset \widetilde\Pi$  which correspond to the area enclosed by the graph of $f$ and the $0$-level. More precisely, we choose an integer $R \leq 0$ such that $f \geq R$ and define
\begin{equation}
  \label{eq:underthegraph}
  [f] \;=\;   [\Delta^R(f)]\; - \;[\Box\! \times\! [R,0]] 
\end{equation}

 Note, that $[f]$ does not depend on the particular choice of $R$. Now, to a divisorial polytope $\Psi:\Box \rightarrow \wdiv_\RR \PP^1$ we associate a class $[\Psi] \in \widetilde\Pi(M_\RR \times \RR)$ by
\begin{equation}
  \label{eq:virtualhstar}
  [\Psi]:=[\widetilde{\Box}] + \sum_P [\Psi_P].
\end{equation}
Here, we set $\widetilde{\Box}:=\Box \times \{0\}$. The following proposition justifies this definition.
\begin{proposition}
\label{sec:prop-virtpol-basics}
  For an ample divisor $D$ on a T-variety $X(\fan)$ of dimension $n$ and the corresponding divisorial polytope $\Psi=\Psi_D$ we have
  \begin{enumerate}
  \item $(D)^n=\vol \Psi = \vol [\Psi]$,
  \item $\dim H^0(X,\CO(D)) = N([\Psi])$. 
  \item $\dim H^0(X,\CO(D))_u = N_u([\Psi])$. 
  \end{enumerate}
Here, $N_u([\Psi])$ is the number of lattice point that project to $u \in M_\RR$ if we consider the projection $M_\RR \times \RR \rightarrow M_\RR$ for every element  in the sum (\ref{eq:virtualhstar}).
\end{proposition}
\begin{proof}
  By Theorem~\ref{sec:intersection-nr} we only have to prove $\vol \Psi = \vol [\Psi]$ to obtain (i). The only summand of (\ref{eq:virtualhstar}) that contributes to the volume is $[\Psi_P]$, since all other polytopes have lower dimension. Now, by the basic rules of integration we have $\int_\Box \Psi_P= \vol [\Psi_P]$. Hence, we obtain
\[\vol [\Psi]= \sum_P \vol [\Psi_P] = \sum_P \int_\Box \Psi_P = \int_\Box \deg \Psi = \vol \Psi.\]

  We prove (iii). Remember, that the global sections of weight $u \in \Box \cap M$ are given by $H^0(\PP^1,\CO(\Psi(u)))$. Hence, we have to check that the number $N_u$ of lattice points of $[\Psi]$, that project to $u$ equals $\dim  H^0(\PP^1,\CO(\Psi(u)))$. Looking at (\ref{eq:virtualhstar}) and (\ref{eq:underthegraph}) gives
\[N_u=1 + \sum_P (\lfloor \Psi_P(u) -R_P \rfloor + 1  - (1-R_P)) = 1 + \sum_P \lfloor \Psi_P(u) \rfloor = 1 + \deg \lfloor \Psi(u) \rfloor,\]
since the intervals $[R_P,\Psi_P(u)]$ and $[R_P,0]$ contain $\lfloor \Psi_y(u) - R_P\rfloor + 1$ and $1-R_P$ integers, respectively. We have $\dim H^0(X,\CO(\Psi(u))= 1 + \deg \lfloor \Psi(u) \rfloor$ and the claim follows. Obviously (ii) follows from (iii).
\end{proof}

\begin{theorem}\label{thm:futaki-character} The Futaki character $F_D:N_\RR \rightarrow \RR$ of an divisor $D$ with corresponding divisorial polytope $\Psi=\Psi_D$ is given by
  \begin{equation}
    \label{eq:futaki-character} 
  F_{D}(v) \;=\; \int_{\partial[\Psi]} \!\!\!\!v  \; - \; \frac{\rvol \partial [\Psi]}{\vol \Psi}\int_{[\Psi]} \!\!\!\!v .
\end{equation}

\end{theorem}
\begin{proof}
  We are using exactly the same arguments as in the proof of Theorem~4.2.1 in \cite{1074.53059}. For an element $v \in N$ we are interested in the total weight of $H^0(X,\CO(D))$ with respect to this one-parameter subgroup. By Proposition~\ref{sec:prop-virtpol-basics} it is given by 
\[w(H^0(X,\CO(D))) = \sum_{u \in \Box \cap M} \langle u, v \rangle \cdot N_u([\Psi])\]
Note, that we have $[\Psi_{kD}] = k*[\Psi_{D}]$, by \cite[Proposition 3.1]{IS10}. We are now constructing an element $\boldsymbol Q$ of  $\widetilde\Pi(M_\RR \times \RR)$ fulfilling 
\begin{equation}
  \label{eq:weights}
  N(k*\boldsymbol Q)-N(k*[\Psi])= \sum_{u \in \Box \cap M} \langle u, v \rangle \cdot N_u(k*[\Psi]).
\end{equation}
For every elementary summand $[\Delta_i]$ of (\ref{eq:virtualhstar}) (which is the class of a polytope) we are using the construction from \cite[Section 4.2]{1074.53059}, i.e. we fix an integer $R$, such that $\langle \cdot, v \rangle \leq R$ on $\Delta_i$ and define 
\[Q_i:=\{(u,t) \mid u \in \Delta_i, 0 \leq t \leq R - \langle u, v \rangle\}.\]
Now, by Proposition~\ref{sec:prop-virtpol-basics} and Proposition~\ref{sec:prop-asymptotic} we have
\[l_k = N(k*[\Psi]) = (\vol[\Psi]) \cdot k^m + \frac{\rvol\partial[\Psi]}{2}\cdot k^{m-1} + \CO(k^{m-2}).\]
Moreover, from the proof of \cite[Theorem~4.2.1]{1074.53059} we have 
\[N(kQ_i) =   k^{m+1} \int_{\Delta_i} (R-v) + \frac{k^{m}}{2} \int_{\partial \Delta_i}(R-v) + N(k\Delta_i) +  \CO(k^{m-1}).\]
By linearity and (\ref{eq:weights}) we obtain
\[w_k= k^{m+1} \int_{[\Psi]} (R-v) + \frac{k^{m}}{2} \int_{\partial [\Psi]}(R-v) +  \CO(k^{m-1})\]
Now, plugging in the coefficient in Definition~\ref{sec:def-futaki} gives the desired result.
\end{proof}

\begin{corollary}
\label{cor:futaki-barycenter}
    \[
    F_D = (\rvol \partial[\Psi])\cdot (\overline{\rbc}(\partial[\Psi])-\overline{\bc}([\Psi])),
    \]
where $\overline{\bc}$ and $\overline{\rbc}$, respectively denotes the projection of the barycenters in $M_\RR \times \RR$ to $M_\RR$.
\end{corollary}
\begin{proof}
  If we plug into the right hand side of (\ref{eq:futaki-character}) the elements of a lattice basis $e_1, \ldots, e_{n}$ of $N \times \ZZ$, by the definition of a barycenter we obtain the coordinates of
\[L=(\rvol \partial[\Psi])(\rbc(\partial[\Psi])-\bc([\Psi]))\]
with respect to the dual basis $e_1^*, \ldots, e_{n}^*$. By definition $L$ is an element of $(N_\RR \times \RR)^*=M_\RR \times \RR$ and $F_D \in N_\RR^*$ is just the restriction $L|_{N_\RR}$, i.e. the projection of $L$ to $M_\RR$.
\end{proof}


To simplify our notation we set $\Delta:=\Delta(\deg\Psi)$ and denote by
$\Delta^\partial$ the part of $\Delta \to \Box$ that projects to $\partial \Box$.

\begin{lemma}
\label{lem:barycenter-deg}
  For the volume and barycenter of $[\Psi]$ we get 
  \begin{align*}
    \overline{\bc}([\Psi])&=\overline{\bc}(\Delta)\\
    \vol([\Psi])&=\vol(\Delta).
  \end{align*}
\end{lemma}
\begin{proof}
In the proof of Proposition~\ref{sec:prop-virtpol-basics} we have seen already that $\vol([\Psi]) = \int_\Box \deg \Psi$ and hence $\vol([\Psi]) = \vol(\Delta)$. For the projected barycenter we use the same argument. The only summands that contribute to $\overline{\bc}([\Psi])$ are $[\Psi_P]$. We may interpret $\overline{\bc}([\Psi_P])$ as a linear form on $v \in N_\RR$ and by definition of the barycenter and basics of integration we have
\begin{align*}
 \vol([\Psi_P]) \cdot \overline{\bc}([\Psi_P])(v)&=\int_{\Delta^{R_P}(\Psi_P)}v - \int_{\Box \times [R_P,0]} v\\
&=\int_{\Box} \left(\Psi_P - R_P \right) \cdot v + R_P\cdot \int_\Box v\\
&=\int_{\Box} \Psi_P \cdot v.
\end{align*}
Hence, we obtain
\begin{align*}
\vol([\Psi])\cdot \overline{\bc} [\Psi] & = \sum_P \vol([\Psi_P])\cdot\overline{\bc}[\Psi_P] \\
&= \sum_P \int_\Box \Psi_P \cdot v \\
&= \int_\Box \deg \Psi \cdot v \\
&= \vol(\Delta)\overline{\bc}(\deg \Psi)\\
&= \vol([\Psi])\overline{\bc}(\deg \Psi).
\end{align*}
\end{proof}

\begin{lemma}
\label{lem:boundary}
If we assume that $\Psi_P$ is constant for $P \notin \{P_1, \ldots, P_r\}$, then we can calculate volume and barycenter of $\partial [\Psi]$ as follows.
\begin{align*}
  \overline{\rbc}(\partial [\Psi]) &\;=\; \overline{\rbc}\left([\walls{\Psi}] \; + \; (2-r)[\widetilde{\Box}] \;+ \;\sum_{i=1}^r [\widehat{\Psi}_{P_i}]\right),\\
  \rvol(\partial [\Psi]) &\;=\; \rvol\left([\walls{\Psi}] \; + \; (2-r)[\widetilde{\Box}] \;+ \;\sum_{i=1}^r [\widehat{\Psi}_{P_i}]\right),\\
\end{align*}
where $\widehat{\Psi}_{P_i}$ denotes the graph of $\Psi_{P_i}$.
\end{lemma}
\begin{proof}
When calculating $\rvol$ and $\rbc$ we have to consider only summands of $\partial[\Psi]$ of codimension one. For every $P$ we have 
\[\partial [\Psi_P] = [\widehat{\Psi}_P] - [\Box \times R_P] + [\Delta(\Psi_P|_{\partial \Box})] - [\partial\Box \times [R_P,0]] + \ldots,\]
where $\ldots$ consists of lower dimensional summands. Now, we have $\overline{\rbc}[\Box \times R_P]=\overline{\rbc}[\Box]$ and $\rvol[\Box \times R_P]=\rvol[\Box]$. Applying  Lemma~\ref{lem:barycenter-deg} to every facet of $\Box$ gives
\begin{align*}
  \overline{\rbc}\left(\sum_P [\Delta^{R_P}(\Psi_P|_{\partial \Box})] - [\partial\Box \times [R_P,0]]\right) =  \overline{\rbc}[\walls{\Psi}],\\
 \rvol\left(\sum_P [\Delta^{R_P}(\Psi_P|_{\partial \Box})] - [\partial\Box \times [R_P,0]]\right) =  \rvol[\walls{\Psi}].
\end{align*}
Hence, 
\begin{align*}
\overline{\rbc}\partial\sum_{i=1}^r [\Psi_{P_i}] = \overline{\rbc}\left([\walls{\Psi}]-r\cdot [\Box] +\sum_{i=1}^r [\widehat{\Psi}_{P_i}]\right)\\
\rvol\partial\sum_{i=1}^r [\Psi_{P_i}] = \rvol\left([\walls{\Psi}]-r\cdot [\Box] +\sum_{i=1}^r [\widehat{\Psi}_{P_i}]\right)
\end{align*}
Looking at (\ref{eq:virtualhstar}) and taking into account our definition of the boundary for lower dimensional polytopes we see, that we still have to add $\overline{\rbc}(\partial \widetilde{\Box})=2\overline{\rbc}(\widetilde{\Box})$ or $\rvol(\partial \widetilde{\Box})=2\rvol(\widetilde{\Box})$, respectively which gives the desired result.
\end{proof}

For a polytope $\nabla \subset M_\RR \times \RR$ we define $\nabla^+$ to be the part of $\nabla$ lying above the $0$-level with respect to the last component. Similarly we define $\nabla^-$ to be the part below the $0$-level. We define the \emph{pyramid} $\pyr(\nabla)$  as the class $\pyr(\nabla):= [\conv(\nabla^+,0)]-[\conv(\nabla^-,0)]$ in $\widetilde\Pi$. The pyramid operator extends by linearity to $\widetilde\Pi$.

Now we have the following equivalent to Lemma~\ref{lem:boundary} for the full-dimensional volume and barycenter.
\begin{lemma}
\label{lem:full-pyramid}
If we assume that $\Psi_P \equiv 0$  for $P \notin \{P_1, \ldots, P_\ell\}$, then we can calculate volume and barycenter of $[\Psi]$ as follows.
\begin{align*}
  \overline{\bc} [\Psi] &\;=\; \overline{\bc}\left(\pyr\big([\walls{\Psi}]\;+ \;\sum_{i=1}^\ell [\widehat{\Psi}_{P_i}]\big)\right),\\
  \vol [\Psi] &\;=\; \vol\left(\pyr\big([\walls{\Psi}] \;+ \;\sum_{i=1}^\ell [\widehat{\Psi}_{P_i}]\big)\right),\\
\end{align*}
\end{lemma}
\begin{proof}
  For a polytope $\nabla \subset M_\RR \times \RR$ and a point $u \in M_\RR$ we denote by $\nabla_u$ the length of its fiber over $u \in M_\RR$. Again we may extend this notion linearly to $\widetilde \Pi$. For proving the claim, by Lemma~\ref{lem:barycenter-deg} it's sufficient to show that 
\[\deg \Psi(u) = \pyr(\walls{\Psi})_u + \sum_{i=1}^r \pyr(\widehat{\Psi}_{P_i})_u.\]
For a point $u=w\in \partial \Box$ this equality holds, since we have $\pyr(\widehat{\Psi}_{P_i})_w=0$ and $\pyr(\walls{\Psi})_w=\deg \Psi(w)$ by definition. Moreover, for the point $u=0$ the equality is fulfilled, as well, since, $\pyr(\walls{\Psi})_0=0$ and $\pyr(\widehat{\Psi}_{P_i})_0=\Psi_{P_i}$ holds. Now, the equality follows by linearity for all points on the line segment connecting $w$ and $0$, hence, for all points in $\Box$.
\end{proof}

\begin{proof}[Proof of Theorem~\ref{thm:futaki-simplified}]
  Consider a Fano T-variety given by an f-divisor $\fan$ with nontrivial slices $\fan_{P_1}, \ldots,  \fan_{P_r}$. First we choose a special representation for the canonical divisor $K_{\PP^1}=\sum_{j=1}^{r-2} P_{r+j} - \sum_{i=1}^{r} P_i$. Where $P_{1}, \ldots, P_{2r-2}$ are distinct points on $\PP^1$. Now, by Remark~\ref{rem:lattice-distance} the corresponding divisorial polytope fulfills
  \begin{enumerate}
  \item $\Psi_P \equiv 0$ for $P \neq P_i$, $i=1,\ldots 2r-2$,
  \item $\Psi_P \equiv -1$ for $P = P_i$, $i=r+1,\ldots 2r-2$,
  \item The facets of the graph of $\Psi_P$ have lattice distance $1$ to the origin for $P \neq P_i, i=1,\ldots 2r-2$,
  \item the facets $F$ of $\Box$ with $\deg \Psi|_F \not \equiv 0$ have lattice distance $1$ from the origin.
  \end{enumerate}

By Corollary~\ref{cor:futaki-barycenter} we have
  \[F(X) = \rvol([\Psi])(\overline{\rbc}\partial[\Psi] - \overline{\bc}[\Psi])\]

We set \[C:=[\walls{\Psi}] + (2-r) [\Box \times \{-1\}] + \sum_{i=1}^\ell [\widehat{\Psi}_{P_i}].\] Now, by Lemma~\ref{lem:boundary} we have $\rvol([\Psi]) = \rvol(C)$ and $\overline{\rbc}([\Psi])=\overline{\rbc}(C)$. Moreover, by Lemma~\ref{lem:full-pyramid} we obtain $\vol([\Psi]) = \vol(\pyr(C))$ and  $\overline{\bc}([\Psi])=\overline{\bc}(\pyr(C))$. Since, $C$ is a sum over polytopes in lattice distance $1$ from the origin we may apply Lemma~\ref{lem:distance-1} and obtain
\begin{align*}
  F(X) &=\rvol(\partial[\Psi])(\overline{\rbc}(\partial[\Psi])-\overline{\bc}[\Psi]) \\
      &=\rvol(C)\cdot \big(\overline{\rbc}(C) - \overline{\bc}(\pyr(C))\big)\\
      &=\vol (\pyr(C)) \cdot \overline{\bc} (\pyr(C))\\
      &=\vol [\Psi] \cdot \overline{\bc}([\Psi])\\
      &=\vol \Delta \cdot \overline{\bc}(\Delta).
\end{align*}
\end{proof}

\begin{lemma}
\label{lem:distance-1}
  Assume that $\Delta_1,\ldots, \Delta_r \subset M_\RR \times \RR$ are polytopes of codimension 1 with lattice distance $1$ from the origin. Then for $C=\sum a_i \cdot [\Delta_i]$ we have
\[\rvol(C)\cdot (\overline{\rbc}(C) - \overline{\bc}(\pyr(C))) = \vol(\pyr(C)) \cdot \overline{\bc} (\pyr(C)).\]
\end{lemma}
\begin{proof}
  We set $n=\dim M_\RR + 1$. Now, by elementary geometry we have
\[\rvol \Delta_i = n \cdot \vol (\pyr(\Delta_i)), \qquad \overline{\rbc}(\Delta_i)=\frac{n}{n-1}\cdot \overline{\bc}(\pyr(\Delta_i)).\] It follows that 
$\rvol C = n \cdot \vol \pyr(C)$. Hence, we get 
\begin{align*}
  \overline{\rbc}(C)=
\sum a_i \frac{\rvol(\Delta_i)}{\rvol(C)} \overline{\rbc}(\Delta_i)& =
\sum a_i \frac{\vol(\pyr(\Delta_i))}{\vol(\pyr(C))}\cdot \frac{n}{n-1}\cdot\overline{\bc}(\pyr(\Delta_i))\\ & = \frac{n}{n-1}\cdot \overline{\bc}(\pyr(C)).
\end{align*}
Now, plugging in $\frac{n}{n-1}\cdot \overline{\bc}(\pyr(C))$ for $\overline{\rbc}(C)$ into $\rvol(C)\cdot(\overline{\rbc}(C) - \overline{\bc}(\pyr(C)))$
gives the desired result.
\end{proof}
\bibliographystyle{halpha}
\bibliography{../bib/all}

\newcommand{\etalchar}[1]{$^{#1}$}
\begin{thebibliography}{KKMSD73}

\bibitem[AH06]{pre05013675}
Klaus Altmann and J{\"u}rgen Hausen.
\newblock {Polyhedral divisors and algebraic torus actions.}
\newblock {\em Math. Ann.}, 334(3):557--607, 2006.

\bibitem[AHHL14]{arzhantsev2012automorphism}
I.~Arzhantsev, J.~Hausen, E.~Herppich, and A.~Liendo.
\newblock The automorphism group of a variety with torus action of complexity
  one.
\newblock {\em Moscow Mathematical Journal}, 14:429--471, 2014.

\bibitem[AHS08]{divfans}
Klaus Altmann, J{\"u}rgen Hausen, and Hendrik S{\"u}{\ss}.
\newblock Gluing affine torus actions via divisorial fans.
\newblock {\em Transformation Groups}, 13(2):215--242, 2008.

\bibitem[AIP{\etalchar{+}}12]{tvars}
Klaus {Altmann}, Nathan~Owen {Ilten}, Lars {Petersen}, Hendrik {S{\"u}{\ss}},
  and Robert {Vollmert}.
\newblock {The geometry of $T$-varieties.}
\newblock In Piotr Pragacz, editor, {\em {Contributions to algebraic geometry.
  Impanga lecture notes.}}, pages 17--69. Z\"urich: European Mathematical
  Society (EMS), 2012.

\bibitem[BDK{\etalchar{+}}18]{grdb}
G.~Brown, S.~Davis, A.~Kasprzyk, M.~Kerber, O.~Sisask, and S.~Tawn.
\newblock The graded ring database webpage.
\newblock http://www.grdb.co.uk/, 2004--2018.

\bibitem[BS99]{0939.32016}
Victor~V. Batyrev and Elena Selivanova.
\newblock {Einstein-K\"ahler metrics on symmetric toric Fano manifolds.}
\newblock {\em J. Reine Angew. Math.}, 512:225--236, 1999.

\bibitem[CCGK13]{coatesFanoCI}
T.~{Coates}, A.~{Corti}, S.~{Galkin}, and A.~{Kasprzyk}.
\newblock {Quantum Periods for 3-Dimensional Fano Manifolds}.
\newblock {\em ArXiv e-prints}, March 2013, 1303.3288.

\bibitem[CPS18]{2018arXiv180909223C}
I.~{Cheltsov}, V.~{Przyjalkowski}, and C.~{Shramov}.
\newblock {Fano threefolds with infinite automorphism groups}.
\newblock {\em ArXiv e-prints}, September 2018, 1809.09223.

\bibitem[DH82]{dh-measure}
J.J. {Duistermaat} and G.J. {Heckman}.
\newblock {On the variation in the cohomology of the symplectic form of the
  reduced phase space.}
\newblock {\em {Invent. Math.}}, 69:259--268, 1982.

\bibitem[Don02]{1074.53059}
Simon~K. Donaldson.
\newblock {Scalar curvature and stability of toric varieties.}
\newblock {\em J. Differ. Geom.}, 62(2):289--349, 2002.

\bibitem[Fut83]{0506.53030}
A.~Futaki.
\newblock {An obstruction to the existence of Einstein K{\"a}hler metrics}.
\newblock {\em Invent. Math.}, 73:437--443, 1983.

\bibitem[HS10]{cox}
J{\"u}rgen Hausen and Hendrik S{\"u}{\ss}.
\newblock The {Cox} ring of an algebraic variety with torus action.
\newblock {\em Advances in Mathematics}, 225(2):977 -- 1012, 2010.

\bibitem[IS11]{IS10}
Nathan~O. Ilten and Hendrik S{\"u}{\ss}.
\newblock {Polarized {T}-varieties of complexity one.}
\newblock {\em Michigan Mathmatical Journal}, 60:561--578, 2011.

\bibitem[IV12]{1244.14044}
Nathan~Owen Ilten and Robert Vollmert.
\newblock {Deformations of rational $T$-varieties.}
\newblock {\em J. Algebr. Geom.}, 21(3):531--562, 2012.

\bibitem[Kas10]{Kas08}
Alexander~Mieczyslaw Kasprzyk.
\newblock Canonical toric fano threefolds.
\newblock {\em Canadian Journal of Mathematics}, 62(6):1293--1309, 2010.

\bibitem[KKMSD73]{0271.14017}
George Kempf, Finn Knudsen, David Mumford, and Bernard Saint-Donat.
\newblock {\em {Toroidal embeddings I}}.
\newblock Lecture Notes in Mathematics, Vol. 339. Springer-Verlag, Berlin,
  1973.

\bibitem[KT03]{kt03}
Yael Karshon and Susan Tolman.
\newblock Complete invariants for {H}amiltonian torus actions with two
  dimensional quotients.
\newblock {\em J. Symplectic Geom.}, 2(1):25--82, 2003.

\bibitem[LS13]{tsing}
Alvaro {Liendo} and Hendrik {S{\"u}{\ss}}.
\newblock {Normal singularities with torus actions}.
\newblock {\em {Tohoku Mathematical Journal}}, 65:105--130, 2013.

\bibitem[Mab87]{0661.53032}
Toshiki Mabuchi.
\newblock {Einstein-K\"ahler forms, Futaki invariants and convex geometry on
  toric Fano varieties.}
\newblock {\em Osaka J. Math.}, 24:705--737, 1987.

\bibitem[MM86]{mm86}
Shigefumi Mori and Shigeru Mukai.
\newblock {Classification of Fano 3-folds with $B_2\geq 2$. I.}
\newblock {Nagata, M. (ed.) et al., Algebraic and topological theories. Papers
  from the symposium dedicated to the memory of Dr. Takehiko Miyata held in
  Kinosaki, October 30-November 9, 1984. Tokyo: Kinokuniya Company Ltd. 496-545
  (1986).}, 1986.

\bibitem[PS11]{tidiv}
Lars Petersen and Hendrik S{\"u}{\ss}.
\newblock Torus invariant divisors.
\newblock {\em Israel Journal of Mathematics}, 182:481--505, 2011.

\bibitem[S{\"u}{\ss}13]{kesym}
Hendrik S{\"u}{\ss}.
\newblock {K}\"ahler-{E}instein metrics on symmetric {F}ano {T}-varieties.
\newblock {\em Advances in Mathematics}, 246:100 -- 113, 2013.

\bibitem[Tia87]{0599.53046}
Gang Tian.
\newblock {On K\"ahler-Einstein metrics on certain K\"ahler manifolds with
  $C_1(M)>0$.}
\newblock {\em Invent. Math.}, 89:225--246, 1987.

\bibitem[Tim08]{1151.14037}
Dmitri Timashev.
\newblock {Torus actions of complexity one.}
\newblock {Harada, Megumi (ed.) et al., Toric topology. International
  conference, Osaka, Japan, May 28--June 3, 2006. Providence, RI: American
  Mathematical Society (AMS). Contemporary Mathematics 460, 349-364 (2008).},
  2008.

\bibitem[Tit66]{zbMATH03234443}
Jacques Tits.
\newblock {Normalisateurs de tores. I: Groupes de Coxeter etendus.}
\newblock {\em {J. Algebra}}, 4:96--116, 1966.

\bibitem[WZ04]{1086.53067}
Xu-Jia Wang and Xiaohua Zhu.
\newblock {K\"ahler--Ricci solitons on toric manifolds with positive first
  Chern class.}
\newblock {\em Adv. Math.}, 188(1):87--103, 2004.

\end{thebibliography}

\end{document}